\documentclass[reqno,11pt]{amsart}

\usepackage{amssymb}
\usepackage[margin=1.05in,letterpaper]{geometry}
\usepackage{layout}
\usepackage{tikz-cd}
\usepackage{amsmath,amsthm,mathrsfs} 
\usepackage{marvosym}
\usepackage{xcolor}
\usepackage{colonequals}
\usepackage{enumitem}
\usepackage{mathdots}
\usepackage{framed}
\usepackage{url}
\usepackage{multicol}
\usepackage{fancyhdr}
\usepackage{nccmath}
\usepackage{color}
\usepackage{xparse}
\usepackage{bbm}
\usepackage[unicode,colorlinks]{hyperref}
\usepackage{microtype}
\usepackage{marginnote}
\usepackage{mathtools}

\usetikzlibrary{decorations.pathmorphing,shapes}

     \definecolor{dark-red}{rgb}{0.54,0,0}
     \definecolor{dark-green}{rgb}{0,0.54,0}
     \definecolor{dark-magenta}{rgb}{0.54,0,0.54}
     \definecolor{dark-cyan}{rgb}{0,0.54,0.54}
     \hypersetup{%
         unicode=true,
         breaklinks=true,
         pdfcreator=,
         pdfproducer=,
         pdfstartview =, 
         pdfpagelayout = ,
         pdfhighlight = /O,
         linkcolor=dark-red, 
         linkbordercolor=dark-red, 
         anchorcolor=red, 
         citecolor=dark-green, 
         citebordercolor=dark-green, 
         filecolor=dark-cyan, 
         filebordercolor=dark-cyan, 
         menucolor=dark-red, 
         menubordercolor=dark-red, 
         runcolor=dark-cyan, 
         runbordercolor=dark-cyan, 
         urlcolor=dark-cyan, 
         urlbordercolor=dark-cyan, 
}

\newcommand\CC{\protect\mathbb{C}}

\newcommand\FF{\protect\mathbb{F}}

\newcommand\QQ{\protect\mathbb{Q}}
\newcommand\RR{\protect\mathbb{R}}

\newcommand\bA{\mathbb{A}}

\newcommand\bG{\mathbb{G}}
\newcommand\bH{\mathbb{H}}

\newcommand\bQ{\mathbb{Q}}
\newcommand\bR{\mathbb{R}}

\newcommand\bV{\mathbb{V}}

\newcommand\bX{\mathbb{X}}
\newcommand\bY{\mathbb{Y}}
\newcommand\bZ{\mathbb{Z}}

\newcommand\sP{\mathscr{P}}

\newcommand\cC{\mathcal{C}}
\newcommand\cD{\mathcal{D}}

\newcommand\cF{\mathcal{F}}
\newcommand\cG{\mathcal{G}}

\newcommand\cO{\mathcal{O}}

\newcommand\cS{\mathcal{S}}

\newcommand\cW{\mathcal{W}}

\DeclareMathOperator{\SL}{SL}

\DeclareMathOperator\GL{GL}
\DeclareMathOperator\Hom{Hom}
\DeclareMathOperator\Tr{Tr}

\DeclareMathOperator\Gal{Gal}

\DeclareMathOperator\Res{Res}
\DeclareMathOperator\tr{tr}

\DeclareMathOperator\Ind{Ind}

\DeclareMathOperator\im{Im}
\DeclareMathOperator\SO{SO}
\DeclareMathOperator\Span{span}

\DeclareMathOperator{\PGL}{PGL}

\DeclareMathOperator{\diag}{diag}

\DeclareMathOperator{\Mp}{Mp}

\DeclareMathOperator\ord{ord}

\newcommand\trans{\text{t}}

\newcommand\op{\text{op}}

\DeclareMathOperator\BC{BC}
\DeclareMathOperator\JL{JL}
\newcommand\bi{\mathbf{i}}
\newcommand\bj{\mathbf{j}}
\DeclareMathOperator\GU{GU}
\let\G\relax
\DeclareMathOperator\G{G}
\let\U\relax
\DeclareMathOperator\U{U}
\let\O\relax
\DeclareMathOperator\O{O}
\DeclareMathOperator\GSp{GSp}

\newcommand\llangle{\langle\!\langle}
\newcommand\rrangle{\rangle\!\rangle}
\DeclareMathOperator\M{M}
\newcommand\bn{\mathbf{n}}
\newcommand\W{\mathbf{W}}
\newcommand\V{\mathbf{V}}
\newcommand\bg{\mathbf{g}}
\newcommand\bchi{\mathbf{\chi}}
\newcommand\blambda{\mathbf{\lambda}}
\newcommand\bs{\mathbf{s}}
\newcommand\bw{\mathbf{w}}
\DeclareMathOperator\GO{GO}
\DeclareMathOperator\vol{vol}
\DeclareMathOperator\Val{Val}
\newcommand\e{\mathbf{e}}
\newcommand\triv{\mathrm{triv}}
\newcommand\Tam{\mathrm{Tam}}

\DeclareMathOperator\PSL{PSL}

\DeclareMathOperator\real{Re}

\newcommand{\from}{\colon}

\theoremstyle{theorem} \newtheorem{proposition}{Proposition}[section]
\theoremstyle{definition} \newtheorem{definition}[proposition]{Definition}
\theoremstyle{theorem} \newtheorem{lemma}[proposition]{Lemma}
\theoremstyle{remark} 
\theoremstyle{remark} 
\theoremstyle{remark} 
\theoremstyle{definition} \newtheorem{example}[proposition]{Example}
\theoremstyle{definition} 
\theoremstyle{theorem} 
\theoremstyle{theorem} 
\theoremstyle{theorem} \newtheorem{theorem}[proposition]{Theorem}
\theoremstyle{theorem} \newtheorem{corollary}[proposition]{Corollary}
\theoremstyle{definition} 
\theoremstyle{theorem} 
\theoremstyle{remark} 
\theoremstyle{definition} 
\theoremstyle{definition} 
\theoremstyle{definition} 
\theoremstyle{definition} 
\theoremstyle{remark} \newtheorem*{claim*}{Claim}
\theoremstyle{remark} 
\theoremstyle{theorem} 
\theoremstyle{theorem} 
\theoremstyle{definition} 
\theoremstyle{definition} 
\theoremstyle{theorem} 
\theoremstyle{remark} 
\theoremstyle{definition} 
\theoremstyle{remark}
\theoremstyle{theorem} \newtheorem*{main}{Main Theorem}

\newcommand\ur{\rm{ur}}

\DeclareMathOperator{\Sp}{Sp}

\DeclareMathOperator\pr{pr}

\DeclareMathOperator\Nm{Nm}

\DeclareMathOperator\val{val}
\newcommand\ad{{\mathrm{ad}}}

\newcommand\quotient{\backslash}

\newcommand\ONE{\mathbbm{1}}

 {\endMakeFramed\noindent\ignorespacesafterend}
 
 \numberwithin{equation}{section}
 
\title{Periods identities of CM forms on quaternion algebras}
\author{Charlotte Chan}
\email{charchan@umich.edu}
\address{Department of Mathematics \\
University of Michigan \\
530 Church St. \\
Ann Arbor 48109 USA}

\begin{document}

\maketitle

\begin{abstract}
Waldspurger's formula gives an identity between the norm of a torus period and an $L$-function of the twist of an automorphic representation on $\GL(2)$. For any two Hecke characters of a fixed quadratic extension $E$, one can consider the two torus periods coming from integrating one character against the automorphic induction of the other. Because the corresponding $L$-functions agree, (the norms of) these periods---which occur on different quaternion algebras---are closely related. In this paper, we give a direct proof of an explicit identity between the torus periods themselves.
\end{abstract}

\tableofcontents

\section{Introduction}\label{ch:introduction}

Waldspurger's work in 1985 sparked the beginnings of a rich theory studying the relationship between special values of $L$-functions and automorphic periods. In \cite{W85a}, he studies torus periods for representations of $B_\bA^\times$, where $B$ is a quaternion algebra over a number field $F$. Consider
\begin{equation*}
\sP(\pi^B, \Omega) \from \pi^B \to \CC, \qquad f^B \mapsto \int_{T_\bQ \backslash T_\bA} f^B(g) \cdot \Omega(g) \, dg,
\end{equation*}
where $\pi^B$ is the Jacquet--Langlands transfer of an irreducible automorphic representation $\pi$ of $\GL_2(\bA_\bQ)$ and $\Omega$ is a character of a maximal torus $T$. Waldspurger establishes a formula 
\begin{equation}\label{e:vague waldspurger}
|\sP(f^B, \Omega)|^2 = * \cdot L(\BC(\pi) \otimes \Omega, \tfrac{1}{2}),
\end{equation}
where $*$ consists of factors that depend only on local data. Combining Waldspurger's formula with Tunnell--Saito's work on $\epsilon$-dichotomy, which characterizes the branching behavior of representations of local quaternion algebras in terms of local $\epsilon$-factors, one sees that there is \textit{at most one} quaternion algebra $B$ such that $*$ is nonzero. If $L(\BC(\pi) \otimes \Omega, \tfrac{1}{2}) \neq 0$ and the central character condition 
\begin{equation*}
\omega_\pi \cdot \Omega|_{\bA_\bQ^\times} = 1, \qquad \text{where $\omega_\pi$ is the central character of $\pi$,}
\end{equation*}
holds, then there is a \textit{unique} quaternion algebra $B$---characterized by local $\epsilon$-factors---such that the linear functional $\sP(\pi^B, \Omega)$ is nonzero. 

In this paper, we will consider the torus periods arising from two symmetric special cases of this: fixing two Hecke characters $\chi_1, \chi_2$ of $E^\times$, consider
\begin{enumerate}[label=(\arabic*)]
\item
$\pi = \pi_{\chi_1}$ and $\Omega = \chi_2$

\item
$\pi = \pi_{\chi_2}$ and $\Omega = \chi_1$
\end{enumerate}
As such, the only automorphic representations of $\GL_2$ we will consider are those that arise as the automorphic induction $\pi_\chi$ of a Hecke character $\chi$. As the central character of $\pi_\chi$ is $\chi|_{\bA_F^\times} \cdot \epsilon_{E/F}$, the analogue of the central character condition for both (1) and (2) is:
\begin{equation}\label{e:ccc intro}
\chi_1|_{\bA_F^\times} \cdot \chi_2|_{\bA_F^\times} \cdot \epsilon_{E/F} = 1.
\end{equation}

Formally, the Rankin--Selberg $L$-function for the $(\GL_2 \times \GL_2)$-representation $\pi_{\chi_1} \otimes \pi_{\chi_2}$ satisfies
\begin{equation}\label{e:L symmetry}
L(\BC(\pi_{\chi_1}) \otimes \chi_2, s) = L(\pi_{\chi_1} \otimes \pi_{\chi_2}, s) = L(\BC(\pi_{\chi_2}) \otimes \chi_1, s).
\end{equation}
On the other hand, as we see in Equation \eqref{e:vague waldspurger}, Waldspurger's formula relates (1) to the left-hand side of \eqref{e:L symmetry} and (2) to the right-hand side of \eqref{e:L symmetry}. Furthermore, the quaternion algebras $B_1$ and $B_2$ arising from (1) and (2) are related by the following simple formula:
\begin{equation}\label{e:B1 B2}
\text{The ramification of $B_1$ and $B_2$ at a place $v$ agree if and only if $-1 \in \Nm(E_v^\times/F_v^\times)$.}
\end{equation}
Therefore one obtains a relationship between (the norms of) the torus periods arising from our two symmetric cases. 

As these torus periods occur on \textit{different} quaternion algebras, it is of interest to study these periods directly, without invoking Waldspurger. In this paper, we do exactly this: we prove an explicit identity between the periods on $B_1$ and $B_2$. We will employ the theta correspondence to construct automorphic forms and compare the resulting torus periods. To this end, the key to our approach is the construction of a \textit{seesaw of dual reductive pairs} that precisely realizes the quaternion algebras $B_1$ and $B_2$. 

\begin{main}[\ref{t:identity of periods}]
There exist explicitly constructed pairs of automorphic forms $f_1^{B_1} \in \JL^{B_1^\times}(\pi_{\chi_1})$ and $f_2^{B_2} \in \JL^{B_2^\times}(\pi_{\chi_2})$ such that
\begin{equation*}
\sP(f_1^{B_1}, \chi_2) = \sP(f_2^{B_2}, \chi_1).
\end{equation*}
\end{main}

We point out the simplest interesting case of the Main Theorem. Let $F = \bQ$ and $E = \bQ(\sqrt{-7})$, and consider the canonical Hecke character $\chi_{\rm can}$ of $E$ in the sense of Rohrlich \cite{Ro80}. Since $\chi_{\rm can}$ restricts to the quadratic character, $\chi_1 = \chi_{\rm can}^n$ and $\chi_2 = \chi_{\rm can}^m$ satisfy \eqref{e:ccc intro} so long as $n$ and $m$ have opposite parity. When $n=2$ and $m=3+2l \geq 3$, $B_1$ is the split quaternion algebra $M_2(\bQ)$ and $B_2$ is the \textit{definite} quaternion algebra $B$ ramified at exactly $7$ and $\infty$. The newform $f$ in the automorphic induction $\pi_{\chi_{\rm can}^2}$ has weight $3$ and level $\Gamma_1(7)$ with nebentypus $\varepsilon_{\bQ(\sqrt{-7})/\bQ}$, and $\delta_3^{l} f$ is a test vector for the torus period against $\chi_{\rm can}^{3 + 2l}$, where $\delta_3^{l}$ is the $l$th iterate of the Shimura--Maass differential operator. The Main Theorem gives an explicit automorphic form $f_l^B$ in the Jacquet--Langlands transfer of $\pi_{\chi_{\rm can}^{3+2l}}$ to a \textit{definite} quaternion algebra such that
\begin{equation}\label{e:ex}
0 \neq \int_{[E^\times]} (\delta_3^{l} f)(g) \cdot \chi_{\rm can}^{3+2l}(g) \, dg = \int_{[E^\times]} f_l^B(g) \cdot \chi_{\rm can}^2(g) \, dg.
\end{equation}
As $l$ changes, the $\delta_3^{l} f$ live in the same representation, but on the definite side, the representation space containing $f_l^B$ also varies. This set-up is now primed for arithmetic application: after dividing by a canonical period and taking $p$-adic limits in $l$, the left-hand side of \eqref{e:ex} is related to logarithms of generalized Heegner cycles via Bertolini--Darmon--Prasanna \cite{BDP13}. Although we do not consider arithmetic consequences of the Main Theorem here, we plan to explore this in future work.

\subsection{Outline}

We begin by establishing notation and background in Sections \ref{s:definition} and \ref{ch:weil}. In Section \ref{ch:two quaternion algebras}, we give a simple description of the relationship between $B_1$ and $B_2$. We then construct dual reductive pairs $(\U_B(V), \U_B(W^*))$ and $(\U_E(\Res V), \U_E(W))$ that both capture the behavior of $E^\times \subset B_1^\times, B_2^\times$ and also compatibly map into the same symplectic group. The goal of this paper is then to study the following seesaw of similitude unitary groups with respect to the theta correspondence:
\begin{equation*}
\begin{tikzcd}
\GU_E(\Res V) \ar[dash]{rd} \ar[dash]{d} & \GU_B(W^*)  \ar[dash]{d} \ar[dash]{ld} \\
\GU_B(V) & \GU_E(W)
\end{tikzcd} \text{`` $=$ ''}
\begin{tikzcd}
B_2^\times \ar[dash]{rd} \ar[dash]{d} & B_1^\times \ar[dash]{d} \ar[dash]{ld} \\
E^\times & E^\times
\end{tikzcd}
\end{equation*}

In Section \ref{ch:splittings}, we use Kudla's splittings for unitary groups and explicitly study their compatibility on $E^\times \times E^\times$. Many of the calculations are similar to the calculations in \cite{IP16b}.  From the compatibility statements about the splittings, we can deduce precise information about how the Weil representations on $\GU_B(V) \times \GU_B(W^*)$ and $\GU_E(\Res V) \times \GU_E(W)$ are related.

In Section \ref{ch:global theta}, we give a representation theoretic description of the global theta lifts. This requires a careful study of Kudla's splittings at the places $v$ where everything is unramified (Section \ref{s:split splittings}). We prove (Theorem \ref{t:theta aut ind}) that the global theta lifts can be described in terms of automorphic induction and Jacquet--Langlands and that the global theta lift vanishes if and only if the Jacquet--Langlands transfer does not exist. Combining these results with the compatibility results of Section \ref{ch:splittings}, we obtain our Main Theorem (Theorem \ref{t:identity of periods}).

In Sections \ref{ch:schwartz} and \ref{ch:explicit rallis}, in the case $E/F$ is CM, we construct a Schwartz function $\varphi$ whose theta lift $\theta_{\varphi}(\chi)$ to $\GL_2(F)$ is the newform. We prove an explicit Rallis inner product formula relating $\theta_\varphi(\chi)$ to $L(1, \widetilde \chi)$, which in particular shows that the theta lift is nonvanishing. These  Schwartz functions have been considered in various places before. At the finite places, they have appeared for example in \cite[Proposition 2.5.1]{P06}, \cite[N1]{X07}. At the infinite places, our choice is constructed from a confluent hypergeometric function ${}_1 F_1(a,b,t)$ of the first type. 

We conclude the paper (Section \ref{ch:example}) with details on the canonical Hecke character $\chi_{\rm can}$ of $\bQ(\sqrt{-7})$, the example mentioned earlier in the introduction.

\subsection*{Acknowledgements}

I'd like to thank my advisor Kartik Prasanna for introducing me to this area of research and to thank Atsushi Ichino for many helpful conversations. A further thank you to both Kartik and Atsushi for sharing their impeccably written preprints with me at an early stage. This work was partially supported by NSF grants DMS-0943832, DMS-1160720, and DMS-1802905.

\section{Definitions} \label{s:definition}

For a number field $F$, let $\cO$ be the ring of integers of $F$ and $\cD$ the different of $F$ over $\QQ$. Let $r_1$ be the number of real embeddings of $F$ and $2r_2$ be the number of complex embeddings of $F$. For each finite place $v$ of $F$, let $\cO_v$ be the ring of integers of $F_v$, $\pi_v$ a uniformizer of $\cO_v$, and $q_v$ the cardinality of the residue field $\cO_v/\pi_v$. Let $D = D_F$ be the discriminant of $F$ and for each finite place $v$ of $F$, let $d_v$ be the non-negative integer such that $\cD \otimes_{\cO} \cO_v = \pi_v^{d_v} \cO_v$. Set $\delta_v = \pi_v^{-d_v}$. Then $|D| = \prod_{v \nmid \infty} q_v^{d_v}$. 

Throughout this paper, let $E$ be a (possibly split) quadratic extension of $F$ and let $B$ be a quaternion algebra over $F$ containing $E$. The main groups in this paper are $\bA_E^\times$, $\bA_E^1$, and $B_\bA^\times$. For shorthand, we write
\begin{equation*}
[E^\times] \colonequals \bA_F^\times E^\times \backslash \bA_E^\times, \qquad [E^1] \colonequals E^1 \backslash \bA_E^1, \qquad [B^\times] \colonequals \bA_F^\times B^\times \backslash B_\bA^\times,
\end{equation*}
where in the last definition, we view $\bA_F^\times$ as the center of $B_\bA^\times$.

\subsection{Measures}\label{s:tamagawa}

Throughout this paper, all integrations over adelic groups are performed with respect to the Tamagawa measure. We define $dx = \prod_v dx_v$ to be the measure on $\bA_F$ that is self-dual with respect to a chosen additive character $\psi$ of $F$. We now describe the Tamagawa measure explicitly in a few special cases.

\begin{example}
The \textit{standard additive character} of $F \backslash \bA_F$ is $\psi \colonequals \psi_0 \circ \Tr_{F/\bQ}$, where $\psi_0 = \otimes_v \psi_{0,v}$ is the non-trivial additive character of $\bQ \backslash \bA_\bQ$ given by
\begin{equation*}
\psi_{0,v}(x) = \begin{cases}
e^{2\pi \sqrt{-1} x} & \text{if $v = \infty$,} \\
e^{-2\pi\sqrt{-1}x} & \text{if $v \nmid \infty$.}
\end{cases}
\end{equation*}
Observe that if $v$ is a finite place of $F$, then $\psi_v$ is trivial on $\pi_v^{-d_v} \cO_{F_v}$ but nontrivial on $\pi_v^{-d_v-1} \cO_{F_v}$. The measure $dx$ on $\bA_F$ that is self-dual with respect to $\psi$ has the property that:
\begin{enumerate}[label=$\cdot$]
\item
If $v$ is finite, then $\vol(\cO_{F_v}, dx_v) = q_v^{-d_v/2}$.

\item
If $v$ is infinite, then $dx_v$ is the Lebesgue measure.
\end{enumerate}
More generally, if $\psi'$ is any additive character of $\bA_F$, then for any finite place $v$, we have $\vol(\cO_v, dx_v) = q_v^{c(\psi_v)/2}$, where $c(\psi_v)$ is the smallest integer such that $\psi_v$ is trivial on $\pi_v^{c(\psi_v)} \cO_{F_v}$.
\end{example}

\begin{example}
For any number field $k$, put
\begin{equation*}
\rho_k \colonequals \Res_{s = 1} \zeta_F(x) = \frac{2^{r_1} (2 \pi)^{r_2} h R}{|D|^{1/2} w},
\end{equation*}
where $r_1$ is the number of real places of $k$, $r_2$ is the number of complex places of $k$, $h = h_k$ is the class number of $k$, $R = R_k$ is the regulator of $k$, $D = D_k$ is the discriminant of $k$, and $w = w_k$ is the number of roots of unity in $k$. Then the Tamagawa measure of $\bA_k^\times$ is
\begin{equation*}
d^\times x^{\Tam} = \rho_k^{-1} \cdot \prod_v d^\times x_v^{\Tam},
\end{equation*}
where
\begin{equation*}
d^\times x_v^{\Tam} \colonequals \begin{cases}
(1-q_v^{-1})^{-1} dx_v/|x|_v & \text{if $v$ is finite,} \\
dx_v/|x|_v & \text{if $v$ is infinite.}
\end{cases}
\end{equation*}
Observe that if $v$ is finite, then $\vol(\cO_v^\times, d^\times x_v^{\Tam}) = q_v^{-d_v/2}$. The Tamagawa number of $\bG_m$ is $1$, i.e.\ $\vol(k^\times \backslash \bA_k^\times, d^\times x^{\Tam}) = 1$.
\end{example}

\begin{example}
The previous example explicitly describes the Tamagawa measure of $\bA_F^\times$ and $\bA_E^\times$. For each place $v$ of $F$, one has a short exact sequence
\begin{equation*}
1 \to F_v^\times \to E_v^\times \to E_v^1 \to 1,
\end{equation*}
and hence we may define a local measure $d^1 g_v^{\Tam}$ on $E_v^1$ as the quotient measure. Then the Tamagawa measure of $E_\bA^1$ is
\begin{equation*}
d^1 g^{\Tam} \colonequals \frac{\rho_F}{\rho_E} \cdot \prod_v d^1 x_v^{\Tam}.
\end{equation*}
Observe that if $v$ is a finite place of $F$, then 
\begin{equation*}
\vol(E_v^1 \cap \cO_{E_v}^\times, d^1 x_v^{\Tam}) = 
\begin{cases}
q_{F_v}^{-1/2} & \text{if $v$ ramifies in $E$,} \\
q_{F_v}^{-d_{F_v}/2} & \text{if $v$ is inert or split in $E$.}
\end{cases}
\end{equation*}
Observe that $\vol(E_v^1 \cap \cO_{E_v}^\times, d^1 x_v^{\Tam}) = 1$ for all but finitely many places $v$. If $F$ is totally real and $E/F$ is totally imaginary, then one can show (for example by calculating the measure of an annulus in $\CC$ containing the unit circle) that
\begin{equation*}
\vol(\CC^1, d^1 x_\infty^{\Tam}) = 2\pi.
\end{equation*}
\end{example}

\subsection{Conductors}\label{s:conductor}

In this section we briefly review the notion of the conductor of an admissible representation. First let $k$ be a non-Archimedean local field with ring of integers $\cO_k$ and a fixed uniformizer $\pi$. For any integer $N \in \bZ_{\geq 0}$, let
\begin{equation*}
K_0'(N) \colonequals \left\{\left(\begin{matrix} a & b \\ c & d \end{matrix}\right) \in \GL_2(\cO_k) : c \in \pi^N \cO_k\right\}.
\end{equation*}

\begin{theorem}[Casselman]
Let $\rho$ be an irreducible admissible infinite-dimensional representation of $\GL_2(k)$ with central character $\omega$. Let $c(\rho) \in \bZ_{\geq 0}$ be the smallest integer such that 
\begin{equation*}
\left\{v \in \rho : \text{$\rho(g) v = \omega(a) v$ for all $g = \left(\begin{matrix} a & b \\ c & d \end{matrix}\right) \in K_0'(c(\rho))$}\right\} \neq \{0\}. 
\end{equation*}
Then this space has dimension one.
\end{theorem}

We call $c(\rho)$ the \textit{conductor} of $\rho$. For a smooth character $\chi \from k^\times \to \CC^\times$, define its \textit{conductor} $c(\chi) \in \bZ_{\geq 0}$ to be the smallest number such that $\chi|_{U_k^{c(\chi)}} = 1$, where $U_k^0 \colonequals \cO_k^\times$ and $U_k^n = 1 + \pi^n \cO_k$ for $n > 0$. Now let $L/k$ be a (possibly split) quadratic extension of $k$. Let $\chi$ be a smooth character of $L^\times$ and let $\pi_\chi$ denote its automorphic induction to $\GL_2(k)$. It will be useful for us to have an explicit description of $c(\pi_\chi)$ in terms of $c(\chi)$ for each place $v$ of $F$. This calculation follows from facts about Artin conductors of Galois representations and the fact that conductors of admissible representations of $\GL_2(k)$ are compatible with Artin conductors of Galois representations under the local Langlands correspondence. We have
\begin{equation}\label{e:c pi chi}
c(\pi_\chi) =
\begin{cases}
c(\chi_1) + c(\chi_2) & \text{if $L = k \oplus k$ and $\chi = \chi_1 \otimes \chi_2$,} \\
c(\pi_\chi) = \val_k(4) + 2c(\chi) & \text{if $L/k$ is unramified,} \\
c(\pi_\chi) = 1 + \val_k(4) + c(\chi) & \text{if $L/k$ is ramified.}
\end{cases}
\end{equation}

\section{Weil representations}\label{ch:weil}

Let $k$ be any field. Let $\bV$ be a symplectic vector space over $k$. The Weil representation of $\Sp(\bV)$ is a representation of a cover of $\Sp(\bV)$. It arises in a very natural way, which we briefly recall. The symplectic space $\bV$ gives rise to a Heisenberg group $H(\bV)$, which is a central extension of $\bV$ by $k$. The natural action of $\Sp(\bV)$ on $\bV$ extends to an action on $H(\bV)$ fixing the center $Z(H(\bV)) = k$. Let $\bV = \bX + \bY$ be a complete polarization. By the Stone--von Neumann theorem, the irreducible representations of $H(\bV)$ with nontrivial central character are uniquely determined by their central character and can be realized on the vector space $\cS(\bX)$ of Schwartz functions. Thus by Schur's lemma, the $\Sp(\bV)$ action on $H(\bV)$ induces an automorphism $\phi_g$ of $\cS(\bX)$ that is unique up to scalars. We therefore have a group homomorphism
\begin{equation*}
[\omega_\psi] \from \Sp(\bV) \to \PGL(\cS(\bX)), \qquad g \mapsto [\phi_g],
\end{equation*}
where $[\phi_g]$ denotes the image of $\phi_g$ under the quotient map $\GL(\cS(\bX)) \to \PGL(\cS(\bX))$. This is the \textit{projective Weil representation} of $\Sp(\bV)$.

It is natural to try to understand when $[\omega_\psi]$ lifts to a genuine representation of $\Sp(\bV)$. When $k = \FF_q$, there exists a lift, but this isn't the case in general. The assignment $g \mapsto \phi_g$ satisfies
\begin{equation*}
\phi_g \phi_h = z_\bY(g,h) \phi_{gh}, \qquad \text{for $g,h \in \Sp(\bV)$}.
\end{equation*}
It is a straightforward check that $(g,h) \mapsto z_\bY(g,h)$ defines a $2$-cocycle in $H^2(\Sp(\bV), \CC^\times)$. The 2-cocycle $z_\bY$ corresponds to a central extension $\Mp(\bV)$ of $\Sp(\bV)$ and certainly the projective Weil representation of $\Sp(\bV)$ lifts to a genuine representation of $\Mp(\bV)$. But we can realize the Weil representation on $\Sp(\bV)$ itself if and only if $z_\bY$ is in fact a $2$-coboundary.

In this paper, we will be interested in the adelic Weil representation, which is comprised of Weil representations of local fields. For the rest of this section, let $k$ be a local field of characteristic zero, fix an additive character $\psi \from k \to \CC^\times$, and fix a complete polarization $\bV = \bX + \bY$. 

\subsection{Metaplectic groups over local fields}

Following \cite[Lemma 3.2]{R93}, there is an explicit unitary lift $r \from \Sp(\bV) \to \GL(\cS(\bX))$ (a map of sets) of the projective Weil representation given by
\begin{equation*}\label{e:r}
\left(r(\sigma) \varphi\right)(x) = \int_{\bY/\ker \gamma} f_\sigma(x + y) \varphi(x \alpha + y \gamma) \mu_\sigma(d\bar y)
\end{equation*}
for any $\varphi \in \cS(\bX)$ and any $\sigma = \left(\begin{smallmatrix} \alpha & \beta \\ \gamma & \delta \end{smallmatrix}\right)$, where $\mu_\sigma$ is a Haar measure on $\bY/\ker \gamma$, $\bar y$ is the coset $y + \ker \gamma \in \bY/\ker \gamma$, and $f_\sigma(x + y) = \psi(q_\sigma(x + y))$ for $q_\sigma(x + y) = \frac{1}{2} \llangle x\alpha, x \beta \rrangle + \frac{1}{2} \llangle y \gamma, y \delta \rrangle + \llangle y \gamma, x \beta \rrangle.$
%
%
Moreover, this lift is the unique lift satisfying the properties in \cite[Theorem 3.5]{R93}. We then define the $2$-cocycle $z_\bY \from \Sp(\bV) \times \Sp(\bV) \to \CC^1$ by 
\begin{equation*}
r(gh) = z_\bY(g,h)^{-1} \cdot r(g) \cdot r(h).
\end{equation*}
This represents a class in $H^2(\Sp(\bV), \CC^1)$ and therefore gives rise to a $\CC^1$-extension $\Mp(\bV)$ of $\Sp(\bV)$ which we call the \textit{metaplectic group}. Explicitly, this group is the set $\Sp(\bV) \times \CC^1$ together with the multiplication rule
\begin{equation*}
(g,x) \cdot (h,y) = (gh, xy \cdot z_\bY(g,h)).
\end{equation*}
We define the Weil representation $\omega_\psi$ on the metaplectic group $\Mp(\bV)$ to be
\begin{equation*}
\omega_\psi \from \Mp(\bV) \to \GL(\cS(\bX)), \qquad (g,z) \mapsto z \cdot r(g).
\end{equation*}
Oftentimes, it is easier to work with the following description of $\omega_\psi$:
\begin{align}
\label{e:weil explicit m}
\omega_\psi\left(\left(\begin{matrix} a & \\ & (a^\trans)^{-1} \end{matrix}\right), z\right) \varphi(x) &= z \cdot |\det a|^{1/2} \cdot \varphi(xa) \\
\label{e:weil explicit n}
\omega_\psi\left(\left(\begin{matrix} \mathbf 1_n & b \\ & \mathbf 1_n \end{matrix}\right), z \right) \varphi(x) &= z \cdot \psi\left(\frac{1}{2} x b^\trans x \right) \cdot \varphi(x) \\
\label{e:weil explicit w}
\omega_\psi\left(\left(\begin{matrix} & \mathbf 1_n \\ -\mathbf 1_n & \end{matrix}\right), z \right) \varphi(x) &= z \cdot \int_{k^n} \varphi(y) \psi(x^\trans y) \, dy
\end{align}
for $\varphi \in \cS(\bX)$, $x \in \bX \cong k^n$, $a \in \GL(\bX) \cong \GL_n(k)$, $b \in \Hom(\bX, \bY) \cong \M_n(k)$ with $b^\trans = b$, and $z \in \CC^1$. In \eqref{e:weil explicit w}, we take $dy$ to be the product of the self-dual Haar measure on $k$ with respect to $\psi$.

It will later (for example, in Section \ref{ch:schwartz}) be convenient to understand how changing the additive character $\psi$ affects the Weil representation $\omega_\psi$. One can check that the Weil representation with respect to the additive character $\psi_\nu(x) \colonequals \psi(\nu x)$ satisfies
\begin{equation}\label{e:change psi}
\omega_\psi(d(\nu)^{-1} g d(\nu), z) = \omega_{\psi_\nu}(g,z), \qquad \text{where $d(\nu) \colonequals \left(\begin{smallmatrix} 1 & 0 \\ 0 & \nu \end{smallmatrix}\right)$ for $\nu \in k$.}
\end{equation}

If for a subgroup $\iota \from G \hookrightarrow \Sp(\bV)$, the restriction of $z_\bY$ represents the trivial class in $H^2(G, \CC^1)$, then via an explicit trivialization $s$ of $z_\bY|_{G \times G}$, we can define the Weil representation $\omega_\psi$ on $G$ as
\begin{equation*}
\omega_\psi \from G \to \GL(\cS(\bX)), \qquad g \mapsto \omega_\psi(g,s(g)).
\end{equation*}

One feature that makes the Weil representation computable is the fact that the $2$-cocycle $z_\bY$ can be expressed in terms of the \textit{Weil index} of the \textit{Leray invariant}. The properties of these that we will use in Section \ref{ch:splittings} can be found in \cite{R93}, \cite[Sections 3.1.1, 3.1.2]{IP16a}.

\subsection{The doubled Weil representation}

Now consider the doubled symplectic space $\bV^\square \colonequals \bV + \bV^-$, where $\bV^-$ has the negated form. Let $\bX^\square = \bX + \bX^-$, $\bY^\square = \bY + \bY^-$, and let $\omega_\psi^\square$ denote the Weil representation on the metaplectic group $\Mp(\bV^\square)$ with respect to $\bV^\square = \bX^\square + \bY^\square.$ We will also make use of the polarization $\bV^\square = \bV^\triangle + \bV^\bigtriangledown$, where $\bV^\triangle = \{(v,v) : v \in \bV\}$ and $\bV^\bigtriangledown = \{(v,-v) : v \in \bV\}$. Identifying $\Sp(\bV^-)$ with $\Sp(\bV)^\op$, we have a natural homomorphism
\begin{equation*}
\tilde \iota \from \Mp(\bV) \times \Mp(\bV)^\op \to \Mp(\bV^\square), \qquad ((g,z),(h,w)) \mapsto (\diag(g,h^{-1}), zw^{-1}).
\end{equation*}

\subsection{Dual reductive pairs and the Howe correspondence}\label{s:howe duality}

A \textit{dual reductive pair} $(G, G')$ in $\Sp(\bV)$ is a pair of reductive subgroups of $\Sp(\bV)$ which are mutual centralizers of each other. There is a natural map
\begin{equation*}
i \from G \times G' \to \Sp(\bV), \qquad (g,g') \mapsto (v \mapsto g^{-1} v g').
\end{equation*}
If the cocycle $z_\bY$ can be trivialized on $i(G \times G') \subset \Sp(\bV)$, we can define the Weil representation on $i(G \times G')$ and pull back to a Weil representation of $G \times G'$. In \cite{K94}, Kudla wrote down explicit splittings of $z_\bY$. We will make use of this work heavily in the present paper.

The Weil representation $\omega_\psi$ on $G \times G'$ has the following multiplicity-one property. For an irreducible $G$-representation $\pi$, let $\cS(\pi)$ denote the largest quotient of $\cS(\bX)$ such that $G$ acts by $\pi$. By \cite[Chapter 2, Lemma III.4]{MVW}, there exists a unique irreducible $G'$-representation $\Theta(\pi)$ such that
\begin{equation*}
\cS(\pi) \cong \pi \otimes \Theta(\pi).
\end{equation*}
We call $\Theta(\pi)$ the \textit{local theta lift} of $\pi$. 

\section{Waldspurger, Tunnell--Saito, and a pair of quaternion algebras} \label{ch:two quaternion algebras}

For any quaternion algebra $B$ over $F$, we write $\Sigma_B \colonequals \{\text{places $v$ of $F$ such that $B_v$ is ramified}\}.$

\subsection{Waldspurger's formula}\label{ss:waldspurger}

Let $\pi$ be an irreducible automorphic representation of $\GL_2(\bA_F)$ with central character $\omega_\pi$ that has a nonzero Jacquet--Langlands transfer $\pi^B$ to $B_\bA^\times$. Recall that this means that $\pi_v$ is discrete series at all $v \in \Sigma_B$. Let $\Omega$ be any Hecke character of $E^\times$ such that $\Omega|_{\bA_F^\times} = \omega_\pi^{-1}$. Define
\begin{equation*}
\sP(\pi^B, \Omega) \from \pi^B \to \CC, \qquad f \mapsto \int_{[E^\times]} f(t) \, \Omega(t) \, dt.
\end{equation*}
We have the following classical theorem, which follows from combining Waldspurger's formula with the local $\epsilon$-dichotomy theorem of Tunnell and Saito.

\begin{theorem}[{Waldspurger \cite{W85a}, Tunnell \cite{T83}, Saito \cite{S93}}]\label{t:global branching}
Let $\pi$ be an irreducible automorphic representation of $\GL_2(\bA_F)$ with central character $\omega_\pi$. If 
\begin{equation*}
L(\BC(\pi) \otimes \Omega, \tfrac{1}{2}) \neq 0, \qquad \text{and} \qquad \Omega|_{\bA_F^\times} = \omega_\pi^{-1},
\end{equation*}
then there exists a unique quaternion algebra $B = B_{\pi, \Omega}$ over $F$ such that 
\begin{equation*}
\sP(\pi^B, \Omega) \neq 0.
\end{equation*}
Moreover, $B$ is the unique quaternion algebra with ramification set
\begin{equation*}
\Sigma_{\pi, \Omega} \colonequals \{v : \epsilon_v(\BC(\pi) \otimes \Omega) \cdot \omega_v(-1) = -1\}.
\end{equation*}
\end{theorem}

\begin{proof}
If $L(\BC(\pi) \otimes \Omega, \tfrac{1}{2}) \neq 0$, then $\epsilon(\BC(\pi) \otimes \Omega) = +1$. Since $\omega$ is a Hecke character of $\bA^\times$, we must have $\omega(-1) = +1$. Therefore, there must be an \textit{even} number of places $v$ of $F$ such that $\epsilon_v(\BC(\pi) \otimes \Omega) \cdot \omega_v(-1) = -1$, and hence there exists a unique quaternion algebra $B_{\pi, \Omega}$ over $F$ with ramification set $\Sigma_{\pi, \Omega}$, and the conclusion now follows from Waldspurger's formula and the local branching criterion of Tunnell and Saito.
\end{proof}

\subsection{A pair of quaternion algebras}\label{ss:two quat}

We now specialize to the setting where $\pi$ comes from automorphic induction. Let $\chi, \chi'$ be Hecke characters of $\bA_E^\times$. One has
\begin{equation*}
L(\BC(\pi_\chi) \otimes \chi', s) = L(\pi_\chi \otimes \pi_{\chi'}, s) = L(\BC(\pi_{\chi'}) \otimes \chi, s),
\end{equation*}
and let us assume that 
\begin{equation}\label{e:nonzero center}
L(\BC(\pi_\chi) \otimes \chi', \tfrac{1}{2}) = L(\BC(\pi_{\chi'}) \otimes \chi, \tfrac{1}{2}) \neq 0.
\end{equation} 
It is a standard calculation to see that the central character of $\pi_\chi$ (and of any Jacquet--Langlands transfer $\pi_\chi^B$) is $\chi|_{\bA_F^\times} \cdot \epsilon_{E/F}$, where $\epsilon_{E/F}$ is the quadratic character of $\bA_F^\times$ associated to the quadratic extension $E/F$. Therefore the central character condition in Theorem \ref{t:global branching} is:
\begin{equation}\label{e:ccc}
\chi|_{\bA_F^\times} \cdot \chi'|_{\bA_F^\times} \cdot \epsilon_{E/F} = 1.
\end{equation}
If $\chi, \chi'$ satisfy \eqref{e:ccc}, then by Theorem \ref{t:global branching}, $B = B_{\pi_\chi, \chi'}$ and $B' = B_{\pi_{\chi'}, \chi}$ are the unique quaternion algebras such that $\sP(\pi_\chi^B, \chi') \neq 0$ and $\sP(\pi_{\chi'}^{B'}, \chi) \neq 0.$

\begin{proposition}\label{p:two quat}
Let $\chi, \chi'$ be Hecke characters of $\bA_E^\times$ satisfying \eqref{e:nonzero center} and \eqref{e:ccc}, and let $E = F(\bi)$ with $\bi^2 = u$. If $B = B_{\pi_\chi, \chi'}$ is the quaternion algebra that corresponds to the Hilbert symbol $(u, J)$, then $B' = B_{\pi_{\chi'}, \chi}$ corresponds to the Hilbert symbol $(u, -J)$.
\end{proposition}

\begin{proof}
It is a standard computation to show that:
\begin{equation*}
\epsilon_v(\BC(\pi_\chi) \otimes \chi') = \epsilon_v(\BC(\pi_{\chi'}) \otimes \chi).
\end{equation*}
Equation \eqref{e:ccc} implies that $\omega_{\pi_\chi} \cdot \omega_{\pi_{\chi'}} \cdot \epsilon_{E/F} = 1.$ Using Theorem \ref{t:global branching}, we see that $\Sigma_{\pi_{\chi'}, \chi}$ can be described in terms of $\Sigma_{\pi_\chi, \chi'}$:
\begin{equation*}
\Sigma_{\pi_{\chi'}, \chi} =
\left\{v :
\begin{gathered}
\text{$v \in \Sigma_{\pi_\chi, \chi'}$ and $\epsilon_{E_v/F_v}(-1) = 1$, or} \\
\text{$v \notin \Sigma_{\pi_\chi, \chi'}$ and $\epsilon_{E_v/F_v}(-1) = -1$.}
\end{gathered}\right\}
\end{equation*}
An equivalent way to state this relationship is the following. The quaternion algebra $B$ can be given an $F$ basis $1, \bi, \bj, \bi\bj$ such that $E = F[\bi]$. Write $\bi^2 = u$ and $\bj^2 = J$ so that $B$ is the quaternion algebra associated to the Hilbert symbol $(u,J)$. That is,
\begin{equation*}
(u,J)_v = -1 \qquad \Longleftrightarrow \qquad v \in \Sigma_{\pi_\chi, \chi'}.
\end{equation*}
By the bimultiplicativity of the Hilbert symbol, $B'$ is the quaternion algebra associated to 
\begin{equation*}
(u,J) \cdot \epsilon_{E/F}(-1) = (u, J) \cdot (u, -1) = (u,-J). \qedhere
\end{equation*}
\end{proof}

\subsection{A seesaw of unitary groups}\label{s:unitary groups}

In this section, we introduce the main dual reductive pairs of interest in this paper. Fix $\mathbf i \in E$ with $\tr_{E/F} \mathbf i = \mathbf i + \overline {\mathbf i} = 0$. Note that $E = F[\mathbf i]$. Let $B$ be a (possibly split) quaternion algebra over $F$ and let $1, \mathbf i, \mathbf j, \mathbf k$ be a standard basis for $B$ over $F$. Viewing $B = E \oplus E \bj$, we set $\pr \from B \to E$ to be the projection onto the $E$-component. We consider the following spaces:
\begin{enumerate}[label=\textbullet]
\item
$V = B = 1$-dimensional right $B$-space with skew-Hermitian form $\langle x,y \rangle = x^* \bi y$

\item
$W^* = B \otimes_E E = 1$-dimensional left $B$-space with Hermitian form $(x,y) = x y^*$

\item
$\Res V = 2$-dimensional right $E$-space with skew-Hermitian form $\langle x,y \rangle = \pr(x^* \bi y)$

\item
$W = E = 1$-dimensional left $E$-space with Hermitian form $(a,b) = a \overline b$

\item
$V_0 = 1$-dimensional right $E$-space with Hermitian form $\langle a,b \rangle_0 = \overline a b$

\item
$W_0 = B = 2$-dimensional left $E$-space with skew-Hermitian form $(x,y)_0 = -\bi \pr(xy^*)$

\item
$\bV = V \otimes W^* = \Res V \otimes W = V_0 \otimes W_0 = 4$-dimensional $F$-space with symplectic form $\frac{1}{2} \Tr_{E/F}(\langle \cdot, \cdot \rangle \otimes \overline{( \cdot, \cdot )})$
\end{enumerate}
Then both pairs $(U_B(V), U_B(W^*))$ and $(U_E(\Res V), U_E(W))$ are irreducible dual reductive pairs (of type 1) in $\Sp(\bV)$. (See, for example, \cite{P93}.) For any pair $(V,W) = (V, W^*)$, $(\Res V, W)$, or $(V_0, W_0)$, we take as our convention 
\begin{equation*}
\GL(V) \times \GL(W) \to \GL(V \otimes W), \qquad (g,h) \mapsto (v \otimes w \mapsto g^{-1}v \otimes wh).
\end{equation*}
It is clear that $U_B(V) \subset U_E(\Res V)$ and that $U_E(W) \subset U_B(W^*)$. Furthermore, we have a commutative diagram
\begin{equation}\label{e:compat subgroups}
\begin{tikzcd}[column sep=tiny]
\U_B(V) \ar[hook]{d} & \times & \U_B(W^*) \ar{rrrr} &&&& \Sp(\Res_{B/F}(V \otimes_B W^*)) \ar{d}{=} \\
\U_E(\Res V) & \times & \U_E(W) \ar[hook]{u} \ar{rrrr} &&&& \Sp(\Res_{E/F}(\Res V \otimes_E W))
\end{tikzcd}
\end{equation} 
Therefore we have the following seesaw of dual reductive pairs
\begin{equation*}
\begin{tikzcd}
U_E(\Res V) \ar[dash]{rd} \ar[dash]{d} & U_B(W^*)  \ar[dash]{d} \ar[dash]{ld} \\
U_B(V) & U_E(W)
\end{tikzcd}
\cong
\begin{tikzcd}
(E^1 \times (B')^1)/F^1  \ar[dash]{d} \ar[dash]{dr} & B^1 \ar[dash]{d} \ar[dash]{dl} \\
E^1 \cup E^{\frac{1}{J}} \mathbf j & E^1
\end{tikzcd}
\end{equation*}
Here, $B' = \left(\frac{\mathbf i^2, -\mathbf j^2}{F}\right)$ and the superscript $r \in \bQ$ picks out the norm-$r$ elements. The analogous seesaw with similitudes is
\begin{equation}\label{e:main seesaw}
\begin{tikzcd}
GU_E(\Res V) \ar[dotted,dash]{rd} \ar[dotted,dash]{d} & GU_B(W^*)  \ar[dotted,dash]{d} \ar[dotted,dash]{ld} \\
GU_B(V) & GU_E(W)
\end{tikzcd}
\cong
\begin{tikzcd}
(E^\times \times (B')^\times)/F^\times \ar[dotted,dash]{d} \ar[dotted,dash]{dr} & B^\times \ar[dotted,dash]{d} \ar[dotted,dash]{dl} \\
E^\times \cup E^\times \mathbf j & E^\times
\end{tikzcd}
\end{equation} 
The only isomorphism that is not straightforward to see is $\GU_E(\Res V) \cong (E^\times \times (B')^\times)/F^\times.$ This comes from a natural right action of $(B')^\times$ on $\Res V = B$ defined by
\begin{equation*}
1 \cdot \bj' = \bj, \qquad \bj \cdot \bj' = -J.
\end{equation*}

We note that the point of introducing the $E$-spaces $V_0$ and $W_0$ is  that we have natural maps
\begin{equation*}
U_B(V)^0 \cong U_E(V_0), \qquad U_B(W^*) \hookrightarrow U_E(W_0).
\end{equation*}
This will allow us to compute splittings on the quaternionic unitary groups $U_B(V)$ and $U_B(W^*)$ by pulling back splittings on $U_E(V_0)$ and $U_E(W_0)$.

\section{Splittings for unitary similitude groups}\label{ch:splittings}

In this section, we define the Weil representation on the dual reductive pairs introduced in Section \ref{s:unitary groups} using the explicit splittings of $z_\bY$ defined by Kudla \cite{K94}. The properties of the Weil index and the Leray invariant we will use in this section can be found in \cite{R93}, \cite[Sections 3.1.1, 3.1.2]{IP16a}. We prove that the splittings are compatible with the seesaws constructed in Section \ref{s:unitary groups}. In Sections \ref{s:kudla}, \ref{s:change polarization}, \ref{s:three seesaws}, and \ref{s:compat}, we fix a place $v$ of $F$ and suppress $v$ from the notation. In Section \ref{s:global splittings}, we combine the local considerations from the preceding subsections into the global picture. Many of these calculations (especially in Sections \ref{s:three seesaws} and \ref{s:compat}) are similar to those in \cite[Appendix C]{IP16a}, \cite{IP16b}. 

In order to describe the global automorphic theta lift from a Hecke character to a quaternion algebra, which we will do later in Section \ref{ch:global theta}, we will need to give an explicit description of the local splittings in Section \ref{s:three seesaws} in the special case that the quaternion algebra is unramified (i.e.\ split) at the place in question. We do this in Section \ref{s:split splittings}.

\subsection{Kudla's splitting for split unitary groups}\label{s:kudla}

We first recall Kudla's splitting \cite{K94} of Rao's cocycle \cite{R93} for split unitary groups over $E$. Let $\W \cong E^{2n}$ (row vectors) be an $E$-vector space of dimension $2n$ with $\epsilon$-skew Hermitian form
\begin{equation*}
\langle (x_1, y_1), (x_2, y_2) \rangle = x_1 \overline y_2^\trans - \epsilon y_1 \overline x_2^\trans,
\end{equation*}
and let $e_1, \ldots, e_n, e_1', \ldots, e_n'$ be the $E$-basis of $\W$ giving the isomorphism $\W \cong E^{2n}$. Let $\V$ be an $E$-vector space of dimension $m$ with a non-degenerate $\epsilon$-Hermitian form $(\cdot, \cdot)$. (Here, $\overline x$ denotes the image of $x$ under the nontrivial involution of $E$ over $F$ and the superscript ${}^\trans$ denotes transposition.) Then $(\U_E(\V), \U_E(\W))$ is a dual reductive pair and there is a natural map
\begin{equation*}
\iota \from \U_E(\V) \times \U_E(\W) \to \Sp(\V \otimes_E \W), \qquad (h,g) \mapsto (w \otimes v \mapsto h^{-1}w \otimes vg).
\end{equation*}
We denote by $\iota_\W \from \U_E(\V) \to \Sp(\V \otimes_E \W)$ and $\iota_\V \from \U_E(\W) \to \Sp(\V \otimes_E \W)$ the restrictions of $\iota$ to $\U_E(\V) \times \{1\}$ and $\{1\} \times \U_E(\W)$, respectively.

For $0 \leq j \leq n$, let $\tau_j \in \U_E(\W)$ be the element defined by
\begin{align*}
e_i \tau_j = \begin{cases}
-\epsilon e_i' & \text{if $1 \leq i \leq j$,} \\
e_i & \text{if $i > j$,}
\end{cases}
\qquad and \qquad
e_i' \tau_j = \begin{cases}
e_i & \text{if $1 \leq i \leq j$,} \\
e_i' & \text{if $i > j$.}
\end{cases}
\end{align*}
Then
\begin{equation*}
\U_E(\W) = \bigsqcup_{i=0}^j P \tau_j P,
\end{equation*}
where $P = P_Y \subset \U_E(\W)$ is the parabolic subgroup stabilizing the maximal isotropic subspace $Y \colonequals \Span_E\{e_1', \ldots, e_n'\}$. If $g = p_1 \tau_j p_2 \in P \tau_j P$, then we define
\begin{equation*}
j(g) \colonequals j, \qquad \text{and} \qquad x(g) \colonequals \det(p_1p_2|_Y) \in E^\times.
\end{equation*}

For any $E$-vector space $\V_0$ endowed with a non-degenerate Hermitian form, define
\begin{equation*}
\gamma_F(\tfrac{1}{2} \psi \circ R\V_0) \colonequals (u, \det(\V_0))_F \gamma_F(-u, \tfrac{1}{2} \psi)^m \gamma_F(-1, \tfrac{1}{2} \psi)^{-m}.
\end{equation*}

\begin{definition}\label{d:kudla splitting}
Define
\begin{equation*}
\beta_{\V,\xi} \from \U_E(\W) \to \CC^1, \qquad g \mapsto\begin{cases}
\xi(x(g)) \gamma_F(\tfrac{1}{2} \psi \circ R\V)^{-j(g)} & \text{if $\epsilon = +1$,} \\
\xi(x(g)) \xi(\bi)^j \gamma_F(\tfrac{1}{2}\psi \circ R\V')^{-j(g)} & \text{if $\epsilon = -1$,}
\end{cases}
\end{equation*}
where $\V'$ is the Hermitian form obtained by scaling the skew-Hermitian form on $\V$ by $\bi$.
\end{definition}

\begin{theorem}[{Kudla, \cite[Thm 3.1]{K94}}]\label{t:kudla splittings}
Let $\xi$ be a unitary character of $E^\times$ whose restriction to $F^\times$ is $\epsilon_{E/F}^m$, where $\epsilon_{E/F}(x) = (x, u)_F$ is the quadratic character corresponding to the extension $E/F$. Then for the maximal isotropic subspace $\bY \colonequals \V \otimes_E Y$ of $\V \otimes_E \W$,
\begin{equation*}
z_{\bY}(\iota_\V(g_1), \iota_\V(g_2)) = \beta_{\V,\xi}(g_1g_2) \beta_{\V,\xi}(g_1)^{-1} \beta_{\V,\xi}(g_2)^{-1}.
\end{equation*}
In other words, with respect to the isomorphism $\Mp(\V \otimes_E \W) \cong \Sp(\V \otimes_E \W) \times \CC^1$ determined by $z_{\bY}$, the following diagram commutes:
\begin{equation*}
\begin{tikzcd}
& \Mp(\V \otimes_E \W)_{\bY} \ar{d} \\
\U_E(\W) \ar{ur}{(\iota_\V, \beta_{\V,\xi})} \ar{r}[below]{\iota_\V} & \Sp(\V \otimes_E \W)
\end{tikzcd}
\end{equation*}
\end{theorem}

\subsection{Changing polarizations}\label{s:change polarization}

\begin{lemma}[{Kudla, \cite[Lemma 4.2]{K94}}]\label{l:change polarization}
Let $\bX + \bY$ and $\bX' + \bY'$ be two polarizations of a symplectic space $\bV$.  Then
\begin{equation*}
z_{\bY'}(g_1, g_2) = \lambda(g_1g_2) \lambda(g_1)^{-1} \lambda(g_2)^{-1} \cdot z_{\bY}(g_1, g_2),
\end{equation*}
where $\lambda \from \Sp(\bV) \to \CC^1$ is given by
\begin{equation*}
\lambda(g) \colonequals \lambda_{\bY \rightsquigarrow \bY'}(g) \colonequals \gamma_F(\tfrac{1}{2} \psi \circ q(\bY, \bY' g^{-1}, \bY')) \cdot \gamma_F(\tfrac{1}{2} \psi \circ q(\bY, \bY', \bY g)).
\end{equation*}
In particular, the bijection
\begin{equation*}
\Mp(\bV)_{\bY} \to \Mp(\bV)_{\bY'}, \qquad (g,z) \mapsto (g, z \cdot \lambda(g))
\end{equation*}
is an isomorphism.
\end{lemma}

\subsection{Three seesaws of unitary groups}\label{s:three seesaws}

For any two unitary similitude groups $\GU_E(\V)$ and $\GU_E(\W)$, we write
\begin{equation*}
\G(\U_E(\V) \times \U_E(\W)) \colonequals \{(g,h) \in \GU_E(\V) \times \GU_E(\W) : \nu(g) = \nu(h)\}.
\end{equation*}
Fix a complete polarization $\bV = \bX + \bY$. In this section, we define splittings (of $z_{\bY}$ or $z_{\bY^\square}$, depending on context) for the unitary groups $\G(\U_E(V_0^\square) \times \U_E(W_0))$, $\G(\U_E(V_0) \times \U_E(W_0))$, $\G(\U_E(\Res V) \times \U_E(W^\square))$, and $\G(\U_E(\Res V) \times \U_E(W))$, which fit into the seesaw
\begin{equation}\label{e:single seesaw}
\begin{tikzcd}
\U_E(\Res V) \ar[dash]{dr} \ar[dash]{d} & \U_E(W_0) \ar[dash]{dl} \ar[dash]{d} \\
\U_E(V_0) & \U_E(W)
\end{tikzcd}
\end{equation}
and the two corresponding doubling seesaws:
\begin{equation}\label{e:double seesaw}
\begin{tikzcd}
\U_E(V_0^\square) \ar[dash]{dr} \ar[dash]{d} & \U_E(W_0) \times \U_E(W_0) \ar[dash]{dl} \ar[dash]{d} \\
\U_E(V_0) \times \U_E(V_0) & \U_E(W_0)^\triangle
\end{tikzcd} \qquad
\begin{tikzcd}
\U_E(\Res V) \ar[dash]{dr} \ar[dash]{d} & \U_E(W^\square) \ar[dash]{dl} \ar[dash]{d} \\
\U_E(\Res V)^\triangle& \U_E(W) \times \U_E(W)
\end{tikzcd}
\end{equation}

\subsubsection{Splittings for $\G(\U_E(V_0^\square) \times \U_E(W_0))$ and $\G(\U_E(V_0) \times \U_E(W_0))$} \label{s:splittings0}

Consider the $2$-dimensional $E$-space $V_0 \otimes_E W_0$ with skew-Hermitian form given by $\overline{(\cdot, \cdot)} \otimes \langle \cdot, \cdot \rangle$. By a straightforward computation, we see that this allows us to identify $V_0 \otimes_E W_0 = W_0$ as $E$-spaces endowed with skew-Hermitian forms. Define
\begin{align*}
i \from \G(\U_E(V_0) \times \U_E(W_0)) &\to \U_E((V_0 \otimes W_0)^\square), \\ 
(g,h) &\mapsto ((v \otimes w,v^- \otimes w^-) \mapsto (g^{-1}v \otimes w h, v^- \otimes w^-)), \\
i^- \from \G(\U_E(V_0) \times \U_E(W_0)) &\to \U_E((V_0 \otimes W_0)^\square), \\
(g,h) &\mapsto ((v \otimes w,v^- \otimes w^-) \mapsto (v \otimes w, g^{-1}v^- \otimes w^-h)), \\
i^\square \from \G(\U_E(V_0^\square) \times \U_E(W_0)) &\to \U_E(V_0^\square \otimes W_0), \\
(g,h) &\mapsto (v \otimes w \mapsto g^{-1}v \otimes wh).
\end{align*}
We may identify $V_0^\square \otimes W_0 = (V_0 \otimes W_0)^\square = W_0^\square$. We have natural embeddings
\begin{align*}
&\G(\U_E(V_0) \times \U_E(V_0) \times \U_E(W_0)) \hookrightarrow \G(\U_E(V_0) \times \U_E(W_0)) \times \G(\U_E(V_0) \times \U_E(W_0)) \\
&\G(\U_E(V_0) \times \U_E(V_0) \times \U_E(W_0)) \hookrightarrow \G(\U_E(V_0^\square) \times \U_E(W_0)).
\end{align*}
Observe that for $(g_1, g_2, h) \in \G(\U_E(V_0) \times \U_E(V_0) \times \U_E(W_0))$,
\begin{equation*}
i(g_1,h) i^-(g_2,h) = i^\square(g_1,g_2,h) \in \U_E(W_0^\square).
\end{equation*}
We identify $\Res_{E/F}(W_0^\square) = \bV^\square$ and let
\begin{equation*}
\iota \from \U_E(W_0^\square) \to \Sp(\Res_{E/F}(W_0^\square)) = \Sp(\bV^\square)
\end{equation*}
be the natural embedding. We will often identify $\U_E(W_0^\square)$ with $\iota(\U_E(W_0^\square))$.

\begin{definition}\label{d:splittings0}
Pick a character $\xi \from E^\times \to \CC^1$ such that $\xi|_{F^\times} = \epsilon_{E/F}$. Define
\begin{equation*}
\beta \from \U_E(W_0^\square) \to \CC^1, \qquad g \mapsto \xi(x(g)) \cdot ((u, -1)_F \gamma_F(u, \tfrac{1}{2} \psi))^{-j(g)}.
\end{equation*}
Define $\lambda \colonequals \lambda_{V_0 \otimes W_0^\triangle \rightsquigarrow \bY^\square} \from \Sp(\bV^\square) \to \CC^1$ and
\begin{align*}
\hat s &\colonequals i^* \beta, & \hat s^- &\colonequals (i^-)^* \beta, & \hat s^\square &\colonequals (i^\square)^* \beta, \\
s &\colonequals i^*(\beta \lambda), & s^- &\colonequals (i^-)^* (\beta \lambda), & s^\square &\colonequals (i^\square)^*(\beta \lambda).
\end{align*}
\end{definition}

\begin{lemma}\mbox{}
\begin{enumerate}[label=(\alph*)]
\item
$\hat s$, $\hat s^-$, and $\hat s^\square$ are splittings of $z_{V_0 \otimes W_0^\triangle}$ on the images of $i$, $i^-$, and $i^\square$, respectively.

\item
$s$ is a splitting of $z_\bY$ on the image of $i$, $s^-$ is a splitting of $z_\bY^{-1}$ on the image of $i^-$, and $s^\square$ is a splitting of $z_{\bY^\square}$ on the image of $i^\square$.
\end{enumerate}
\end{lemma}

\begin{proof}
Observe that $\det(V_0) = 1$ and $\dim(V_0) = 1$ so that
\begin{equation*}
\gamma_F(\tfrac{1}{2} \psi \circ RV_0) = (u, 1)_F \gamma_F(-u, \tfrac{1}{2} \psi) \gamma_F(-1, \tfrac{1}{2} \psi)^{-1} = (u, -1)_F \gamma_F(u, \tfrac{1}{2} \psi).
\end{equation*}
This implies that $\beta = \beta_{\U_E(V_0), \xi}$ (see Definition \ref{d:kudla splitting}) and hence is a splitting of $z_{V_0 \otimes_E W_0^\triangle}$. Since $\hat s$, $\hat s^-$, and $\hat s^\square$ are pullbacks of $\beta$, they must also be splittings of the same cocycle.
\end{proof}

\begin{lemma}
For any $(g,h) \in \G(\U_E(V_0) \times \U_E(W_0))$,
\begin{equation*}
\hat s^-(g,h) = \overline{\hat s(g,h)} \cdot \xi(\det(g,h)).
\end{equation*}
\end{lemma}

\begin{proof}
Let $d_{W_0^\triangle}(-1) = \left(\begin{smallmatrix} 1 & 0 \\ 0 & -1 \end{smallmatrix}\right)$ and set
\begin{equation*}
j_{W_0^\triangle} \from \U_E(W_0^\square) \to \U_E(W_0^\square), \qquad g \mapsto d_{W_0^\triangle}(-1) g d_{W_0^\triangle}(-1).
\end{equation*}
Let $g \in \G(\U_E(V_0) \times \U_E(W_0))$. By a straightforward computation, we have
\begin{equation*}
x(i^-(g)) = (-1)^{j(g)} x(i(g)), \qquad \text{and} \qquad j(i^-(g)) = j(i(g)).
\end{equation*}
Therefore, since $\gamma_F(u, \tfrac{1}{2} \psi)^2 = (u,-1)_F$,
\begin{align*}
\hat s^-(g)
&= \xi(x(i^-(g))) ((u,-1)_F \gamma_F(u, \tfrac{1}{2} \psi))^{-j(i^-(g))} \\
&= \xi(x(i(g))) ((u,-1)_F \gamma_F(u, \tfrac{1}{2} \psi))^{j(i(g))} 
= \xi(x(i(g)))^2 \overline{\hat s(g)} 
= \xi(\det(g)) \overline{\hat s(g)}. \qedhere
\end{align*}
\end{proof}

\begin{lemma}\label{l:double compat0}
For $(g_1, g_2, h) \in \G(\U_E(V_0) \times \U_E(V_0) \times \U_E(W_0))$,
\begin{equation*}
s^\square(g_1,g_2,h) = s(g_1,h) \cdot \overline{s(g_2,h)} \cdot \xi(\det(i(g_2,h))).
\end{equation*}
\end{lemma}

\begin{proof}
This is \cite[Lemma 1.1]{HKS96}. See also [Lemma D.4, periods2].
\end{proof}

\subsubsection{Splittings for $\G(\U_E(\Res V) \times \U_E(W^\square))$ and $\G(\U_E(\Res V) \times \U_E(W))$}\label{s:splittings'}

This section is completely analogous to Section \ref{s:splittings0}. The $2$-dimensional $E$-space $\Res V \otimes_E W$ with skew-Hermitian form $\overline{(\cdot, \cdot)} \otimes \langle \cdot, \cdot \rangle$ can be identified with $\Res V$. Define
\begin{align*}
i' \from \G(\U_E(\Res V) \times \U_E(W)) &\to \U_E(\Res V^\square), &
(g,h) &\mapsto ((v,v^-) \mapsto (g^{-1}vh, v^-)), \\
i^-{}' \from \G(\U_E(\Res V) \times \U_E(W)) &\to \U_E(\Res V^\square), &
(g,h) &\mapsto ((v,v^-) \mapsto (v, g^{-1}v^-h), \\
i^\square{}' \from \G(\U_E(\Res V) \times \U_E(W^\square)) &\to \U_E(\Res V^\square), &
(g,h) &\mapsto (v \mapsto g^{-1} v h).
\end{align*}
We have natural embeddings
\begin{align*}
&\G(\U_E(\Res V) \times \U_E(W) \times \U_E(W)) \hookrightarrow \G(\U_E(\Res V) \times \U_E(W)) \times \G(\U_E(\Res V) \times \U_E(W)), \\
&\G(\U_E(\Res V) \times \U_E(W) \times \U_E(W)) \hookrightarrow \G(\U_E(\Res V) \times \U_E(W^\square)).
\end{align*}
Observe that for $(g,h_1, h_2) \in \G(\U_E(\Res V) \times \U_E(W) \times \U_E(W))$,
\begin{equation*}
i'(g, h_1) i^-{}'(g,h_2) = i^\square{}'(g,h_1,h_2) \in \U_E(\Res V^\square).
\end{equation*}
We identify $\Res_{B/F}(V^\square) = \bV^\square$ and let
\begin{equation*}
\iota' \from \U_E(\Res V^\square) \to \Sp(\bV^\square)
\end{equation*}
be the natural embedding. We will often identify $\U_E(\Res V^\square)$ with $\iota(\U_E(\Res V^\square))$.

\begin{definition}\label{d:splittings'}
Pick a character $\xi' \from E^\times \to \CC^1$ such that $\xi|_{F^\times} = \epsilon_{E/F}$. Define
\begin{equation*}
\beta' \from \U_E(\Res V^\square) \to \CC^1, \qquad g \mapsto \xi'(x(g)) \cdot ((u, -1)_F \gamma_F(u, \tfrac{1}{2} \psi))^{-j(g)}.
\end{equation*}
Define
\begin{equation*}
\lambda' \colonequals \lambda_{\Res V^\triangle \otimes W \rightsquigarrow \bY^\square} \from \Sp(\bV^\square) \to \CC^1.
\end{equation*}
Define
\begin{align*}
\hat s' &\colonequals (i')^* \beta', & \hat s^-{}' &\colonequals (i^-{}')^* \beta', & \hat s^\square{}' &\colonequals (i^\square{}')^* \beta', \\
s' &\colonequals (i')^* (\beta'\lambda'), & s^-{}' &\colonequals (i^-{}')^* (\beta'\lambda'), & s^\square{}' &\colonequals (i^\square{}')^* (\beta'\lambda').
\end{align*}
\end{definition}

\begin{lemma}\mbox{}
\begin{enumerate}[label=(\alph*)]
\item
$\hat s'$, $\hat s^-{}'$, and $\hat s^\square{}'$ are splittings of $z_{\Res V^\triangle \otimes W}$ on the images of $i'$, $i^-{}'$, and $i^\square{}'$, respectively.

\item
$s'$ is a splitting of $z_{\bY}$ on the image of $i'$, $s^-{}'$ is a splitting of $z_{\bY}^{-1}$ on the image of $i^-{}'$, and $s^\square{}'$ is a splitting of $z_{\bY^\square}$ on the image of $i^\square{}'$.
\end{enumerate}
\end{lemma}


\begin{lemma}\label{l:double compat'}
For $(g,h_1, h_2) \in \G(\U_E(\Res V) \times \U_E(W) \times \U_E(W))$, 
\begin{equation*}
s^\square{}'(g, h_1, h_2) = s'(g, h_1) \cdot \overline{s'(g,h_2)} \cdot \xi'(\det(i'(g,h_2))).
\end{equation*}
\end{lemma}

\subsection{Compatibility between the splittings for the three seesaws} \label{s:compat}

In this section, we investigate the compatibility of the splittings of the four pairs of unitary groups relative to the three seesaws presented in \eqref{e:single seesaw} and \eqref{e:double seesaw}. Compatibility of the splittings in the two doubling seesaws of \eqref{e:double seesaw} is explicated in Lemmas \ref{l:double compat0} and \ref{l:double compat'}. Hence it remains to investigate the compatibility of the splittings
\begin{equation*}
s \from \G(\U_E(V_0) \times \U_E(W_0)) \to \CC^1 \qquad \text{and} \qquad s' \from \G(\U_E(\Res V) \times \U_E(W)) \to \CC^1.
\end{equation*}
Precisely, we would compare $s$ and $s'$ on the subgroup 
\begin{equation*}
\G(\U_E(V_0) \times \U_E(W)) \cong \{(\alpha, \beta) \in E^\times \times E^\times : \Nm(\alpha) = \Nm(\beta)\}.
\end{equation*}
We prove a sequence of lemmas that to break up the computation showing Proposition \ref{p:compat}. 

Let $\alpha, \beta \in E^\times$ with $\Nm(\alpha) = \Nm(\beta)$ so that $(\alpha, \beta) \in \G(\U_E(V_0) \times \U_E(W))$. Let $g \in \U_E(W_0^\square)$ denote the map $(w, w^-) \mapsto (\alpha^{-1} w \beta, w^-)$ and let $g' \in \U_E(\Res V^\square)$ denote the map $(v, v^-) \mapsto (\alpha^{-1} v \beta, v^-)$. Define:
\begin{equation*}
v_1 \colonequals \left(-\frac{\bi}{2u}, \frac{\bi}{2u}\right), \qquad v_2 \colonequals \left(\frac{\bi\bj}{2uJ}, -\frac{\bi\bj}{2uJ}\right), \qquad v_1' \colonequals (1,1), \qquad v_2' \colonequals (\bj, \bj).
\end{equation*}
This defines an $E$-basis of $W_0^\square$ and of $\Res V^\square$ with the following property:
\begin{equation*}
( v_i, v_j' )_0 = \delta_{ij}, \qquad ( v_i, v_j )_0 = ( v_i', v_j' )_0 = 0, \qquad
\langle v_i, v_j' \rangle = \delta_{ij}, \qquad \langle v_i, v_j \rangle = \langle v_i', v_j' \rangle = 0.
\end{equation*}
With respect to the basis $\{v_1, v_2, v_1', v_2'\}$,
\begin{align} \label{e:g}
g &= \left(\begin{matrix}
\frac{1 + \alpha^{-1}\beta}{2} & 0 & \frac{1 - \alpha^{-1}\beta}{4u} \bi & 0 \\
0 & \frac{1 + \alpha^{-1} \overline \beta}{2} & 0 & -\frac{1 - \alpha^{-1} \overline \beta}{4uJ} \bi \\
(1 - \alpha^{-1} \beta) \bi & 0 & \frac{1 + \alpha^{-1} \beta}{2} & 0 \\
0 & -(1 - \alpha^{-1} \overline \beta) \bi J & 0 & \frac{1 + \alpha^{-1} \overline \beta}{2}
\end{matrix}\right) \\ \label{e:g'}
g' &= \left(\begin{matrix}
\frac{1 + \alpha^{-1}\beta}{2} & 0 & \frac{1 - \alpha^{-1}\beta}{4u} \bi & 0 \\
0 & \frac{1 + \overline \alpha^{-1} \beta}{2} & 0 & \frac{1 - \overline \alpha^{-1} \beta}{4uJ} \bi \\
(1 - \alpha^{-1} \beta) \bi & 0 & \frac{1 + \alpha^{-1} \beta}{2} & 0 \\
0 & (1 - \overline \alpha^{-1} \beta) \bi J & 0 & \frac{1 + \overline \alpha^{-1} \beta}{2}
\end{matrix}\right)
\end{align}
Here, we view each unitary group as a subgroup of $\GL_4(E)$ with $\GL_4(E)$ acting formally by right-multiplication. Note however that $W_0^\square$ is a left $E$-space, and so we interpret the formal multiplication $v \cdot a$ for $v \in W_0^\square$ and $a \in E$ as $av$. Throughout this section, we write $\bg$ when we want to refer to one of $g$ or $g'$ simultaneously.

\begin{lemma}\label{l:pwp}
We have

\begin{equation*}
\begin{tabular}{| c | c | c | c |} \hline
Conditions & $x(g)$ & $x(g')$ & $j(\bg)$ \\ \hline
$\alpha^{-1} \beta = 1,$ $\alpha^{-1} \overline \beta = 1$ & $1$ & $1$ & $0$ \\ \hline
$\alpha^{-1} \beta = 1,$ $\alpha^{-1} \overline \beta \neq 1$ & $-(1-\alpha^{-1} \overline \beta) \bi J$ & $(1 - \overline \alpha^{-1} \beta) \bi J$ & $1$ \\ \hline
$\alpha^{-1} \beta \neq 1$, $\alpha^{-1} \overline \beta = 1$ & $(1 - \alpha^{-1} \beta) \bi$ & $(1 - \alpha^{-1} \beta)\bi$ & $1$ \\ \hline
$\alpha^{-1} \beta \neq 1$, $\alpha^{-1} \overline \beta \neq 1$ & $-(1 - \alpha^{-1} \beta)(1 - \alpha^{-1} \overline \beta) uJ$ & $(1 - \alpha^{-1} \beta)(1 - \overline \alpha^{-1} \beta) u J$ & $2$ \\ \hline
\end{tabular}
\end{equation*}
\end{lemma}

\begin{proof}
The proof amounts to giving explicit decompositions
\begin{equation*}
\bg = p_1 w p_2, \qquad \text{where $p_i \in P_{\V^\triangle}$ and $w = \tau_j = \left(\begin{smallmatrix} 1_{2-j} & & & \\ & & & -1_j \\ & & 1_{2-j} & \\ & 1_j & & & \end{smallmatrix}\right)$.}
\end{equation*}
There are four cases:
\begin{enumerate}[label=(\alph*)]
\item
If $\alpha^{-1} \beta = 1$ and $\alpha^{-1} \overline\beta = 1$, then
\begin{equation*}
g = 1, \qquad g' = 1.
\end{equation*}

\item
If $\alpha^{-1} \beta = 1$ and $\alpha^{-1} \overline \beta \neq 1$, then $g = p_1 \tau_1 p_2$ and $g' = p_1' \tau_1 p_2'$ for
\begin{align*}
p_1 &= \left(\begin{smallmatrix}
1 & 0 & 0 & 0 \\
0 & 1 & 0 & \frac{1 + \alpha^{-1} \overline \beta}{2(-1 + \alpha^{-1} \overline \beta) \bi J} \\
0 & 0 & 1 & 0 \\
0 & 0 & 0 & 1
\end{smallmatrix}\right),
&
p_2 &= \left(\begin{smallmatrix}
1 & 0 & 0 & 0 \\
0 & (-1 + \alpha^{-1} \overline \beta) \bi J & 0 & \frac{1+\alpha^{-1} \overline \beta}{2} \\
0 & 0 & 1 & 0 \\
0 & 0 & 0 & \frac{\alpha^{-1} \overline \beta}{(-1 + \alpha^{-1} \overline \beta) \bi J}
\end{smallmatrix}\right), \\
p_1' &= \left(\begin{smallmatrix}
1 & 0 & 0 & 0 \\
0 & 1 & 0 & -\frac{1 + \overline \alpha^{-1} \beta}{2(-1 + \overline \alpha^{-1} \beta) \bi J} \\
0 & 0 & 1 & 0 \\ 
0 & 0 & 0 & 1
\end{smallmatrix}\right),
&
p_2' &= \left(\begin{smallmatrix}
1 & 0 & 0 & 0 \\
0 & -(-1 + \overline \alpha^{-1} \beta) \bi J & 0 & \frac{1 + \overline \alpha^{-1} \beta}{2} \\
0 & 0 & 1 & 0 \\
0 & 0 & 0 & -\frac{\overline \alpha^{-1} \beta}{(-1 + \overline \alpha^{-1} \beta)\bi J}
\end{smallmatrix}\right).
\end{align*}

\item
If $\alpha^{-1} \beta \neq 1$ and $\alpha^{-1} \overline \beta = 1$, then
\begin{align*}
g &= g' = \left(\begin{smallmatrix}
0 & 1 & 0 & \frac{1 + \alpha^{-1} \beta}{2\bi(1-\alpha^{-1} \beta)} \\
1 & 0 & 0 & 0 \\
0 & 0 & 0 & 1 \\
0 & 0 & 1 & 0
\end{smallmatrix}\right) \cdot \tau_1 \cdot
\left(\begin{smallmatrix}
0 & 1 & 0 & 0 \\
(1-\alpha^{-1}\beta) \bi & 0 & \frac{1 + \alpha^{-1} \beta}{2} & 0 \\
0 & 0 & 0 & 1\\
0 & 0 & \frac{\alpha^{-1} \beta}{(1 - \alpha^{-1} \beta) \bi} & 0
\end{smallmatrix}\right). 
\end{align*}

\item
If $\alpha^{-1} \beta \neq 1$ and $\alpha^{-1} \overline \beta \neq 1$, then $g = p_1 \tau_2 p_2$ and $g' = p_1' \tau_2 p_2'$ for
\begin{align*}
p_1 &= 
\left(\begin{smallmatrix}
1 & 0 & \frac{1 + \alpha^{-1} \beta}{2(1 - \alpha^{-1} \beta) \bi} & 0 \\
0 & 1 & 0 & -\frac{1 + \alpha^{-1} \overline\beta}{2(1 -  \alpha^{-1} \overline\beta) \bi J} \\
0 & 0 & 1 & 0 \\
0 & 0 & 0 & 1
\end{smallmatrix}\right),
&
p_2 &= \left(\begin{smallmatrix}
(1 - \alpha^{-1} \beta) \bi & 0 & \frac{1 + \alpha^{-1} \beta}{2} & 0 \\
0 & -(1 -  \alpha^{-1} \overline\beta) \bi J  & 0 & \frac{1 +  \alpha^{-1} \overline\beta}{2} \\ 
0 & 0 & \frac{\alpha^{-1} \beta}{(1 - \alpha^{-1} \beta) \bi} & 0 \\
0 & 0 & 0 & -\frac{ \alpha^{-1} \overline\beta}{(1 -  \alpha^{-1} \overline\beta)\bi J}
\end{smallmatrix}\right), \\
p_1' &=
\left(\begin{smallmatrix}
1 & 0 & \frac{1 + \alpha^{-1} \beta}{2(1 - \alpha^{-1} \beta) \bi} & 0 \\
0 & 1 & 0 & \frac{1 + \overline \alpha^{-1} \beta}{2(1 - \overline \alpha^{-1} \beta) \bi J} \\
0 & 0 & 1 & 0 \\
0 & 0 & 0 & 1
\end{smallmatrix}\right),
&
p_2' &= \left(\begin{smallmatrix}
(1 - \alpha^{-1} \beta) \bi & 0 & \frac{1 + \alpha^{-1} \beta}{2} & 0 \\
0 & (1 - \overline \alpha^{-1} \beta) \bi J  & 0 & \frac{1 + \overline \alpha^{-1} \beta}{2} \\ 
0 & 0 & \frac{\alpha^{-1} \beta}{(1 - \alpha^{-1} \beta) \bi} & 0 \\
0 & 0 & 0 & \frac{\overline \alpha^{-1} \beta}{(1 - \overline \alpha^{-1} \beta)\bi J}
\end{smallmatrix}\right).
\end{align*}
\end{enumerate}
From the above decompositions, we can easily read off the desired information.
\end{proof}

\begin{lemma}\label{l:s hat a,a}
Let $\alpha = a_1 + b_1\bi$. Then
\begin{align*}
\hat s(\alpha,\alpha) &= \begin{cases}
\xi(\alpha^{-1}) \cdot (a_1,u)_F & \text{if $b_1 = 0$,} \\
\xi(\alpha^{-1}) \cdot (-2b_1 u J, u)_F \cdot \gamma_F(u, \frac{1}{2} \psi) \cdot (-1,-u)_F & \text{otherwise.}
\end{cases} \\
\hat s'(\alpha,\alpha) &= \begin{cases}
\xi'(\overline \alpha^{-1}) \cdot (a_1, u)_F & \text{if $b_1 = 0$,} \\
\xi'(\overline \alpha^{-1}) \cdot (-2b_1 u J, u)_F \cdot \gamma_F(u, \frac{1}{2} \psi) \cdot (-1,-u)_F & \text{otherwise.}
\end{cases}
\end{align*}
\end{lemma}

\begin{proof}
We use Lemma \ref{l:pwp} in the two cases where $\alpha^{-1} \beta = 1$. If $\alpha^{-1} \overline \alpha = 1$, then $\alpha = \overline \alpha$ and so $b_1 = 0$. By Lemma \ref{l:pwp}, we have
\begin{equation*}
\hat s(\alpha,\alpha) = \hat s'(\alpha,\alpha) = 1 = \xi(\alpha^{-1}) \cdot (a_1, u)_F = \xi'(\alpha^{-1}) \cdot (a_1, u)_F.
\end{equation*}
If $\alpha^{-1} \overline \alpha \neq 1$, then $b_1 \neq 0$. Note that
\begin{equation*}
1 - \alpha^{-1} \overline \alpha = \alpha^{-1}(\alpha - \overline \alpha) = \alpha^{-1} \cdot 2 b_1 \bi, \qquad
1 - \overline \alpha^{-1} \alpha = \overline{1 - \alpha^{-1} \overline\alpha} = -\overline \alpha^{-1} \cdot 2 b_1 \bi.
\end{equation*}
The desired conclusion now follows by Lemma \ref{l:pwp}.
\end{proof}

\begin{lemma}\label{l:s hat 1,z}
Let $\zeta = a + b \bi \in E^1$. Then
\begin{equation*}
\hat s(1,\zeta) = 
\begin{cases}
1 & \text{if $a = 1$,} \\
((2-2a)uJ, u)_F & \text{if $a \neq 1$,}
\end{cases}
\qquad
\hat s'(1,\zeta) =
\begin{cases}
1 & \text{if $a = 1$,} \\
\xi'(\zeta) \cdot ((2-2a)uJ, u)_F & \text{if $a \neq 1$.}
\end{cases}
\end{equation*}
\end{lemma}

\begin{proof}
We use Lemma \ref{l:pwp}. If $\zeta = 1$, this corresponds to the case $\alpha^{-1} \beta = 1$, $\alpha^{-1} \overline \beta = 1$, and
\begin{align*}
\hat s(1,\zeta) = \hat s'(1,\zeta) = 1.
\end{align*}
If $\zeta \neq 1$, this corresponds to the case $\alpha^{-1} \beta \neq 1$, $\alpha^{-1} \overline \beta \neq 1$, and
\begin{equation*}
\hat s(1,\zeta) = \xi(-(1-\zeta)(1-\overline \zeta) uJ) \cdot (-1,u)_F, \qquad
\hat s'(1,\zeta) = \xi'((1-\zeta)^2 uJ) \cdot (-1,u)_F.
\end{equation*}
The desired conclusion follows from the simple observation
\begin{equation*}
(1-\zeta)(1-\overline \zeta) = 2-2a, \qquad
(1-\zeta)^2 = -\zeta(1-\zeta)(1-\overline \zeta) = -\zeta(2-2a). \qedhere
\end{equation*}
\end{proof}

\begin{proposition}\label{p:compat}
Let $g \in \G(\U_E(V_0) \times \U_E(W)) \subset \G(\U_E(V_0) \times \U_E(W_0))$ and $g' \in \G(\U_E(V_0) \times \U_E(W)) \subset \G(\U_E(\Res V) \times \U_E(W))$ correspond to $(\alpha,\beta) \in E^\times \times E^\times$ with $\Nm(\alpha) = \Nm(\beta)$. Then
\begin{equation*}
s'(g') = \xi(\alpha) \xi'(\beta) s(g).
\end{equation*}
\end{proposition}

\begin{proof}
We use the formulas given in Lemma \ref{l:s hat a,a} and Lemma \ref{l:s hat 1,z} together with Lemma \ref{l:change polarization}. Recall that $g = g_1 \cdot g_2,$ $g' = g_1' \cdot g_2',$ where $\bg_1$ corresponds to $(\alpha,\alpha)$ and $\bg_2$ corresponds to $(1,\beta/\alpha)$.

First notice that under the natural maps
\begin{align*}
i &\from \U_E(V_0 \otimes W_0) \to \Sp(\bV), &
i^\square &\from \U_E(V_0 \otimes W_0) \to \Sp(\bV^\square), \\
i' &\from \U_E(\Res V \otimes W) \to \Sp(\bV), &
i^\square{}' &\from \U_E(\Res V \otimes W) \to \Sp(\bV^\square),
\end{align*}
we have
\begin{equation*}
i(g_\bullet) = i'(g_\bullet') \in \Sp(\bV), \qquad
i^\square(g_\bullet) = i^\square{}'(g_\bullet') \in \Sp(\bV^\square),
\end{equation*}
where $\bg_\bullet$ denotes any of $\bg,$ $\bg_1$, $\bg_2$. This implies that for $\lambda \colonequals \lambda_{\bV^\triangle \rightsquigarrow \bY^\square},$
\begin{equation*}
\lambda(i^\square(g_\bullet)) = \lambda(i^\square{}'(g_\bullet')), \qquad \text{and} \qquad z_\bY(i(g_1), i(g_2)) = z_\bY(i'(g_1'), i'(g_2')).
\end{equation*}
By definition,
\begin{align*}
s(\bg) &= \hat s(\bg_1) \cdot \mu(\bg_1) \cdot \hat s(\bg_2) \cdot \mu(\bg_2) \cdot z_\bY(i(\bg_1), i(\bg_2)), \\
s'(\bg) &= \hat s'(\bg_1) \cdot \mu(\bg_1) \cdot \hat s'(\bg_2) \cdot \mu(\bg_2) \cdot z_\bY(i'(\bg_1), i'(\bg_2)),
\end{align*}
Thus we have
\begin{align*}
\bchi(\alpha,\beta) = s(g) \cdot s'(g')^{-1} = \hat s(g_1) \cdot \hat s(g_2) \cdot \hat s'(g_1')^{-1} \cdot \hat s'(g_2')^{-1}.
\end{align*}

Now we combine the results of Lemmas \ref{l:s hat a,a} and \ref{l:s hat 1,z} to compute $\bchi(\alpha,\beta)$. Using the fact
\begin{equation*}
\overline \alpha^{-1} \cdot \beta \cdot \alpha^{-1} = \overline \beta^{-1},
\end{equation*}
in the calculation of $\hat s'(g_1') \hat s'(g_2')$ when $\alpha \neq \beta$, we have:
\begin{align*}
\hat s(g_1) \cdot \hat s(g_2)
&= \begin{cases}
\xi(\alpha^{-1}) \cdot (a_1, u)_F & \text{$\alpha \in F^\times$, $\alpha = \beta$} \\
\xi(\alpha^{-1}) \cdot (a_1, u)_F \cdot ((2-2a)uJ,u)_F & \text{$\alpha \in F^\times$, $\alpha \neq \beta$} \\
\xi(\alpha^{-1}) \cdot (-2b_1uJ, u)_F \cdot \gamma_F(u, \frac{1}{2} \psi) \cdot (-1,-u)_F & \text{$\alpha \not\in F^\times$, $\alpha = \beta$} \\
\xi(\alpha^{-1}) \cdot (-2b_1 u J, u)_F \cdot \gamma_F(u, \frac{1}{2} \psi) \cdot (-1,-u)_F \\
\qquad\qquad\qquad\qquad\qquad\qquad\qquad \cdot ((2-2a) uJ, u)_F & \text{$\alpha \not\in F^\times$, $\alpha \neq \beta$}
\end{cases} \\
\hat s'(g_1') \cdot \hat s'(g_2')
&= \begin{cases}
\xi'(\overline \alpha^{-1}) \cdot (a_1, u)_F &\text{$\alpha \in F^\times$, $\alpha = \beta$} \\
\xi'(\overline \beta^{-1}) \cdot (a_1, u)_F \cdot ((2-2a)uJ, u)_F &\text{$\alpha \in F^\times$, $\alpha \neq \beta$} \\
\xi'(\overline \alpha^{-1}) \cdot (-2b_1uJ, u)_F \cdot \gamma_F(u, \frac{1}{2}\psi) \cdot (-1,-u)_F & \text{$\alpha \not\in F,$ $\alpha = \beta$} \\
\xi'(\overline \beta^{-1}) \cdot (-2b_1 uJ, u)_F \cdot \gamma_F(u, \frac{1}{2} \psi) \cdot (-1, -u)_F \\
\qquad\qquad\qquad\qquad\qquad\qquad\qquad \cdot ((2-2a)uJ, u)_F & \text{$\alpha \not\in F^\times$, $\alpha \neq \beta$}
\end{cases}
\end{align*}
Therefore
\begin{align*}
\bchi(\alpha,\beta)
&= \begin{cases}
\xi(\alpha^{-1}) \cdot \xi'(\overline \alpha) & \text{$\alpha \in F^\times$, $\alpha = \beta$} \\
\xi(\alpha^{-1}) \cdot \xi'(\overline \beta) & \text{$\alpha \in F^\times$, $\alpha \neq \beta$} \\
\xi(\alpha^{-1}) \cdot \xi'(\overline \alpha) & \text{$\alpha \not \in F^\times$, $\alpha = \beta$} \\
\xi(\alpha^{-1}) \cdot \xi'(\overline \beta) & \text{$\alpha \not\in F^\times$, $\alpha \neq \beta$}
\end{cases} \\
&= \xi(\alpha^{-1}) \cdot \xi'(\overline \beta) = \xi(\alpha^{-1}) \cdot \xi'(\beta^{-1}) \cdot \xi'(\beta \overline \beta) = \xi(\alpha^{-1}) \xi'(\beta^{-1}). \qedhere
\end{align*}
\end{proof}

\subsection{Product formula}\label{s:global splittings}

In this section, we put the local considerations of the Sections \ref{s:kudla}, \ref{s:change polarization}, \ref{s:three seesaws}, and \ref{s:compat} into the global picture. Once and for all, pick Hecke characters
\begin{equation*}
\xi, \xi' \from E^\times \backslash \bA_E^\times \to \CC^1 \qquad \text{such that $\xi|_{\bA_F^\times} = \xi'|_{\bA_F^\times} = \epsilon_{E/F}.$}
\end{equation*}
Note that $\U_E(V_0) \cong E^\times \cong \U_B(V)^0$ and hence we have a natural embeddings
\begin{align*}
\G(\U_B(V)^0 \times \U_B(W)) &\hookrightarrow \G(\U_E(V_0) \times \U_E(W_0)) \\
\G(\U_B(V^\square)^0 \times \U_B(W)) &\hookrightarrow \G(\U_E(V_0^\square) \times \U_E(W_0)).
\end{align*}
Thus functions defined on the unitary spaces pull back to functions on the quaternionic unitary spaces. For each place $v$ of $F$, by Definition \ref{d:splittings0} and \ref{d:splittings'}, we have functions 
\begin{align*}
s_v &\from \G(\U_B(V_{v}) \times \U_B(W_{v}^*)) \to \CC^1, & s_v^\square &\from \G(\U_B(V_{v}^\square)^0 \times \U_B(W_{v}^*)) \to \CC^1, \\
s_v' &\from \G(\U_E(\Res V_v) \times \U_E(W_v)) \to \CC^1, & s_v^\square{}' &\from \G(\U_E(\Res V_v) \times \U_E(W_v^\square)) \to \CC^1.
\end{align*}
Formally define
\begin{equation*}
s \colonequals \prod_v s_v, \qquad s' \colonequals \prod_v s_v', \qquad s^\square \colonequals \prod_v s_v^\square, \qquad s'{}^\square \colonequals \prod_v s_v'{}^\square.
\end{equation*}
These products converge by the following lemma, where we write ``a.a.'' for ``all but finitely many.''

\begin{lemma} \mbox{}
\begin{enumerate}[label=(\alph*)]
\item
Let $\gamma \in \G(\U_B(V)(F) \times \U_B(W)(F))$. Then $s_v(\gamma) = 1$ for a.a.\ $v$ and $s(\gamma) = 1$.

\item
Let $\gamma \in \G(\U_B(V^\square)^0(F) \times \U_B(W)(F))$. Then $s_v^\square(\gamma) = 1$ for a.a.\ $v$ and $s^\square(\gamma) = 1$.

\item
Let $\gamma \in \G(\U_E(\Res V)(F) \times \U_E(W)(F)).$ Then $s_v'(\gamma) = 1$ for a.a.\ $v$ and $s'(\gamma) = 1$.

\item
Let $\gamma \in \G(\U_E(\Res V)(F) \times \U_E(W^\square)(F))$. Then $s_v^\square{}'(\gamma) = 1$ for a.a.\ $v$ and $s^\square{}'(\gamma) = 1$.
\end{enumerate}
\end{lemma}

\begin{proposition}\label{p:global compat}\mbox{}
\begin{enumerate}[label=(\alph*)]
\item {} [Lemma \ref{l:double compat0}]
For $(g_1, g_2, h) \in \G(\U_B(V)^0(\bA) \times \U_B(V)^0(\bA) \times \U_B(W)(\bA))$,
\begin{equation*}
s^\square(g_1, g_2, h) = s(g_1, h) \cdot \overline{s(g_2, h)} \cdot \xi(\det(i(g_2, h))).
\end{equation*}

\item {} [Lemma \ref{l:double compat'}]
For $(h, g_1, g_2) \in \G(\U_E(\Res V)(\bA) \times \U_E(W)(\bA) \times \U_E(W)(\bA))$,
\begin{equation*}
s^\square{}'(h, g_1, g_2) = s'(h, g_1) \cdot \overline{s'(h, g_2)} \cdot \xi'(\det(i'(h,g_2))).
\end{equation*}

\item {} [Proposition \ref{p:compat}]
For $\alpha, \beta \in \bA_E^\times$ such that $\Nm(\alpha) = \Nm(\beta)$,
\begin{equation*}
s'(\alpha, \beta) = \xi(\alpha) \xi'(\beta) s(\alpha, \beta).
\end{equation*}
\end{enumerate}
\end{proposition}

\subsection{Two splittings on $E_v^\times \times \GL_2(F_v)$} \label{s:split splittings}

To calculate the theta lift at all the unramified places, we will have to understand the Weil representation more concretely. In particular, we will need to explicate the local splittings defined in Section \ref{ch:splittings} in the cases $v \notin \Sigma_B$ and $v \notin \Sigma_{B'}$. These exactly correspond, respectively, to the cases when $W_{0,v}$ and $\Res V_v$ are split Hermitian spaces. For notational convenience, we drop the subscript $v$ in this section.

Consider the group
\begin{equation*}
R \colonequals \G(E^\times \times \GL_2(F)) = \left\{(\alpha,g) \in E^\times \times \GL_2(F) : \Nm(\alpha) = \det(g)\right\}.
\end{equation*}
Assume that the $2$-dimensional $E$-spaces $W_0$ and $\Res V$ are hyperbolic planes (i.e.\ they are split Hermitian spaces). Then we have embeddings
\begin{align*}
R &\hookrightarrow \G(\U_E(V_0) \times \U_E(W_0)), \qquad (\alpha, g) \mapsto (\alpha, g) \\
R &\hookrightarrow \G(\U_E(\Res V) \times \U_E(W)), \qquad (\alpha,g) \mapsto (g,\alpha).
\end{align*}
Furthermore, any decomposition of $W_0$ or $\Res V$ into maximal isotropic subspaces induces a complete polarization 
\begin{equation*}
\bV = \bX' + \bY'.
\end{equation*}
Our goal in this section is to explicate the values of the splittings in Section \ref{s:three seesaws} associated to this particular polarization. To make it clear that we are working in this specialized context, we let
\begin{equation*}
\bs \from \G(\U_E(V_0) \times \U_E(W_0)) \to \CC^1, \qquad \bs' \from \G(\U_E(\Res V) \times \U_E(W)) \to \CC^1
\end{equation*}
denote the splittings for $z_{\bY'}$ defined in Section \ref{s:three seesaws}. 

We briefly recall the construction of $\bs$, $\bs'$. 
Recall that from Sections \ref{s:splittings0} and \ref{s:splittings'}, we have natural maps
\begin{equation*}
i \from \G(\U_E(V_0) \times \U_E(W_0)) \to \U_E(W_0^\square), \qquad i' \from \G(\U_E(\Res V) \times \U_E(W)) \to \U_E(\Res V^\square).
\end{equation*}
If we let $\blambda \from \Sp(\bV^\square) \to \CC^1$ be given by
\begin{equation*}
\blambda(g) \colonequals \gamma_F(\tfrac{1}{2} \psi \circ q(\bV^\triangle, \bY'{}^\square g^{-1}, \bY'{}^\square)) \cdot \gamma_F(\tfrac{1}{2} \circ q(\bV^\triangle, \bY'{}^\square, \bV^\triangle g)),
\end{equation*}
then we have
\begin{equation*}
\bs \colonequals \hat s \cdot \blambda \from \G(\U_E(V_0) \times \U_E(W_0)) \to \CC^1, \qquad
\bs' \colonequals \hat s' \cdot \blambda \from \G(\U_E(\Res V) \times \U_E(W)) \to \CC^1,
\end{equation*}
where $\partial \hat s = z_{V_0 \otimes W_0^\triangle}$ and $\partial \hat s' = z_{\Res V^\triangle \otimes W}$.

Let $W_1$ and $W_2$ be isotropic subspaces such that $W_0 = W_1 + W_2$ and fix $w_i \in W_i$ so that $\langle w_1, w_2 \rangle = 1$. Analogously, let $V_1$ and $V_2$ be isotropic subspaces such that $\Res V = V_1 + V_2$ and fix $w_i \in V_i$ such that $\langle w_1, w_2 \rangle = 1$.  
Define
\begin{equation*}
\bw_1 = (\tfrac{1}{2} w_1, -\tfrac{1}{2} w_1), \qquad 
\bw_2 = (-\tfrac{1}{2} w_2, \tfrac{1}{2} w_2), \qquad
\bw_1^* = (w_2, w_2), \qquad
\bw_2^* = (w_1, w_1)
\end{equation*}
so that we have $\langle \bw_i, \bw_j \rangle = \langle \bw_i^*, \bw_j^* \rangle,$ $\langle \bw_i, \bw_j^* \rangle = \delta_{ij},$ $\langle \bw_i^*, \bw_j \rangle = -\delta_{ij},$ and\begin{align*}
W_0^\square &= W_0^\bigtriangledown + W_0^\triangle, && \text{where $W_0^\bigtriangledown = \Span\{\bw_1, \bw_2\}$ and $W_0^\triangle = \Span\{\bw_1^*, \bw_2^*\}$}, \\
\Res V^\square &= \Res V^\bigtriangledown + \Res V^\triangle, && \text{where $\Res V^\bigtriangledown = \Span\{\bw_1, \bw_2\}$ and $\Res V^\triangle = \Span\{\bw_1^*, \bw_2^*\}$}.
\end{align*}
Then a symplectic basis preserving the complete polarization$\bV^\square = \bV^\bigtriangledown + \bV^\triangle$ is given by
\begin{equation}\label{e:double F basis}
\bw_1, \quad \tfrac{-1}{u} \bi \bw_1, \quad \bw_2, \quad \tfrac{-1}{u} \bi \bw_2, \quad \bw_1^*, \quad \bi \bw_1^*, \quad \bw_2^*, \quad \bi \bw_2^*.
\end{equation}

\subsubsection{A splitting $\bs$ of $z_{\bY'}$}

For $a,d \in F^\times$, write $D(a,d) \colonequals \diag(a,d)$. 

\begin{lemma}\label{l:bs D}
Let $\left(\alpha, D(a,d)\right) \in R$. Then
\begin{equation*}
\bs(\alpha, D(a,d)) = \xi(-(\alpha^{-1} a - 1)(\alpha^{-1} d - 1)).
\end{equation*}
In particular, for $a \in F^\times$ and $\alpha \in E^\times$,
\begin{equation*}
\bs(1, D(a,a^{-1})) = (u, a)_F, \qquad
\bs(\alpha,D(1,\Nm(\alpha))) = \xi(\alpha^{-1}).
\end{equation*}
\end{lemma}

\begin{proof}
We have $(1, D(1,1)) = (1, U(0))$, and this is proved in Lemma \ref{l:bs U}, so we assume that $(\alpha, D(a,d)) \neq (1, D(1,1))$. This assumption will be necessary when we calculate $\hat s$.

Recall that $(\alpha, D(a,d))$ sends $w_1 \mapsto \alpha^{-1} a w_1$ and $w_2 \mapsto \alpha^{-1} d w_2$. Recalling that $i \from \U_E(W_0) \to \U_E(W_0 + W_0^-)$ is defined by $\U_E(W_0)$ acting linearly on $W_0$ and trivially on $W_0^-$, it is a straightforward computation to see that the image of $(\alpha, D(a,d))$ in $\U_E(W_0 + W_0^-)$ with respect to the basis $\bw_1, \bw_2, \bw_1^*, \bw_2^*$ is
\begin{equation*}
\left(\begin{smallmatrix}
\tfrac{\alpha^{-1} a +1}{2} & 0 & 0 & \tfrac{\alpha^{-1} a - 1}{4} \\
0 & \tfrac{\alpha^{-1} d + 1}{2} & -\tfrac{\alpha^{-1} d - 1}{4} & 0 \\
0 & -(\alpha^{-1} d - 1) & \tfrac{\alpha^{-1} d + 1}{2} & 0 \\
-(\alpha^{-1} a - 1) & 0 & 0 & \tfrac{\alpha^{-1} a +1}{2}
\end{smallmatrix}\right).
\end{equation*}
We have
\begin{equation*}
i(\alpha, D(a,d)) = 
p_1 \left(\begin{smallmatrix}
& & 1 & \\
& & & 1 \\
-1 & & & \\
& -1 & & 
\end{smallmatrix}\right) p_2,
\end{equation*}
where
\begin{align*}
p_1 &= \left(\begin{smallmatrix}
0 & -\tfrac{(\alpha^{-1} a + 1)^2}{4(\alpha^{-1} a -1)} + \tfrac{\alpha^{-1} a - 1}{4} & -\tfrac{\alpha^{-1} a + 1}{2} & 0 \\
\tfrac{(\alpha^{-1} d + 1)^2}{4(\alpha^{-1} d - 1)} - \tfrac{\alpha^{-1} d - 1}{4} & 0 & 0 & -\tfrac{\alpha^{-1} d + 1}{2} \\
0 & 0 & 0 & (\alpha^{-1} d - 1) \\
0 & 0 & -(\alpha^{-1} a - 1) & 0
\end{smallmatrix}\right), \\
p_2 &= \left(\begin{smallmatrix}
1 & 0 & 0 & \tfrac{\alpha^{-1} a + 1}{2(\alpha^{-1} a - 1)} \\
0 & 1 & -\tfrac{\alpha^{-1} d + 1}{4(\alpha^{-1} d - 1)} & 0 \\
0 & 0 & 1 & 0 \\
0 & 0 & 0 & 1
\end{smallmatrix}\right).
\end{align*}
This implies that
\begin{equation*}
x(i(\alpha, D(a,d))) = (\alpha^{-1} a - 1)(\alpha^{-1} d - 1), \qquad j(i(\alpha,D(a,d))) = 2,
\end{equation*}
and therefore by Definition \ref{d:splittings0},
\begin{equation*}
\hat s(i(\alpha, D(a,d))) = \xi((\alpha^{-1} a - 1)(\alpha^{-1} d - 1)) \cdot \gamma_F(u, \tfrac{1}{2} \psi)^{-2} = \xi(-(\alpha^{-1} a - 1)(\alpha^{-1} d - 1)).
\end{equation*}
With respect to the symplectic basis given in \eqref{e:double F basis}, the image of $i(\alpha, D(a,d))$ in $\Sp(\bV^\square)$ is
\begin{equation*}
g = \left(\begin{matrix}
\tfrac{xa+1}{2} & -\tfrac{yau}{2} & & & & & \tfrac{xa-1}{4} & \tfrac{ya}{4} \\
-\tfrac{ya}{2} & \tfrac{xa+1}{2} & & & & & \tfrac{ya}{4u} & \tfrac{xa - 1}{4} \\
& & \tfrac{xd+1}{2} & -\tfrac{ydu}{2} & -\tfrac{xd-1}{4} & -\tfrac{yd}{4} & & \\
& & -\tfrac{yd}{2} & \tfrac{xd+1}{2} & -\tfrac{yd}{4u} & -\tfrac{xd-1}{4} & & \\
& & -(xd - 1) & ydu & \tfrac{xd+1}{2} & \tfrac{yd}{2} & & \\
& & yd & -(xd - 1) & \tfrac{yd}{2u} & \tfrac{xd+1}{2} & & \\
xa-1 & -yau & & & & & \tfrac{xa+1}{2} & \tfrac{ya}{2} \\
ya & xa-1 & & & & & \tfrac{ya}{2u} & \tfrac{xa+1}{2}
\end{matrix}\right) \in \Sp(\bV^\square).
\end{equation*}
By definition,
\begin{equation*}
\blambda(\alpha, D(a,d)) = \gamma_F(\tfrac{1}{2} \psi \circ q(\bV^\triangle, \bY'{}^\square g^{-1}, \bY'{}^\square)) \cdot \gamma_F(\tfrac{1}{2} \psi \circ q(\bV^\triangle, \bY'{}^\square, \bV^\triangle g)).
\end{equation*}
Since $g$ stabilizes $\bY'{}^\square$,
\begin{equation*}
\gamma_F(\tfrac{1}{2} \psi \circ q(\bV^\triangle, \bY'{}^\square g^{-1}, \bY'{}^\square)) = 1.
\end{equation*}
To calculate the second factor, notice that
\begin{align*}
\bV^\triangle &= \{(0, 0, 0, 0, z_1, z_2, z_3, z_4)\}, \\
\bY'{}^\square &= \{(0, 0, z_1, z_2, z_3, z_4, 0, 0)\}, \\
\bV^\triangle g &= \{((xa-1)z_2 + ya z_4, -yau z_3 + (xa-1) z_4, -(xd-1) z_1 + y dz_2, ydu z_1 - (xd-1)z_2, \\ &\qquad\qquad\qquad-\tfrac{xd+1}{2} z_1 - \tfrac{yd}{2u} z_2, -\tfrac{yd}{2} z_1 - \tfrac{xd+1}{2} z_2, \tfrac{xa+1}{2} z_3 + \tfrac{ya}{2u} z_4, \tfrac{ya}{2} z_3 + \tfrac{xa+1}{2} z_4)\},
\end{align*}
and one can see that this implies that $\bR = \{(0, 0, *, *, *, *, 0, 0)\}$ and hence 
\begin{equation*}
\gamma_F(\tfrac{1}{2} \psi \circ q(\bV^\triangle, \bY'{}^\square, \bV^\triangle g)) = 1.
\end{equation*}
We therefore have
\begin{equation*}
\bs(\alpha, D(a,d)) = \hat s(\alpha, D(a,d)) = \xi(-(\alpha^{-1} a - 1)(\alpha^{-1} d - 1)).
\end{equation*}
This proves the main assertion and the remaining formulas can be deduced as follows: Assuming $a \neq 1$ and $\alpha \neq 1$ (observe that if $\alpha \in E^1$, then $x = 1$ if and only if $\alpha = 1$),
\begin{equation*}
\bs(1, D(a,a^{-1})) 
= \xi(-(a - 1)(a^{-1} - 1)) = \xi(a^{-1}(a-1)^2) = \xi(a^{-1}) = (u, a)_F.
\end{equation*}
If $\alpha \in E^\times$, then
\begin{equation*}
\bs(\alpha, D(1, \Nm(\alpha)))
= \xi(-(\alpha^{-1}-1)(\alpha^{-1}\alpha\bar\alpha-1)) 
= \xi(\alpha^{-1}) \epsilon_{E/F}(\Nm(\alpha-1)) = \xi(\alpha^{-1}). \qedhere
\end{equation*}
\end{proof}

\begin{lemma}\label{l:bs U}
Let $a \in F$. Then
\begin{equation*}
\bs(1,U(a)) = 1.
\end{equation*}
\end{lemma}

\begin{proof}
The matrix $U(a)$ sends $w_1 \mapsto w_1 + a w_2$ and $w_2 \mapsto w_2$. Recalling that $i \from \U_E(W_0) \to \U_E(W_0 + W_0^-)$ is defined by $\U_E(W_0)$ acting linearly on $W_0$ and trivially on $W_0^-$, it is a straightforward computation to see that 
\begin{equation*}
i(1, U(a)) = \left(\begin{smallmatrix}
1 & -\tfrac{a}{2} & \tfrac{a}{4} & 0 \\
0 & 1 & 0 & 0 \\
0 & 0 & 1 & 0 \\
0 & -a & \tfrac{a}{2} & 1
\end{smallmatrix}\right).
\end{equation*}
We have
\begin{equation*}
\left(\begin{smallmatrix}
1 & \tfrac{a}{2} & \tfrac{a}{4} & 0 \\
0 & 1 & 0 & 0 \\
0 & 0 & 1 & 0 \\
0 & a & \tfrac{a}{2} & 1
\end{smallmatrix}\right)
\left(\begin{smallmatrix}
1 & 0 & 0 & 0 \\
0 & a^{-1} & 0 & 1 \\
0 & 0 & 1 & 0 \\
0 & 0 & 0 & a
\end{smallmatrix}\right)
=
\left(\begin{smallmatrix}
1 & \tfrac{a}{2} & \tfrac{a}{4} & -\tfrac{1}{2} \\
0 & -1 & 0 & a^{-1} \\
0 & 0 & 1 & 0 \\
0 & 0 & \tfrac{a}{2} & -1
\end{smallmatrix}\right)
\left(\begin{smallmatrix}
1 & & & \\
& & & -1 \\
& & 1 & \\
& 1 & &
\end{smallmatrix}\right),
\end{equation*}
and therefore $x(i(1, U(a))) = -a^{-1}$ and $j(i(1, U(a))) = 1.$ By Definition \ref{d:splittings0}, we have
\begin{equation*}
\hat s(1,U(a)) = \begin{cases}
1 & \text{if $a = 0$,} \\
\xi(-a^{-1}) \cdot (u,-1)_F \cdot \gamma_F(u, \tfrac{1}{2} \psi)^{-1} = (u, a)_F \cdot \gamma_F(u, \tfrac{1}{2} \psi)^{-1}, & \text{if $a \in F^\times.$}
\end{cases}
\end{equation*}
We next calculate $\blambda(1, U(a))$. Since $g = (1,U(a))$ stabilizes $\bY'{}^\square$, 
\begin{equation*}
\blambda(g) = \gamma_F(\tfrac{1}{2} \psi \circ q), \qquad q \colonequals q(\bV^\triangle, \bY'{}^\square, \bV^\triangle g).
\end{equation*}
Working in the $F$-basis given in \eqref{e:double F basis}
\begin{align*}
\bV^\triangle &= \{(0, 0, 0, 0, y_1, y_2, y_3, y_4)\}, \\
\bY'{}^\square &= \{(0, 0, y_1, y_2, y_3, y_4, 0, 0)\}, \\
\bV^\triangle g &= \{(0, 0, -a y_3, \tfrac{a}{u} y_4, y_1 + \tfrac{a}{2} y_3, y_2 - \tfrac{au}{2} y_4, y_3, y_4)\}.
\end{align*}
If $a = 0$, then $\blambda(1, U(0)) = 1,$ and the lemma holds. It remains to prove the assertion for when $a \in F^\times$. Then we have $\bR = \{(0, 0, 0, 0, *, *, 0, 0)\},$ $\bR^\perp = \{(0, 0, *, *, *, *, *, *)\}$ so that
\begin{align*}
(\bV^\triangle)_\bR &= \{(0, 0, 0, 0, 0, 0, y_1, y_2)\}, \\
(\bY'{}^\square)_\bR &= \{(0, 0, y_1, y_2, 0, 0, 0, 0)\}, \\
(\bV^\triangle g)_\bR &= \{(0, 0, -ay_1, \tfrac{a}{u}y_2, 0, 0, y_1, y_2)\}.
\end{align*}
It is clear from the above equations that
\begin{equation*}
(\bY'{}^\square)_\bR \left(\begin{smallmatrix} 1 & b \\ 0 & 1 \end{smallmatrix}\right) = (\bV^\triangle g)_\bR, \qquad \text{where $\left(\begin{smallmatrix}
1 & b \\ 0 & 1
\end{smallmatrix}\right) \in P_{(\bV^\triangle)_\bR} \subset \Sp(\bR^\perp/\bR),$ for $b = \left(\begin{smallmatrix} -\tfrac{1}{a} & 0 \\ 0 & \tfrac{u}{a} \end{smallmatrix}\right)$.} 
\end{equation*}
By definition, $q = (\bY'{}^\square)_\bR$ with the symmetric bilinear form given by
\begin{equation*}
q((x_1, x_2), (y_1, y_2)) = -\tfrac{1}{a} x_1 y_1 + \tfrac{u}{a} x_2 y_2.
\end{equation*}
Therefore we have
\begin{equation*}
\dim q = 2, \qquad \det q = -\tfrac{u}{a^2}, \qquad h_F(q) = (-\tfrac{1}{a}, \tfrac{u}{a})_F.
\end{equation*}
Observe that $(-\tfrac{1}{a}, \tfrac{u}{a})_F = (-a, au)_F (-a, a)_F = (-a, u)_F$, and so
\begin{equation*}
\blambda(1, U(a)) 
= \gamma_F(\tfrac{1}{2} \psi)^2 \cdot \gamma_F(-\tfrac{u}{a^2}, \tfrac{1}{2} \psi) \cdot (-\tfrac{1}{a}, \tfrac{u}{a})_F 
= \gamma_F(-1, \tfrac{1}{2} \psi)^{-1} \cdot \gamma_F(-u, \tfrac{1}{2} \psi) \cdot (-a, u)_F.
\end{equation*}
Finally, we have
\begin{equation*}
\bs(1, U(a))
= (u, a)_F \cdot \gamma_F(u, \tfrac{1}{2} \psi)^{-1} \cdot \gamma_F(-1, \tfrac{1}{2} \psi)^{-1} \cdot \gamma_F(-u, \tfrac{1}{2} \psi) \cdot (-a, u)_F = 1. \qedhere
\end{equation*}
\end{proof}

\begin{lemma}\label{l:bs W}
We have
\begin{equation*}
\bs(1, W) = (u, -1)_F \cdot \gamma_F(u, \tfrac{1}{2} \psi).
\end{equation*}
In particular, if $\ord(u)$ is even, then
\begin{equation*}
\bs(1, W) = 1.
\end{equation*}
\end{lemma}

\begin{proof}
The matrix $W$ sends $w_1 \mapsto -w_1$ and $w_2 \mapsto -w_2$. Recalling that $i(W)$ acts linearly on $W_0$ and trivially on $W_0^-$, it is a straightforward computation to see that
\begin{equation*}
i(1, W) = \left(\begin{smallmatrix}
\tfrac{1}{2} & -\tfrac{1}{2} & \tfrac{1}{4} & -\tfrac{1}{4} \\
\tfrac{1}{2} & \tfrac{1}{2} & \tfrac{1}{4} & \tfrac{1}{4} \\
-1 & 1 & \tfrac{1}{2} & -\tfrac{1}{2} \\
-1 & -1 & \tfrac{1}{2} & \tfrac{1}{2}
\end{smallmatrix}\right) = 
\left(\begin{smallmatrix}
1 & -1 & -\tfrac{1}{4} & \tfrac{1}{4} \\
1 & 1 & -\tfrac{1}{4} & -\tfrac{1}{4} \\
0 & 0 & \tfrac{1}{2} & -\tfrac{1}{2} \\
0 & 0 & \tfrac{1}{2} & \tfrac{1}{2}
\end{smallmatrix}\right)
\left(\begin{smallmatrix}
& & 1 & \\
& & & 1 \\
-1 & & & \\
& -1 & &
\end{smallmatrix}\right)
\left(\begin{smallmatrix}
2 & 0 & -1 & 0 \\
0 & 2 & 0 & -1 \\
0 & 0 & \tfrac{1}{2} & 0 \\
0 & 0 & 0 & \tfrac{1}{2}
\end{smallmatrix}\right).
\end{equation*}
Therefore we have $x(i(1,W)) = \tfrac{1}{8}$ and $j(i(1,W)) = 2,$ and by Definition \ref{d:splittings0},
\begin{equation}\label{e:s hat W}
\hat s(1, W) = \xi(\tfrac{1}{8}) \cdot ((u, -1)_F \cdot \gamma_F(u, \tfrac{1}{2} \psi))^{-2} = (u, -2)_F.
\end{equation}
We next calculate $\blambda(1, W)$. With respect to the symplectic basis given in \eqref{e:double F basis}, the image of $i(1,W)$ in $\Sp(\bV^\square)$ is
\begin{equation*}
g = \left(\begin{smallmatrix}
\tfrac{1}{2} & & -\tfrac{1}{2} & & \tfrac{1}{4} & & -\tfrac{1}{4} & \\
& \tfrac{1}{2} & & \tfrac{1}{2} & & -\tfrac{u}{4} & & \tfrac{u}{4} \\
\tfrac{1}{2} & & \tfrac{1}{2} & & \tfrac{1}{4} & & \tfrac{1}{4} & \\
& \tfrac{1}{2} & & \tfrac{1}{2} & & -\tfrac{u}{4} & & -\tfrac{u}{4} \\
-1 & & 1 & & \tfrac{1}{2} & & -\tfrac{1}{2} & \\
& u & & -u & & \tfrac{1}{2} & & -\tfrac{1}{2} \\
-1 & & -1 & & \tfrac{1}{2} & & \tfrac{1}{2} & \\
& u & & u & & \tfrac{1}{2} & & \tfrac{1}{2}
\end{smallmatrix}\right) \in \Sp(\bV^\square).
\end{equation*}
By definition,
\begin{equation*}
\blambda(g) = \gamma_F(\tfrac{1}{2} \psi \circ q(\bV^\triangle, \bY'{}^\square g^{-1}, \bY'{}^\square)) \cdot \gamma_F(\tfrac{1}{2} \psi \circ q(\bV^\triangle, \bY'{}^\square, \bV^\triangle g)).
\end{equation*}
We have
\begin{align*}
\bV^\triangle &= \{(0, 0, 0, 0, y_1, y_2, y_3, y_4)\}, \\
\bY'{}^\square g^{-1} &= \{(y_1, y_2, y_3, y_4, \tfrac{1}{2} y_3, -\tfrac{1}{2u} y_4, \tfrac{1}{2} y_1, -\tfrac{1}{2u} y_2)\}, \\
\bY'{}^\square &= \{(0, 0, y_1, y_2, y_3, y_4, 0, 0)\},
\end{align*}
which implies that $\bR = \{(0, 0, *, *, *, *, 0, 0)\}$ and hence
\begin{equation}\label{e:lambda 1 W}
\gamma_F(\tfrac{1}{2} \psi \circ q(\bV^\triangle, \bY'{}^\square g^{-1}, \bY'{}^\square)) = 1.
\end{equation}
Now we calculate the second factor of $\blambda(1, W)$. We have
\begin{align*}
\bV^\triangle &= \{(0, 0, 0, 0, y_1, y_2, y_3, y_4)\}, \\
\bY'{}^\square &= \{(0, 0, y_1, y_2, y_3, y_4, 0, 0)\}, \\
\bV^\triangle g &= \{(y_1, y_2, y_3, y_4, -\tfrac{1}{2} y_1, \tfrac{1}{2u} y_2, -\tfrac{1}{2} y_3, \tfrac{1}{2u} y_4)\},
\end{align*}
and hence $\bR = \{(0, 0, 0, 0, *, *, 0, 0)\},$ $\bR^\perp = \{(0, 0, *, *, *, *, *, *)\}.$ This implies that
\begin{align*}
(\bV^\triangle)_\bR &= \{(0, 0, 0, 0, 0, 0, y_1, y_2)\}, \\
(\bY'{}^\square)_\bR &= \{(0, 0, y_1, y_2, 0, 0, 0, 0)\}, \\
(\bV^\triangle g)_\bR &= \{(0, 0, y_1, y_2, 0, 0, -\tfrac{1}{2} y_1, \tfrac{1}{2u} y_2)\},
\end{align*}
and we have
\begin{equation*}
(\bY'{}^\square)_\bR \left(\begin{smallmatrix} 1 & b \\ 0 & 1 \end{smallmatrix}\right) = (\bV^\triangle g)_\bR, \qquad \text{for $b = \left(\begin{smallmatrix} -\tfrac{1}{2} & \\ & \tfrac{1}{2u} \end{smallmatrix}\right)$.}
\end{equation*}
It follows that
\begin{equation}\label{e:lambda 2 W}
\gamma_F(\tfrac{1}{2} \psi \circ q(\bV^\triangle, \bY'{}^\square, \bV^\triangle g)) 
= \gamma_F(\tfrac{1}{2} \psi)^2 \cdot \gamma_F(-\tfrac{1}{4u}, \tfrac{1}{2} \psi) \cdot (-\tfrac{1}{2}, \tfrac{1}{2u})_F 
= \gamma_F(u, \tfrac{1}{2} \psi) \cdot (2, u)_F.
\end{equation}
Putting together Equations \eqref{e:s hat W}, \eqref{e:lambda 1 W}, and \eqref{e:lambda 2 W}, we have
\begin{equation*}
\bs(1, W) 
= \hat s(1, W) \cdot \blambda(1, W) 
= (u, -2)_F \cdot \gamma_F(u, \tfrac{1}{2} \psi) \cdot (u, 2)_F 
= (u, -1)_F \cdot \gamma_F(u, \tfrac{1}{2} \psi). 
\end{equation*}
To see the final assertion, first observe that if $\ord(u)$ is even, then either $E$ is split or unramified over $F$. In either case, $(u, -1)_F = 1$. By \cite[Proposition A.11]{R93}, $\ord(u)$ even implies that $\gamma_F(u, \tfrac{1}{2} \psi) = 1$.
\end{proof}

\begin{lemma}
Let $a \in F$. Then
\begin{equation*}
\bs(1, D(-1)) \bs(1, W) \bs(1, U(a)) \bs(1, W) = 1.
\end{equation*}
\end{lemma}

\begin{proof}
We have $\bs(1, D(-1)) \bs(1, W) \bs(1, U(a)) \bs(1, W)
= (u, -1)_F ((u, -1)_F \gamma_F(u, \tfrac{1}{2} \psi))^2 = 1.$ 
\end{proof}

\begin{lemma}
If $F = \RR$ and $E = \CC$, then for any $(\alpha, g) \in R$,
\begin{equation*}
\bs(\alpha, g) = \xi(\alpha^{-1}).
\end{equation*}
\end{lemma}

\begin{proof}
Since $(\alpha, D(1, \Nm(\alpha)))$ stabilizes $\bY'$, 
\begin{equation*}
\bs(\alpha, g) = \bs(\alpha, D(1, \Nm(\alpha))) \cdot \bs(1, D(1, \Nm(\alpha)^{-1}) g).
\end{equation*}
By Lemma \ref{l:bs D}, to prove the desired assertion, it remains to show that $\bs(1, g) = 1$ for $g \in \SL_2(\RR)$. But this follows from \cite[Proposition A.10(1)]{R93}. 
\end{proof}


\subsubsection{A splitting $\bs'$ of $z_{\bY'}$}

As in the previous subsection, write $D(a,d) \colonequals \diag(a,d)$.

\begin{lemma}\label{l:bs' D}\mbox{}
\begin{enumerate}[label=(\roman*)]
\item
If $(D(a,d), \alpha) \in \G(\GL_2(F) \times E^\times)$, then $\bs'(D(a,d), \alpha) = \xi'(-(a^{-1} \alpha - 1)(d^{-1} \alpha - 1)).$ In particular, we have $\bs'(D(a,a^{-1}), 1) = (u, a)_F$ and 
$\bs'(D(1,\Nm(\alpha)), \alpha) = \xi'(\alpha).$

\item
For $a \in F$, we have $\bs'(1, U(a)) = 1$.

\item
We have $\bs'(1, W) = (u, -1)_F \cdot \gamma_F(u, \tfrac{1}{2} \psi)$. In particular, if $\ord(u) \in 2\bZ$, then $\bs'(1, W) = 1$.

\item
For $a \in F$, we have $\bs'(1, D(-1)) \bs'(1, W) \bs'(1, U(a)) \bs'(1, W) = 1$.

\item
If $F = \RR$ and $E = \CC$, then $\bs'(\alpha, g) = \xi'(\alpha)$.
\end{enumerate}
\end{lemma}

\begin{proof}
The proof of (i) is similar to Lemma \ref{l:bs D} except that $(D(a,d), \alpha)$ sends $w_1 \mapsto a^{-1} \alpha w_1$ and $w_2 \mapsto d^{-1} \alpha w_2$. Thus the image of $(D(a,d), \alpha)$ in $\U_E(\Res V + \Res V^-)$ with respect to the basis $\bw_1, \bw_2, \bw_1^*, \bw_2^*$ is
\begin{equation*}
\left(
\begin{matrix}
\frac{a^{-1} \alpha + 1}{2} & 0 & 0 & \frac{a^{-1} \alpha - 1}{4} \\
0 & \frac{d^{-1} \alpha + 1}{2} & - \frac{d^{-1} \alpha - 1}{4} & 0 \\
0 & -(d^{-1} \alpha - 1) & \frac{d^{-1} \alpha + 1}{2} & 0 \\
-(a^{-1} \alpha - 1) & 0 & 0 & \frac{a^{-1} \alpha + 1}{2}
\end{matrix}
\right).
\end{equation*}
To be more precise, this proof is the proof of Lemma \ref{l:bs D} except with $a$ replaced by $a^{-1}$, $b$ replaced by $b^{-1}$, and $\alpha^{-1}$ replaced by $\alpha$. The proofs of the remaining parts are exactly the same as that of the analogous statements for $\bs$.
\end{proof}

\section{Global theta lifts}\label{ch:global theta}

In this section, we examine the global theta lifts in the similitude seesaw \eqref{e:main seesaw} in comparison to automorphic induction. Let $\chi$ be a Hecke character and recall that its automorphic induction $\pi_\chi$ to $\GL_2(\bA_F)$ has a Jacquet--Langlands transfer to $B^\times$ if and only if the following condition holds:
\begin{center}
If $B_v$ is ramified, then $\chi_v$ does not factor through $\Nm \from E_v^\times \to F_v^\times$.
\end{center}
We write $\pi_\chi^B$ to denote the Jacquet--Langlands transfer to $B^\times$ if the pair $(B, \chi)$ satisfies the above condition, and we set $\pi_\chi^B = 0$ otherwise. The main theorem of this section is:

\begin{theorem}\label{t:theta aut ind}
The theta lifts $\Theta(\chi \cdot \xi)$ from $\GU_B(V)$ to $\GU_B(W^*) \cong B^\times$ and $\Theta'(\overline{\chi' \cdot \xi'{}^{-1}})$ from $\GU_E(W)$ to $\GU_E(\Res V) \cong (E^\times \times (B')^\times)/F^\times$ can be described in terms of automorphic induction and the Jacquet--Langlands correspondence:
\begin{equation*}
\Theta(\chi \cdot \xi) \cong \pi_{\chi}^B, \qquad \text{and} \qquad \Theta'(\overline{\chi' \cdot  \xi'{}^{-1}})^\vee \cong \pi_{\chi'}^{B'} \otimes (\chi'{}^{-1} \cdot \xi'),
\end{equation*}
where the right-hand side is viewed as a representation of $\GU_E(\Res V)$ descended from $(B_\bA')^\times \times \bA_E^\times$.
\end{theorem}

To prove Theorem \ref{t:theta aut ind}, we will need two arguments. 
\begin{enumerate}[label=(\arabic*)]
\item
If $\Theta(\chi \cdot \xi) = 0$, then $\pi_\chi^B = 0$.

\item
If $\Theta(\chi \cdot \xi) \neq 0$, then $\Theta(\chi \cdot \xi) \cong \pi_\chi^B$.
\end{enumerate}
To prove (1), we will need to make use of the theory of doubling zeta integrals. Since the nonvanishing of the global theta lift $\Theta(\chi \cdot \xi)$ is determined by the nonvanishing of local doubling zeta integrals (Section \ref{s:rallis}), the crux of (1) is to establish the local zeta integral is vanishing only if the local theta lift is. To prove (2), we will need to calculate the local theta lift from $\GU(1)_v$ to $\GU(2)_v$ at all places where $\GU(2)_v \cong \GU(1,1)_v$. After showing that $\Theta(\chi \cdot \xi)$ must be cuspidal if it is nonzero, we apply Jacquet--Langlands \cite{JL} to conclude.

\subsection{Theta lifts with similitudes}\label{s:theta sim}

We first recall some general properties of Weil representations. Denote by $\omega_\psi$ and $\omega_\psi^\square$ the Weil representations of $\Mp(\bV)$ on $\cS(\bX)$ and of $\Mp(\bV^\square)$ on $\cS(\bX^\square) = \cS(\bX) \otimes \cS(\bX)$. We have a natural map
\begin{equation*}
\tilde \iota \from \Mp(\bV) \times \Mp(\bV) \to \Mp(\bV^\square)
\end{equation*}
inducing $(z_1, z_2) \mapsto z_1 \overline z_2$ on $\CC^1$, and $\omega_\psi$, $\omega_\psi^\square$ enjoy the following compatibility:
\begin{equation*}
\omega_\psi^\square \circ \tilde \iota \cong \omega_\psi \otimes (\omega_\psi \circ \tilde{\mathfrak j}_\bY),
\end{equation*}
where $\tilde{\mathfrak j}_\bY$ is the automorphism of $\Mp(\bV)_\bY = \Sp(\bV) \times \CC^1$ defined by
\begin{equation*}
\tilde{\mathfrak j}_\bY(g,z) = (\mathfrak j_\bY(g), z^{-1}), \qquad \mathfrak j_\bY(g) = d_\bY(-1) \cdot g \cdot d_\bY(-1).
\end{equation*}

We make the following definitions:
\begin{align*}
G^\square &\colonequals \GU_B(V^\square) &
G^\square{}' &\colonequals \GU_E(W^\square) \\
G &\colonequals \GU_B(V)^\circ \cong E^\times \cong \GU_E(V_0) & G' &\colonequals \GU_E(W) \\
H &\colonequals \GU_B(W^*) \cong B^\times \subset \GU_E(W_0) &
H' &\colonequals \GU_E(\Res V) \cong ((B')^\times \times E^\times)/F^\times
\end{align*}
Recall that these groups fit into the following seesaws:
\begin{equation}\label{e:all seesaws}
\begin{tikzcd}[row sep=small]
H' \ar[dash]{d} \ar[dash]{dr} & H \ar[dash]{d} \ar[dash]{dl} \\
G & G'
\end{tikzcd} \quad
\begin{tikzcd}[row sep=small]
G^\square \ar[dash]{d} \ar[dash]{dr} & H \times H \ar[dash]{d} \ar[dash]{dl} \\
G \times G & H
\end{tikzcd} \quad
\begin{tikzcd}[row sep=small]
H' \times H' \ar[dash]{d} \ar[dash]{dr} & G^\square{}' \ar[dash]{d} \ar[dash]{dl} \\
H' & G' \times G'
\end{tikzcd}
\end{equation}
Adding a subscript $1$ to any of the above groups indicates that we take the kernel of the similitude character. If $G^{(1)}, \ldots, G^{(n)}$ is a collection of unitary similitude groups, we define
\begin{equation*}
\cG_{G^{(1)} \times \cdots \times G^{(n)}} \colonequals \{(g_1, \ldots, g_n) \in G^{(1)} \times \cdots \times G^{(n)} : \nu(g_1) = \cdots = \nu(g_n)\}.
\end{equation*}

We also define $Z \colonequals F^\times$ and $\cC\colonequals (\bA^\times)^2 (F^\times)^+ \quotient (\bA^\times)^+$, where
\begin{align*}
(\bA^\times)^+ &\colonequals \nu(G(\bA)) \cap \nu(H(\bA)) = \nu(G'(\bA)) \cap \nu(H'(\bA)) = \Nm_{E/F}(\bA_E^\times), \\
(F^\times)^+ &\colonequals F^\times \cap (\bA^\times)^+.
\end{align*}
Adding a superscript $+$ to any of the groups $G, H, G', H'$ means we take the preimage of $(\bA^\times)^+$ (or $(F^\times)^+$, etc.) under the similitude map. 

Fix sections 
\begin{equation*}
\cC \to G(\bA)^+, \qquad \cC \to H(\bA)^+, \qquad \cC \to G'(\bA)^+, \qquad \cC \to H'(\bA)^+,
\end{equation*}
of the natural surjections induced by the similitude character. We write $g_c, h_c, g_c', h_c'$ for the images of $c \in \cC$ under these sections. The following lemma is straightforward:

\begin{lemma}\label{l:index 2}
The similitude character induces isomorphisms
\begin{align*}
Z(\bA) G_1(\bA) G(F)^+ \quotient G(\bA)^+ &\cong \cC, & Z(\bA) H_1(\bA) H(F)^+ \quotient H(\bA)^+ &\cong \cC, \\
Z(\bA) G_1'(\bA) G'(F)^+ \quotient G'(\bA)^+ &\cong \cC, & Z(\bA) H_1'(\bA) H'(F)^+ \quotient H'(\bA)^+ &\cong \cC.
\end{align*}
and
\begin{align*}
H(\bA)/(H(F) H(\bA)^+) &\cong H'(\bA)/(H'(F) H'(\bA)^+) \cong \Gal(E/F), \\
G^\square(\bA)/(G^\square(F) G^\square(\bA)^+) &\cong G^\square{}'(\bA)/(G^\square{}'(F) G^\square{}'(\bA)^+) \cong \Gal(E/F).
\end{align*}
\end{lemma}

Recall that in Section \ref{ch:splittings} (see Definitions \ref{d:splittings0} and \ref{d:splittings'}), for each place $v$ of $F$, we defined splittings of $z_{\bY_v}$ and $z_{\bY_v^\square}$ on certain unitary groups. Recall also that the discussion in Section \ref{s:global splittings} allowed us to multiply the local splittings to obtain global splittings of $z_\bY$
\begin{equation*}
s \from \cG_{G \times H}(\bA) \to \CC^1, \qquad
s' \from \cG_{H' \times G'}(\bA) \to \CC^1,
\end{equation*}
and global splittings of $z_{\bY^\square}$ 
\begin{equation*}
s^\square \from \cG_{G^\square \times H}(\bA) \to \CC^1, \qquad
s^\square{}' \from \cG_{H' \times G^\square{}'}(\bA) \to \CC^1.
\end{equation*}
These allow us to define corresponding Weil representations $\omega_\psi, \omega_\psi', \omega_\psi^\square, \omega_\psi^\square{}'$. By Proposition \ref{p:global compat},
\begin{align}
\label{e:doubled weil0}
\omega_\psi^\square(g_1, g_2, h) &= \omega_\psi(g_1, h) \otimes \xi(\det(g_2,h)) \overline{\omega_\psi(g_2,h)}, && (g_1, g_2, h) \in \cG_{G \times G \times H}(\bA), \\
\label{e:doubled weil'}
\omega_\psi^\square{}'(h, g_1, g_2) &= \omega_\psi'(h,g_1) \otimes \xi'(\det(h, g_2)) \overline{\omega_\psi'(h,g_2)}, && (h, g_1, g_2) \in \cG_{H' \times G' \times G'}(\bA), \\
\label{e:weil compat}
\omega_\psi(g,g') &= \xi(g) \xi'(g') \omega_\psi'(g,g'), && (g,g') \in \cG_{G \times G'}(\bA).
\end{align}

Define a theta distribution
\begin{equation*}
\Theta \from \cS(\bX(\bA)) \to \CC, \qquad \varphi \mapsto \sum_{x \in \bX(F)} \varphi(x).
\end{equation*}
Let $\varphi \in \cS(\bX(\bA))$ and let $\chi$ be a Hecke character. For $h = h_1 h_c \in H(\bA)^+$ where $h_1 \in H_1(\bA)$, define
\begin{equation*}
\theta_\varphi(\chi)(h) \colonequals \int_{G_1(F) \quotient G_1(\bA)} \Theta(\omega_\psi(g_1g_c,h) \varphi) \chi(g_1g_c) \, dg_1.
\end{equation*}
Here, $dg = \prod_v dg_{1,v}$ is the Tamagawa measure on $G_1(\bA)$. Note that $\theta_\varphi(\chi)(\gamma h) = \theta_\varphi(f)(\gamma h)$ for $\gamma \in H(F) \cap H(\bA)^+$ and $h \in H(\bA)^+$. By declaring 
\begin{equation*}
\theta_\varphi(\chi)(\gamma h) = \theta_\varphi(\chi)(h), \qquad \text{for all $\gamma \in H(F)$ and $h \in H(\bA)^+$},
\end{equation*}
we obtain an automorphic form on the subgroup $H(F) H(\bA)^+$ of $H(\bA)$. Let $\varphi \in \cS(\bX(\bA))$ and let $\chi'$ be a Hecke character. For $h' = h_1' h_c' \in H'(\bA)^+$ where $h_1' \in H_1'(\bA)$, define
\begin{equation*}
\theta_\varphi'(\chi')(h') \colonequals \int_{G_1'(F) \quotient G_1'(\bA)} \overline{\Theta(\omega_\psi'(h',g_1'g_c') \varphi)} \chi'(g_1'g_c') \, dg_1'.
\end{equation*}
Here, $dg_1' = \prod_v dg_{1,v}'$ is the Tamagawa measure on $G_1'(\bA).$ 

Let $\Theta_+(\chi)$ be the automorphic representation of $H(F) H(\bA)^+$ generated by $\theta_\varphi(\chi)$ for $\varphi \in \cS(\bX(\bA))$ and let $\Theta_+'(\chi')$ be the automorphic representation of $H'(F)H'(\bA)^+$ generated by $\theta_\varphi'(\chi')$ for all $\varphi \in \cS(\bX(\bA))$. Define
\begin{equation*}
\Theta(\chi) \colonequals \Ind_{H(F)H(\bA)^+}^{H(\bA)}\left(\Theta_+(\chi)\right), \qquad \Theta'(\chi') \colonequals \Ind_{H'(F)H'(\bA)^+}^{H'(\bA)}\left(\Theta_+'(\chi')\right).
\end{equation*}
By Lemma \ref{l:index 2}, $[H(\bA): H(F) H(\bA)^+] = 2$, so $\theta_\varphi(\chi)$ extends to an automorphic form in $\Theta(\chi)$ via
\begin{equation*}
\theta_\varphi(\chi)(h) \colonequals 
\begin{cases}
\theta_\varphi(\chi)(h_+) & \text{if $h = \gamma h_+$ for $\gamma \in H(F)$ and $h_+ \in H(\bA)^+$,} \\
0 & \text{otherwise.}
\end{cases}
\end{equation*}
Similarly, $\theta_\varphi'(\chi')$ extends to an automorphic form in $\Theta'(\chi')$ by setting
\begin{equation*}
\theta_\varphi'(\chi')(h') \colonequals 
\begin{cases}
\theta_\varphi'(\chi')(h_+') & \text{if $h' = \gamma h_+'$ for $\gamma \in H'(F)$ and $h_+' \in H'(\bA)^+$,} \\
0 & \text{otherwise.}
\end{cases}
\end{equation*}
The theta lifts for $\omega_\psi^\square$ and $\omega_\psi^\square{}'$ are defined analogously.

\subsection{The Rallis inner product formula}\label{s:rallis}

In this section, we will write down an equation relating the Petersson inner product of a theta lift to a theta lift to a doubled unitary similitude group. To this end, we will use the doubled seesaws in \eqref{e:double seesaw}, \eqref{e:all seesaws}.

For automorphic forms $f_1, f_2$ on $H(\bA) \cong B_\bA^\times$ and $f_1', f_2'$ on $H'(\bA) \cong (B_\bA'{}^\times \times \bA_E^\times)/\bA_F^\times$, define
\begin{equation*}
\langle f_1, f_2 \rangle_H \colonequals \int_{[H]} f_1(h) \cdot \overline{f_2(h)} \, dh, \qquad 
\langle f_1', f_2' \rangle_{H'} \colonequals \int_{[H']} f_1'(h') \cdot \overline{f_2'(h')} \, dh,
\end{equation*}
where $dh = \prod_v dh_v$ and $dh' = \prod_v dh_v'$ are the Tamagawa measures of $H(\bA)$ and $H'(\bA)$.

Recall from Proposition \ref{p:global compat} that the splittings $s \from \cG_{G \times H}(\bA) \to \CC^1$ and $s^\square \from \cG_{G^\square \times H}(\bA) \to \CC^1$ enjoy the property that for $(g_1, g_2, h) \in \cG_{G \times G \times H}$,
\begin{equation*}
s^\square(g_1,g_2,h) = s(g_1,h) \cdot \overline{s(g_2, h)} \cdot \xi(\det(i(g_2,h))).
\end{equation*}
This compatibility implies that for any $h_1 \in H_1$, $g_1, g_1' \in G_1$, and $(g_c, h_c) \in \cG_{G \times H}(\bA)$,
\begin{align*}
\Theta(\omega_\psi(g_1g_c, &h_1h_c) \varphi_1) \cdot \overline{\Theta(\omega_\psi(g_1'g_c,h_1h_c)\varphi_2)} \\
&= \Theta(\omega_\psi^\square((g_1g_c, g_1'g_c),h_1h_c) \varphi_1 \otimes \overline \varphi_2) \cdot \xi(\det(h_1h_c))^{-1} \cdot \xi(g_1' g_c)^2.
\end{align*}
Hence for $\varphi_1, \varphi_2 \in \cS(\bX(\bA))$ and Hecke characters $\chi_1, \chi_2$ of $E^\times$, by formally switching the integrals at the equality, we have
\begin{align}
\nonumber
\langle \theta_{\varphi_1}(\chi_1 \cdot \xi)&, \theta_{\varphi_2}(\chi_2 \cdot \xi) \rangle_H 
\\ \nonumber
&= \int_{\cC} \int_{[H_1]} \theta_{\varphi_1}(\chi_1 \cdot \xi)(h_1 h_c) \cdot \overline{\theta_{\varphi_2}(\chi_2 \cdot \xi)(h_1 h_c)} \, dh_1 \, dc 
\\ \nonumber
&= \int_{\cC} \int_{[H_1]} \int_{[G_1]} \int_{[G_1]} \Theta(\omega_\psi(g_1g_c, h_1h_c) \varphi_1) (\chi_1 \xi)(g_1g_c) \cdot 
\\ \nonumber
&\qquad\qquad\qquad \overline{\Theta(\omega_\psi(g_1'g_c,h_1h_c)\varphi_2) (\chi_2\xi)(g_1'g_c)} \, dg_1 \, dg_1' \, dh \, dc 
\\ \label{e:formal outer}
&= \int_{\cC} \int_{[G_1]} \int_{[G_1]} (\chi_1 \xi)(g_c g_c) \cdot (\overline \chi_2 \xi)(g_1' g_c) \cdot 
\\ \label{e:formal double}
&\qquad \int_{[H_1]} \Theta(\omega_\psi^\square((g_1 g_c, g_1'g_c), h_1 h_c)(\varphi_1 \otimes \overline \varphi_2)) \cdot \xi(\det(h_1 h_c))^{-1}  \, dh_1 \, dg_1 \, dg_1' \, dc.
\end{align}
The inner integral in Equation \eqref{e:formal double} is the theta lift of $\xi(\det)^{-1}$ to $\GU_B(V^\square)$, but to make actual sense of the above, one must be careful about convergence issues. In the case that $B$ is division, the quotient $B^\times \backslash B_\bA^\times$ is compact, and therefore the integral in \eqref{e:formal double} is absolutely convergent. Hence the formal manipulation above is completely justified. We can then use the Siegel--Weil formula together with the theory of doubling integrals \cite{PSR87} to obtain a Rallis inner product formula. In the case that $B$ is split (i.e.\ $B \cong M_2(F)$), \eqref{e:formal double} does not converge absolutely in general, so the last equality does not make sense. In this case, we use the regularized Siegel--Weil formula of \cite{GQT}.
 
\subsubsection{The Siegel--Weil formula for division quaternion algebras}

In this section, we explain how to obtain a Rallis inner product formula in the case that $B$ is division. For $\varphi \in \cS(\bX^\bigtriangledown(\bA))$, define
\begin{equation*}
E(g, \cF_\varphi) = \sum_{\gamma \in P(F) \backslash \U(1,1)} \cF_\varphi(\gamma g), \qquad \text{where $\cF_\varphi(g) \colonequals (\omega_\psi^\square(d(\nu(g)^{-1})g) \varphi)(0)$.}
\end{equation*}
This is the value of an Eisenstein series at $s = \tfrac{1}{2}$. In this case, the Siegel--Weil formula states that for $g, g' \in \GU(1)$ such that $\nu(g) = \nu(g')$,
\begin{equation*}
E(i(g, g'), \cF_\varphi) = \int_{[H_1]} \Theta(\omega_\psi^\square((g,g'),h)(\varphi_1 \otimes \overline \varphi_2)) \cdot \xi(\det(h))^{-1} \, dh
\end{equation*} 
where $i \from \G(\U(1) \times \U(1)) \to \U(1,1)$ and $\varphi \in \cS(\bV^\bigtriangledown(\bA))$ is the partial Fourier transform of $\varphi_1 \otimes \overline \varphi_2 \in \cS(\bX^\square(\bA))$. We now see that, continuing from \eqref{e:formal outer}, \eqref{e:formal double}, we have
\begin{equation*}
\langle \theta_{\varphi_1}(\chi_1 \cdot \xi), \theta_{\varphi_2}(\chi_2 \cdot \xi) \rangle_H = \int_{\cC} \int_{[G_1]} \int_{[G_1]} (\chi_1 \xi)(g_1 g_c) \cdot (\overline \chi_2 \xi)(g_1' g_c) \cdot E(i(g_1 g_c, g_1' g_c), \cF_\varphi) \, dg_1 \, dg_1' \, dc.
\end{equation*}
We have $\cF_\varphi(i(g_1g_c,g_1'g_c)) = \cF(i(g_1'{}^{-1} g_1, 1)) \overline\xi{}^2(g_1')$, and hence unfolding the above integral and making the substitution $g = g_1 g_c$, $g' = g_1'{}^{-1} g_1$ gives
\begin{equation*}
=  \int_{G_1(\bA)} \int_{[G]} (\chi_1 \xi)(gg') \cdot (\overline {\chi_2 \xi})(g) \cdot \cF_\varphi(i(g,1)) \, dg \, dg'.
\end{equation*}
The Tamagawa measure on $G_1(\bA)$ can be written as a product of local measures $dg_{1,v}$ on $G_{1,v}$ times a global factor $\rho_F/\rho_E$ (see Section \ref{s:tamagawa}). Hence if $\chi_1 = \chi_2 = \chi$ and $\varphi_1 = \varphi_2 = \phi = \otimes_v \phi_v$, we have
\begin{align*}
\langle \theta_\varphi(\chi \cdot \xi), \theta_\phi(\chi \cdot \xi) \rangle_H 
&= \int_{G_1(\bA)} \cF_\varphi(i(g,1)) \langle (\chi \xi)(g') (\chi\xi), (\chi\xi) \rangle_{[G]} \, dg' = \frac{\rho_F}{\rho_E} \cdot \prod_v Z(\tfrac{1}{2}, \cF_{\varphi_v}, \chi_v),
\end{align*}
where
\begin{equation}\label{e:local zeta nonsplit}
Z(\tfrac{1}{2}, \cF_{\varphi_v}, \chi_v) \colonequals \int_{G_{1,v}} \langle \omega_\psi(g_{1,v}) \phi, \phi \rangle \cdot (\chi_v \xi_v)(g_{1,v}) \, dg_{1,v}.
\end{equation}

\subsubsection{The regularized Siegel--Weil formula for $(E^\times, \GL(2))$}

In this section, we follow \cite{GQT} and describe how to make sense of \eqref{e:formal double} and obtain a Rallis inner product formula in the case that $B$ is split. We will need to translate between the quaternionic unitary groups $(\GU_B(V)^\circ, \GU_B(W^*)) \cong (E^\times, \GL_2(F))$ and the dual reductive pair $(\GO(2), \GSp(2)) \cong (E^\times, \GL_2(F))$. In the notation of \cite{GQT}, we have $n = m = 2$, $r = 1$, $\epsilon = 1$, which puts us in the second term range since $1 < 2 \leq 2 \cdot 1$. 

Recall that we have an embedding
\begin{equation*}
\G(\U_B(V)^\circ, \U_B(W^*)) \hookrightarrow \G(\U_E(V_0) \times \U_E(W_0)).
\end{equation*}
When $B$ is split, then there is a decomposition $W_0 = W_1 + W_2$ of the $E$-space $W_0$ into isotropic subspaces of dimension 1. Set
\begin{equation*}
\bX' = \Res_{E/F}(V_0 \otimes W_1), \qquad \bY' = \Res_{E/F}(V_0 \otimes W_2)
\end{equation*}
so that $\bV = \bX' + \bY'$ forms a complete polarization. In Section \ref{s:split splittings}, we explicated a splitting $\bs$ of $z_{\bY'}$. Comparing $\bs$ to the splitting 
\begin{equation*}
s_{(\O(2),\Sp(2))} \from \G(\O(2) \times \Sp(2))_\bA \to \CC^1
\end{equation*}
defined in \cite{K94}, we see that for $\alpha \in E^\times$, $a \in F^\times$, and $a' \in F$,
\begin{align*}
\bs(\alpha,d(\Nm(\alpha))) &= \xi(\alpha)^{-1} \cdot s_{(\O(2),\Sp(2))}(\alpha, d(\Nm(\alpha))), \\
\bs\left(1,\diag(a,a^{-1})\right) &= \xi(a)^{-1} \cdot s_{(\O(2),\Sp(2))}\left(1,\diag(a,a^{-1})\right), \\
\bs\left(1, \left(\begin{smallmatrix} 1 & a' \\ 0 & 1 \end{smallmatrix}\right)\right) &= s_{(\O(2),\Sp(2))}\left(1, \left(\begin{smallmatrix} 1 & a' \\ 0 & 1 \end{smallmatrix}\right)\right), \\
\bs\left(1, \left(\begin{smallmatrix} 0 & 1 \\ -1 & 0 \end{smallmatrix}\right)\right) &= s_{(\O(2),\Sp(2))}\left(1, \left(\begin{smallmatrix} 0 & 1 \\ -1 & 0 \end{smallmatrix}\right)\right).
\end{align*}
Now set $V_0^\bigtriangledown \colonequals \{(v,-v) : v \in V_0\}$ and $V_0^\triangle \colonequals \{(v,v) : v \in V_0\}$ so that
\begin{equation*}
\bV^\bigtriangledown = \Res_{E/F}(V_0^\bigtriangledown \otimes W_0), \qquad \bV^\triangle = \Res_{E/F}(V_0^\triangle \otimes W_0)
\end{equation*}
gives a complete polarization $\bV^\square = \bV^\bigtriangledown + \bV^\triangle$ of the doubled symplectic space. Let $\hat s_{(\O(2,2),\Sp(2))}$ denote the splitting of $z_{\bV^\triangle}$ defined in \cite{K94} and define
\begin{equation*}
s_{(\O(2,2), \Sp(2))}(h,g) \colonequals \hat s_{(\O(2,2), \Sp(2))}(h,g) \cdot \lambda_{\bY'{}^\square \rightsquigarrow \bV^\triangle}^{-1}(g,h) \qquad \text{for $(g,h) \in \G(\O(2,2),\Sp(2))$},
\end{equation*}
where $\lambda \colonequals \lambda_{\bY'{}^\square \rightsquigarrow \bV^\triangle}$ is the change-of-polarization function defined in Lemma \ref{l:change polarization}. Then using Proposition \ref{p:global compat}(a),
\begin{align}
\nonumber
\hat s&(g_1, g_2, h) \\
\nonumber
&= \bs^\square(g_1,g_2, 1) \cdot \lambda(g_1,g_2,h) \\
\nonumber
&= \bs(g_1, h) \cdot \overline{\bs(g_2, h)} \cdot \xi(\det(i(g_2,h))) \cdot \lambda(g_1,g_2,h) \\
\nonumber
&= s_{(\O(2),\Sp(2))}(g_1,h) \xi(g_1)^{-1} \cdot \overline{s_{(\O(2), \Sp(2))}(g_2,h)  \xi(g_2)^{-1}} \cdot \xi(g_2)^{-2} \xi(\det(h)) \cdot \lambda(g_1,g_2,h)  \\
\nonumber
&= s_{(\O(2),\Sp(2))}(g_1,h) \cdot \overline{s_{(\O(2), \Sp(2))}(g_2,h)} \cdot \xi(g_1)^{-1} \xi(g_2)^{-1} \xi(\det(h)) \cdot \lambda(g_1,g_2,h) \\
\nonumber
&= s_{(\O(2,2),\Sp(2))}(g_1,g_2,h) \cdot \xi(g_1)^{-1} \xi(g_2)^{-1} \cdot \lambda(g_1,g_2,h) \\
\label{e:O U double compat}
&= \hat s_{(\O(2,2),\Sp(2))}(g_1,g_2,h) \cdot \xi(g_1)^{-1} \xi(g_2)^{-1}.
\end{align}

Define $P_O \subset \GO(\Res_{E/F} V_0^\square) \cong \GO(2,2)$ to be the stabilizer of the totally isotropic subspace $\Res_{E/F} V_0^\triangle$ of $\Res_{E/F} V_0^\square$. For $\phi \in \cS(\bV^\bigtriangledown(\bA))$, define the Siegel--Weil sections
\begin{align*}
\Phi(\phi)(g) &\colonequals (\omega_\psi^\square(g) \phi)(0), && \text{for $g \in \GO(2,2)_\bA \subset \GU_E(V_0^\square)_\bA$}, \\
\Phi^{\O,\Sp}(\phi)(g) &\colonequals (\omega_\psi^{\O(2,2),\Sp(2)}(g) \phi)(0), && \text{for $g \in \GO(2,2)_\bA$}.
\end{align*}
Observe that $\Phi(\phi)(g) = \hat s(g) \cdot \hat s_{(\O(2,2),\Sp(2))}(g)^{-1} \cdot \Phi^{\O,\Sp}(\phi)(g)$. We make the analogous definitions for the local objects $\Phi_v(\phi_v)$ and $\Phi_v^{\O,\Sp}(\phi_v)$. The Siegel--Weil section $\Phi^{\O,\Sp}(\phi) \in \Ind_{P_O}^{\GO(2,2)}(\det) \cdot |\det|^{1/2}$ determines a standard section $\Phi^{\O,\Sp}_s(\phi) \in \Ind_{P_O}^{\GO(2,2)}(\det) \cdot |\det|^s$ and we may form the associated Eisenstein series
\begin{equation*}
E(s, \Phi^{\O,\Sp}(\phi))(g) \colonequals \sum_{\gamma \in P_O(F) \backslash \GO(2,2)} \Phi_s^{\O,\Sp}(\gamma g), \qquad \text{for $g \in \GO(2,2)_\bA$.}
\end{equation*}
Define
\begin{equation*}
Z(s, \Phi, \chi) \colonequals \int_{[\G(\O(2) \times \O(2))]} E(s, \Phi)(g_1,g_2) \cdot \chi(g_1) \cdot \overline \chi(g_2) \, dg_1 \, dg_2.
\end{equation*}
If $\Phi = \otimes_v \Phi_v$, define
\begin{equation*}
Z_v(s,\Phi_v,\chi_v) = \int_{E_v^1} \Phi_v(g_v,1) \cdot \chi_v(g_v)  \, dg_v.
\end{equation*}
By construction of the Tamagawa measure of $\bA_E^1$ (see Section \ref{s:tamagawa}), one has
\begin{equation*}
Z(s, \Phi, \chi) \colonequals \frac{\rho_F}{\rho_E} \cdot \prod_v Z_v(s, \Phi_v, \chi_v).
\end{equation*}

Define the partial Fourier transform $\delta \from \cS(\bX'{}^\square(\bA)) \to \cS(\bV^\bigtriangledown(\bA))$ by
\begin{equation*}
\delta(\varphi)(u) = \int_{((\bV^\triangle \cap \bY'{}^\square)\backslash\bV^\triangle)(\bA)} \varphi(x) \psi\left(\tfrac{1}{2}(\llangle x,y \rrangle - \llangle u,v \rrangle)\right) \, dv,
\end{equation*}
where we write $u + v = x + y$ with $u \in \bV^\bigtriangledown(\bA)$, $v \in \bV^\triangle(\bA)$, $x \in \bX'{}^\square(\bA)$, $y \in \bY'{}^\square(\bA)$, and $dv$ is the Tamagawa measure.

Observe that if $\phi \in \cS(\bV^\bigtriangledown(\bA))$ is the partial Fourier transform of $\varphi_1 \otimes \overline \varphi_2$ for $\varphi_1,\varphi_2 \in \cS(\bX'(\bA))$, then for the Siegel--Weil section $\Phi = \Phi^{\O,\Sp}(\delta(\varphi_1 \otimes \overline \varphi_2))$, we have
\begin{align}
\nonumber
Z_v(\tfrac{1}{2}, \Phi_v, \chi_v) 
&= \vol(E_v^1)  \int_{E_v^1} \Phi^{\O,\Sp}(\delta(\varphi_1 \otimes \overline \varphi_2)) (i(g_{1,v},1)) \cdot \chi_v(g_v) \, dg_v \\
\nonumber
&= \vol(E_v^1)  \int_{E_v^1} (\omega^{\O(2,2),\Sp(2)}_\psi(g_v,1) \delta(\varphi_1 \otimes \overline \varphi_2))(0) \cdot \chi_v(g_v) \, dg_v \\
\nonumber
&= \vol(E_v^1)  \int_{E_v^1} (\omega^\square_\psi(g_v,1) \delta(\varphi_1 \otimes \overline \varphi_2))(0) \cdot \chi_v(g_v) \cdot \xi_v(g_v) \, dg_v \\
\label{e:local zeta}
&= \vol(E_v^1)  \int_{E_v^1} \langle \omega_\psi(g_v) \varphi_1, \varphi_2 \rangle \cdot (\chi_v\xi_v)(g_v) \, dg_v
\end{align}

\begin{proposition}\label{p:split rallis}
For $\varphi_1, \varphi_2 \in \cS(\bX'(\bA))$, we have
\begin{equation*}
\langle \theta_{\varphi_1}(\chi \xi), \theta_{\varphi_2}(\chi \xi) \rangle 
=
\frac{\rho_F}{\rho_E} \cdot \prod_v Z_v(\tfrac{1}{2}, \Phi_v^{\O,\Sp}(\delta(\varphi_1 \otimes \overline \varphi_2)), \chi_v).
\end{equation*}
\end{proposition}

\begin{proof}
We use \eqref{e:O U double compat} to translate between our setting and that of \cite[Proposition 11.1]{GQT}. We have
\begin{align*}
\langle \theta_{\varphi_1}&(\chi \cdot \xi), \theta_{\varphi_2}(\chi \cdot \xi) \rangle_H 
\\ 
&= \int_{\cC} \int_{[H_1]} \theta_{\varphi_1}(\chi \cdot \xi)(h_1 h_c) \cdot \overline{\theta_{\varphi_2}(\chi \cdot \xi)(h_1 h_c)} \, dh_1 \, dc 
\\ 
&= \int_{\cC} \int_{[H_1]} \int_{[G_1]} \int_{[G_1]} \Theta(\omega_\psi(g_1g_c, h_1h_c) \varphi_1) (\chi \xi)(g_1g_c) \cdot \\ 
&\qquad\qquad\qquad \overline{\Theta(\omega_\psi(g_1'g_c,h_1h_c)\varphi_2) (\chi \xi)(g_1'g_c)} \, dg_1 \, dg_1' \, dh \, dc 
\\
&= \int_{\cC} \int_{[\Sp(2)]} \int_{[\O(2)]} \int_{[\O(2)]} \Theta(\omega_\psi^{\O,\Sp}(g_1g_c, h_1h_c) \varphi_1) (\chi \xi)(g_1g_c) \cdot \\ 
&\qquad\qquad\qquad \overline{\Theta(\omega_\psi^{\O,\Sp}(g_1'g_c,h_1h_c)\varphi_2) (\chi \xi)(g_1'g_c)} \cdot \xi^{-1}(g_1) \overline \xi^{-1}(g_1') \, dg_1 \, dg_1' \, dh \, dc 
\\ 
&= \Val_{s = 1/2} \int_{\cC} \int_{[\O(2)]} \int_{[\O(2)]} E(s, \Phi_{\O(2,2_,\Sp(2))}(\delta(\varphi_1 \otimes \overline \varphi_2)))(g_1 g_c, g_1' g_c) \cdot \\
&\qquad\qquad\qquad\qquad\qquad \chi(g_1 g_c) \cdot \overline \chi(g_1' g_c) \, dg_1 \, dg_1' \, dc
\\
&= \Val_{s=1/2} Z(s,\Phi(\delta(\varphi_1 \otimes \overline \varphi_2)), \chi). \qedhere
\end{align*}
\end{proof}

\subsection{Local doubling zeta integrals}\label{s:local zeta}

Let $v$ be a nonsplit place of $F$. For notational convenience, we drop all subscripts $v$ in this section. We preemptively note that the notation we use to describe the zeta integrals in this section differ from the notation used to describe the same (local) zeta integrals in the rest of the paper. We learned the proof of Proposition \ref{p:zeta vanishing} from A.\ Ichino. Similar arguments appear in \cite{GI14}.

Consider the Siegel parabolic subgroup
\begin{equation*}
P = \left\{\left(\begin{matrix} a & * \\ 0 & (a^*)^{-1} \end{matrix}\right) \in \GL_2(E)\right\} \subset \U(1,1),
\end{equation*}
and for any unitary character $\eta \from \U(1) \to \CC^1$, consider the functional
\begin{equation*}
Z(s, \eta, \xi^2) \from I(s, \xi^2) \to \CC, \qquad \cF \mapsto \int_{E^1} \cF(i(g,1)) \eta(g) \, dg,
\end{equation*}
where $\iota \from \U(1) \times \U(1) \to \U(1,1)$ is the natural map and
\begin{equation*}
I(s, \xi^2) 
\colonequals \Ind_P^{\U(1,1)}(\xi^2 \cdot | \cdot |^s) \colonequals \left\{\cF \from \U(1,1) \to \CC \; \Bigg| \;
\begin{gathered} \text{$\cF(pg) = \xi^2(a) |a|_E^{s + 1/2} \cF(g)$} \\ \text{for all $g \in \U(1,1)$ and $p = \left(\begin{smallmatrix} a & * \\ 0 & \overline a^{-1} \end{smallmatrix}\right) \in P$} \end{gathered} \right\}
\end{equation*}
is the normalized principal series representation. One has an intertwining operator
\begin{equation*}
M(s,\xi^2) \from I(s, \xi^2) \to I(-s, \overline \xi{}^{-2}) \cong I(-s, \xi^2)
\end{equation*}
given by 
\begin{equation*}
M(s, \xi^2) \cF(g) = \int_{N_P} \cF(w n g) \, dn,
\end{equation*}
where $w = \diag(1, -1)$ and $N_P$ is the unipotent radical of the parabolic $P$. 

Following Lapid--Rallis (see also Gan--Ichino, Section 10), after normalizing the intertwining operator by some rational function $c_\psi(s, \xi^2)$,
\begin{equation*}
M_\psi^{\textrm{LR}}(s, \xi^2) \colonequals c_\psi(s, \xi^2) M(s, \xi)
\end{equation*}
has a functional equation of the shape
\begin{equation}\label{e:zeta functional}
Z(-s, \eta, \xi^2)(M_\psi^{\textrm{LR}}(s,\xi^2) \cF) = * \cdot \gamma\left(s + \tfrac{1}{2}, \eta, \overline\xi, \psi\right) \cdot Z(s, \eta, \xi^2)(\cF),
\end{equation}
where $*$ denotes some nonzero factors. In particular, if we understand the behavior of the intertwining operator $M(s,\eta)$ and if $\gamma(s_0 + \tfrac{1}{2}, \eta) \neq 0$, the functional equation gives a relation between the nonvanishing of $Z(-s_0, \eta, \xi^2)$ and the nonvanishing of $Z(s_0, \eta, \xi^2)$.

We take a short detour to examine when the local theta lift to the nonsplit unitary group $\U(2)$ vanishes. Define
\begin{equation*}
V_n^+ \colonequals \bH_n, \qquad V_n^- \colonequals D \oplus \bH_{n-1},
\end{equation*}
where $\bH_n$ is the $2n$-dimensional split Hermitian $E$-space and $D$ is the nonsplit quaternion algebra over $F$ viewed as a 2-dimensional Hermitian $E$-space via $\langle x,y \rangle = \pr_E(x^* y)$. For a character $\eta \from \U(1) \cong E^1 \to \CC^1$, denote its theta lift to $\U(V_n^\pm)$ by $\Theta_{V_n^\pm}(\eta)$. To make tower ``compatible'' one takes the Weil representation for $\U(1) \times \U(V_n^+)$ to be such that the splitting on $\U(1)$ is given by $\xi.$ In particular, the Weil representation on $\U(1) \times \U(V_0^+) = \U(1) \times \{1\}$ is given by the one-dimensional representation $\xi.$ The \textit{first occurrence} of the theta lift in the towers $\{\U(V_n^+) : n \geq 0\}$, $\{\U(V_n^-) : n \geq 0\}$ is defined to be 
\begin{equation*}
n^+ = \min\{n : \Theta_{V_n^+}(\eta) \neq 0\}, \qquad n^- = \min\{n : \Theta_{V_n^-}(\eta) \neq 0\}.
\end{equation*}
The following result is a special case of a theorem of Sun--Zhu \cite{SZ15}:

\begin{theorem}[Sun--Zhu]
$n^+(\eta) + n^-(\eta) = 2$.
\end{theorem}

We can describe the first occurrence in this setting more explicitly. By the compatible choice of splittings in the tower of unitary groups $\U(V_n^+)$, we have that $\Theta_{V_0^+}(\chi \xi) \neq 0$ if and only if $\chi$ is the trivial character. Hence we must necessarily be in the setting $n^+(\chi \xi) + n^-(\chi \xi) = 0 + 2$, and in particular, $\Theta_{V_1^-}(\chi \xi) = 0$. 

Now suppose that $\chi$ is nontrivial. Then by the previous paragraph, $\Theta_{V_0^+}(\chi \xi) = 0$. We now argue that $\Theta_{V_1^+}(\chi \xi) \neq 0$. One explicit way to see this is as follows. Let $V_1^+ = V_1^\bigtriangledown + V_1^\triangle$ be a decomposition of $V_1^+$ into totally isotropic $E$-subspaces. For the Schwartz function $\varphi(x) = \chi(x) \ONE_{\cO_E^\times}(x) \in \cS(\Res_{E/F} V_1^\bigtriangledown)$, we have
\begin{equation*}
\int_{E^1} (\omega_\psi(g)\varphi)(0) \cdot (\chi \xi)(g) \, dg \neq 0,
\end{equation*}
which proves that there is a nontrivial $E^1$-equivariant map \begin{equation*}
(\cS(\Res_{E/F} V_1^\bigtriangledown), \omega_\psi) \to (\CC, \chi \xi).
\end{equation*}
Hence $\Theta_{V_1^+}(\chi \xi) \neq 0$ by definition of the local theta lift. This now implies that we must necessarily be in the setting $n^+(\chi \xi) + n^-(\chi \xi) = 1 + 1$, and $\Theta_{V_1^-}(\chi \xi) \neq 0$. 

In summary, the above arguments prove:

\begin{lemma}\label{l:lift to U2}\mbox{}
\begin{enumerate}[label=(\alph*)]
\item
$\Theta_{V_1^-}(\chi \xi) \neq 0$ if and only if $\chi \from E^1 \to \CC^\times$ is nontrivial.

\item
If $\chi \from E^1 \to \CC^1$ is nontrivial, $\Theta_{V_1^+}(\chi \xi) \neq 0$.
\end{enumerate}
\end{lemma}

We now discuss the relationship between the theory of the doubling zeta integral and the local theta correspondence. Consider the two doubling seesaws for $V_1^+$ and $V_1^-$:
\begin{equation*}
\begin{tikzcd}
\U(1,1) \ar[dash]{d} \ar[dash]{dr} & \U(V_1^\pm) \times \U(V_1^\pm) \ar[dash]{d} \ar[dash]{dl} \\
\U(1) \times \U(1) & \U(V_1^\pm)
\end{tikzcd}
\end{equation*}
If we have $\U(1,1) = \U(W)$, then one has a decomposition $W = W_1 + W_2$ of $W$ into 1-dimensional isotropic $E$-spaces, and hence by viewing $V_1^\pm$ as the $F$-space $\Res_{E/F}(W_1 \otimes_E V_1^\pm) = \Res_{E/F}(V_1^\pm),$ the Weil representation $\omega_\psi^\square$ for $\U(1,1) \times \U(V_1^\pm)$ can then be modeled on the space of Schwartz functions $\cS(V_1^\pm)$. Define
\begin{equation*}
\cS(V_1^\pm) \to I(\tfrac{1}{2}, \xi^2), \qquad \varphi \mapsto (g \mapsto (\omega_\psi^\square(i(g,1)) \varphi)(0)),
\end{equation*}
where $i \from \U(1) \times \U(1) \to \U(1,1)$ is the natural map. Let $R(V_1^\pm)$ denote the image of this map. Since $\xi^2|_{F^\times} = 1$, there is a unique one-dimensional representation $\widetilde \xi^2$ of $\U(1,1)$ extending the representation defined by $\left(\begin{smallmatrix} a & * \\ 0 & \overline a{}^{-1} \end{smallmatrix}\right) \mapsto \xi^2(a)$.  For the $0$-dimensional Hermitian space $V_0^+$, we define a map
\begin{equation*}
\cS(V_0^+) = \CC \to I(-\tfrac{1}{2}, \xi^2), \qquad z \mapsto (g \mapsto \widetilde \xi^2(g)).
\end{equation*}
Let $R(V_0^+)$ denote the image of this map. We say that $\Theta_{V_0^+}(\chi \xi) \neq 0$ if and only if $\Hom_{\U(1)}(\widetilde \xi^2, \chi \xi) \neq 0$. Since $\widetilde \xi^2$ is one-dimensional, we have $\Hom_{\U(1)}(\widetilde \xi^2, \chi \xi) \neq 0$ if and only if $Z(-\tfrac{1}{2}, \chi \xi, \xi^2)|_{R(V_0^+)} \neq 0$. Observe also that $\Theta_{V_0^+}(\chi \xi) \neq 0$ if and only if $\chi = 1$.

The goal of the remainder of this section is to prove the following:

\begin{proposition}\label{p:zeta vanishing}
Let $\xi \from \bA_E^\times \to \CC^1$ be such that $\xi|_{\bA_F^\times} = \epsilon_{E/F}$. Then 
\begin{equation*}
\Theta_{V_1^-}(\chi \xi) \neq 0 \qquad \Longrightarrow \qquad Z(\tfrac{1}{2}, \chi \xi, \xi^2)|_{R(V_1^-)} \neq 0.
\end{equation*}
\end{proposition}

We first remark that the converse of Proposition \ref{p:zeta vanishing} is true and straightforward to see: If $Z(\tfrac{1}{2}, \chi\xi, \xi^2)|_{R(V_1^-)} \neq 0$, then this immediately implies that $\Hom_{\U(1)}(\omega_\psi^\square|_{i(\U(1) \times \{1\})}, (\chi \xi)^{-1}) \neq 0$. But since $\omega_\psi^\square \cong \omega_\psi \otimes \overline \omega_\psi \xi^2$ (see Lemma \ref{l:double compat0}) as a representation of $\U(1) \times \U(1)$, we have $\Hom_{\U(1)}(\omega_\psi, (\chi \xi)^{-1}) \neq 0$, and so $\Theta_{V_1^-}(\chi \xi) \neq 0$ by definition.


Before we prove Proposition \ref{p:zeta vanishing}, we recall a special case of a theorem of Kudla--Sweet:.

\begin{theorem}[Kudla--Sweet, {\cite[Theorem 1.2(1),(4)]{KS97}}]\label{t:decomp U(1,1)} \mbox{}
\begin{enumerate}[label=(\roman*)]
\item
$R(V_0^+)$ is the unique irreducible submodule of $I(-\tfrac{1}{2}, \xi^2)$.
\item
$I(-\tfrac{1}{2}, \xi^2)/R(0, \xi^2)$ is an irreducible representation of $\U(1,1)$.
\item
$R(V_1^+) = I(\tfrac{1}{2}, \xi^2)$.
\item
$R(V_1^-)$ is the unique maximal submodule of $I(\tfrac{1}{2}, \xi^2)$ and is irreducible of codimension $1$.
\end{enumerate}
\end{theorem}

We are now ready to prove the proposition.

\begin{proof}[Proof of Proposition \ref{p:zeta vanishing}]
By Lemma \ref{l:lift to U2}(a), we may assume that $\chi_v \from E_v^1 \to \CC^\times$ is nontrivial. Since $\chi \xi \overline \xi = \chi$ and $\overline{\chi \xi} \xi = \overline \chi$, by the ``Ten Commandments'' for $\gamma$-factors \cite[Theorem 4]{LR05}, we have
\begin{equation*}
L^S(s, \chi) = \prod_{v \in S} \gamma_v(s, (\chi \xi)_v, \overline\xi_v, \psi_v) \cdot L^S(1-s, \overline\chi),
\end{equation*}
where $S$ is a finite set of places containing all the archimedean places and all places where $\chi_v$ is ramified. Now, since $\chi$ is nontrivial, we must have $L^S(0, \chi) \neq 0$ and $L^S(1,\overline\chi) \neq 0$
, and therefore 
\begin{equation*}
\gamma_v(0, (\chi \xi)_v, \xi_v, \psi_v) \neq 0.
\end{equation*}
This implies that Equation \eqref{e:zeta functional} gives
\begin{equation}\label{e:functional -1/2}
Z\left(\tfrac{1}{2}, \chi \xi, \xi^2\right)\left(M_\psi^{\textrm{LR}}(-\tfrac{1}{2}, \xi^2)(\cF)\right) = * \cdot Z\left(-\tfrac{1}{2}, \chi \xi, \xi^2\right)\left(\cF\right),
\end{equation}
where $*$ is nonzero. We now investigate the intertwining operator 
\begin{equation*}
M_\psi^{\textrm{LR}}(-\tfrac{1}{2}, \xi) \from I(-\tfrac{1}{2}, \xi^2) \to I(\tfrac{1}{2}, \xi^2).
\end{equation*}
We refer to Theorem \ref{t:decomp U(1,1)} for the decomposition of the $\U(1,1)$-representations $I(\pm \tfrac{1}{2}, \xi^2)$. By \cite[Proposition 6.4]{KS97},
\begin{equation*}
\ker(M_\psi^{\textrm{LR}}(-\tfrac{1}{2}, \xi^2)) = R(0, \xi^2), \qquad \im(M_\psi^{\textrm{LR}}(-\tfrac{1}{2}, \xi^2)) = R(V_1^-).
\end{equation*}
Since $\chi$ is nontrivial, $\Theta_{V_0^+}(\chi \xi) = 0$, and therefore $Z(-\tfrac{1}{2}, \chi\xi, \xi^2)|_{R(V_0^+)} = 0$. On the other hand, $Z(-\tfrac{1}{2}, \chi \xi, \xi^2)$ is a nonzero functional, and therefore one can find $\cF \in I(-\tfrac{1}{2}, \xi^2)$ such that $M_\psi^{\textrm{LR}}(-\tfrac{1}{2}, \xi^2)(\cF) \neq 0$. By Theorem \ref{t:decomp U(1,1)}(iv), it follows that $Z(\tfrac{1}{2}, \xi\chi, \xi^2)|_{R(V_1^-)} \neq 0$.
\end{proof}

\subsection{Unramified local theta lifts from $\GU(1)$ to $\GU(1,1)$}\label{s:unram theta lift}

For convenience of notation, in this subsection we drop the subscript $v$. We denote by $\overline x$ the image of $x \in E$ under the nontrivial involution of $E/F$.

Consider the $2$-dimensional $E$-space $V' = V_1' + V_2'$ with skew-Hermitian form
\begin{equation*}
\langle (x_1, x_2), (y_1, y_2) \rangle = \overline x_1 y_2 + \overline x_2 y_1
\end{equation*}
for $(x_1, x_2), (y_1, y_2) \in V_1' + V_2'$. Then
\begin{equation*}
\GU(V') = \GU(1,1) = \left\{g \in \GL_2(E) : \text{$\overline g^\trans \left(\begin{smallmatrix} 0 & 1 \\ -1 & 0 \end{smallmatrix}\right) g = \nu(g) \left(\begin{smallmatrix} 0 & 1 \\ -1 & 0 \end{smallmatrix}\right)$ for some $\nu(g) \in F^\times$}\right\}.
\end{equation*}
The upper-triangular matrices in $\GU(V')$ form a parabolic subgroup 
\begin{equation*}
P \colonequals \left\{\left(\begin{smallmatrix} a & \nu' a \\ 0 & \nu a \end{smallmatrix}\right)  \in \GL_2(E) : a \in E^\times, \, \nu \in F^\times, \, \nu' \in F \right\}.
\end{equation*}
Let $P_F$ denote the Borel subgroup of $\GL_2(F)$ consisting of upper-triangular matrices in $\GL_2(F)$. Observe that there are natural inclusions $\GL_2(F) \hookrightarrow \GU(V')$ and $E^\times \hookrightarrow \GU(V')$ given by
\begin{equation*}
\GL_2(F) = \left\{\left(\begin{smallmatrix} a & b \\ c & d \end{smallmatrix}\right) \in \GU(V') : a,b,c,d \in F\right\}, \qquad E^\times = \left\{\left(\begin{smallmatrix} a & \\ & a \end{smallmatrix}\right) \in \GU(V') : a \in E^\times\right\}.
\end{equation*}
We have $\GU(V') \cong (\GL_2(F) \times E^\times)/F^\times$ and $P \cong (P_F \times E^\times)/F^\times$.

Endow $E$ with the Hermitian form $(x,y) = x \overline y$ so that $\GU(E) = \GU(1) = E^\times.$ Note that the similitude character on $\GU(E)$, which we also denote by $\nu$, is given by the norm map $E^\times \to F^\times.$ Now consider the group
\begin{equation*}
R \colonequals \{(h,g) \in E^\times \times \GU(V') : \nu(g) = \nu(h)\}.
\end{equation*}
Endow the $4$-dimensional $F$-space $\bV' = \Res_{E/F}(V')$ with the symplectic form $\llangle v, w \rrangle = \frac{1}{2} \Tr_{E/F}(\langle v, w \rangle).$ There is a natural map
\begin{equation*}
\iota \from R \to \Sp(\bV), \qquad (h, g) \mapsto (v \mapsto h^{-1} v g).
\end{equation*}
The decomposition $V_1' + V_2'$ of $V'$ into isotropic subspaces induces a polarization of $\bV'$ given by
\begin{equation*}
\bV' = \bX' + \bY', \qquad \text{where $\bX' = \Res_{E/F}(V_1')$ and $\bY' = \Res_{E/F}(V_2')$.}
\end{equation*}
Choose a basis $\e_1$, $\e_2$, $\e_1^*$, $\e_2^*$ of $\bV'$ such that
\begin{equation*}
\bX' = F \e_1 + F \e_2, \qquad \bY' = F \e_1^* + F \e_2^*, \qquad \llangle \e_i, \e_j^* \rrangle = \delta_{ij}.
\end{equation*}
Now assume that we have a splitting $\beta \from R \to \CC^1$ of $z_{\bY'}$. Then
\begin{equation*}
R \to \Mp(\bV')_{\bY'}, \qquad g \mapsto (\iota(g), \beta(g))
\end{equation*}
is a group homomorphism and the Weil representation $\omega_\psi$ on $\Mp(\bV')_{\bY'}$ pulls back to a representation of $R$, which we also denote by $\omega_\psi$. 

Abusing notation, define
\begin{equation*}
\beta \from E^\times \to \CC^1, \qquad h \mapsto \beta(h, d(\nu(h))).
\end{equation*}
Observe that this defines a character since for any $h \in E^\times$, $\iota(h, d(\nu(h)))$ stabilizes $\bY'$. Define
\begin{equation*}
L(h)\phi(x) \colonequals \omega_\psi(h, d(\nu(h)))\phi(x) = \beta(h) |h|^{-1/2} \phi(x h^{-1})
\end{equation*}
for $h \in E^\times$ and $\phi \in \cS(\bX')$. Then for any $(h,g) \in R$, 
\begin{equation}\label{e:Weil L}
\omega_\psi(h, g)\phi(x) = L(h) \omega_\psi(d(\nu(g)^{-1}) g)\phi(x) = \beta(h) |h|^{-1/2} (\omega_\psi(d(\nu(g)^{-1})g)\phi)(x h^{-1}).
\end{equation}

Consider the semidirect product $E^\times \ltimes \U(V')$ with multiplication
\begin{equation*}
(h_1, g_1) * (h_2, g_2) = (h_1 h_2, d(\nu(h_2)) g_1 d(\nu(h_2)^{-1}) g_2), \qquad \text{where $h \in E^\times$ and $g \in \U(V')$.}
\end{equation*}
This defines a group multiplication since the map $d$ is multiplicative and $\nu$ is a group homomorphism to $F^\times$, an abelian group. It is easy to show:

\begin{lemma}\label{l:semidirect}
The Weil representation $\omega_\psi$ on $R$ extends to a representation of $E^\times \ltimes \U(V')$ defined by
\begin{equation*}
\omega_\psi(h,g) = L(h) \omega_\psi(g), \qquad \text{$h \in E^\times$, $g \in \U(V')$}.
\end{equation*}
In particular, the Weil representation on the quotient 
\begin{equation*}
\Theta^{(1)}(\triv) \colonequals \cS(\bX')/\bigcap_{\alpha \in \Hom_{E^1}(\cS(\bX'), \triv)} \ker(\alpha)
\end{equation*}
extends to a representation of $\GU(V')^+ \cong \{d(\nu) : \nu \in \Nm(E^\times)\} \ltimes \U(V')$ satisfying
\begin{equation*}
\omega_\psi(d(\nu)) = L(h),
\end{equation*}
where $h \in E^\times$ is any element such that $\nu(h) = \nu$.
\end{lemma}

%
%

\begin{definition}
For any character $\eta_0 \from F^\times \to \CC$ and any $\phi \in \cS(\bX')$, define 
\begin{equation*}
\cF_{\phi,\eta_0} \from \GU(V') \to \CC^\times, \qquad g \mapsto |\nu(g)|^{-1/2} \eta_0(\nu(g))^{-1} (\omega_\psi(d(\nu(g)^{-1}) g)\phi)(0).
\end{equation*}
\end{definition}

The following is straightforward:

\begin{lemma}
For any $p = \left(\begin{smallmatrix} a & b \\ 0 & d \end{smallmatrix}\right) \in \GU(V')$,
\begin{equation*}
\cF_{\phi, \eta_0}(p g) = |a|^{1/2} |d|^{-1/2}  \eta_0(\overline a d)^{-1} \beta(a)^{-1} \cF_{\phi, \eta_0}(g)
\end{equation*}
for all $g \in \GU(V')$ so that
\begin{equation*}
\cF_{\phi, \eta_0} \in \Ind_P^{\GU(V')}(\widetilde \eta_0), \qquad \text{where $\widetilde \eta_0\left(\begin{smallmatrix} a & b \\ 0 & d \end{smallmatrix}\right) \colonequals \eta_0(\overline a d)^{-1} \beta(a)^{-1}.$}
\end{equation*} 
In particular, $\cF_{\phi, \eta_0}|_{\GSp(2)}$ is an element of the (normalized) principal series representation
\begin{equation*}
\Ind_{B}^{\GSp(2)}(\eta_0^{-1} \beta^{-1} \otimes \eta_0^{-1}).\end{equation*}
\end{lemma}

%

\begin{lemma}\label{l:R equiv}
We have a nonzero $R$-equivariant map
\begin{equation*}
(\omega_\psi, \cS(\bX')) \to \Ind_{P}^{\GU(V')}(\widetilde \eta_0) \otimes (\eta_0(\Nm) \cdot \beta), \qquad \phi \mapsto \cF_{\phi, \eta_0}.
\end{equation*}
The right-hand side is irreducible and we have an isomorphism
\begin{equation*}
\Ind_P^{\GU(V')}(\widetilde \eta_0) \cong \Ind_{P_F}^{\GL_2(F)}(\eta_0^{-1} \otimes (\eta_0 \cdot \beta)^{-1}) \otimes (\eta_0(\Nm) \cdot \beta)^{-1},
\end{equation*}
where the right-hand side is a representation of $\GL_2(F) \times E^\times$ that descends to the quotient $(\GL_2(F) \times E^\times)/F^\times \cong \GU(V')$.
\end{lemma}

\begin{proof}
It is clear by definition that the map is nonzero. For $R$-equivariance:
\begin{align*}
\cF_{\omega_\psi(h', g')\phi, \eta_0}(g) 
&= |\nu(g)|^{-1/2} \eta_0(\nu(g))^{-1} (\omega_\psi(d(\nu(g)^{-1})g)L(h') \omega_\psi(d(\nu(g')^{-1})g')\phi)(0) \\
&= |\nu(g)|^{-1/2} \eta_0(\nu(g))^{-1} (L(h') \omega_\psi(d(\nu(gg')^{-1}) gg')\phi)(0) \\
&= |\nu(g)|^{-1/2} \eta_0(\nu(g))^{-1} |\nu(h')|^{-1/2} \beta(h') (\omega_\psi(d(\nu(gg')^{-1}) gg')\phi)(0) \\
&= \beta(h') \eta_0(\nu(h')) |\nu(gg')|^{-1/2} \eta_0(\nu(gg'))^{-1} (\omega_\psi(d(\nu(gg')^{-1}) gg')\phi)(0) \\
&= \beta(h') \eta_0(\nu(h')) \cF_{\phi, \eta_0}(gg').
\end{align*}
The last assertion in the lemma holds since $P \cong (P_F \times E^\times)/F^\times$ and $\GU(V') \cong (\GL_2(F) \times E^\times)/F^\times$. The representation $\Ind_{P_F}^{\GL_2(F)}(\widetilde \eta_0)$ is irreducible since the character $\eta_0^{-1} \beta^{-1} \eta_0 = \beta^{-1}$ is not $| \cdot |$ or $| \cdot |^{-1}$. It follows that $\Ind_P^{\GU(V')}(\widetilde \eta_0)$ is irreducible.
\end{proof}

The map defined in Lemma \ref{l:R equiv} factors through 
\begin{equation*}
\Theta^{(1)}(\beta) \colonequals \cS(\bX')/\bigcap_{\alpha \in \Hom_{E^1}(\cS(\bX'), \beta)} \ker \alpha,
\end{equation*}
the largest quotient of $\cS(\bX')$ such that $E^1$ acts by $\beta$. Note that by construction, $\Theta^{(1)}(\beta)$, as a representation of $\U(V')$, is the local theta lift of $\beta$ to $\U(V')$.

There are many extensions of $\Theta^{(1)}(\beta)$ to a representation of $E^\times \times \GU(V')^+$, but specifying an action of $E^\times$ determines such an extension. Explicitly, define $\Theta_{\ur,\beta}(\beta \cdot \eta_0(\Nm))$ to be the unique representation of $\GU(V')^+$ such that for $g = \left(\begin{smallmatrix} 1 & 0 \\ 0 & \nu \end{smallmatrix}\right) \in \GU(V')^+$, 
\begin{equation*}
\Theta_{\ur,\beta}(\beta \cdot \eta_0(\Nm))(g) \colonequals \eta_0(\Nm(h))^{-1} \cdot \Theta^{(1)}(\beta)(h, g),
\end{equation*}
where $h \in E^\times$ is any element such that $\nu(h) = \nu(g) = \nu$.

\begin{theorem}[Rallis]\label{t:unram theta lift}
The $R$-equivariant map in Lemma \ref{l:R equiv} factors through $\Theta_{\ur,\beta}(\beta \cdot \eta_0(\Nm))$ and induces an injective map:
\begin{equation*}
\begin{tikzcd}
(\omega_\psi, \cS(\bX')) \ar{r} \ar{d} & \Ind_P^{\GU(V')}(\widetilde \eta_0) \\
\Theta_{\ur,\beta}(\beta \cdot \eta_0(\Nm)) \ar[hookrightarrow]{ru}
\end{tikzcd}
\end{equation*}
Moreover,
\begin{equation*}
\Theta_{\ur,\beta}(\beta \cdot \eta_0(\Nm)) \cong \Ind_{P_F}^{\GL_2(F)}(\eta_0^{-1} \epsilon_{E/F} \otimes \eta_0^{-1}) \otimes (\eta_0(\Nm)^{-1} \cdot \beta^{-1}),
\end{equation*}
where the right-hand side is viewed as a representation of $\GL_2(F) \times E^\times$ that descends to the quotient $(\GL_2(F) \times E^\times)/F^\times \cong \GU(V')$.
\end{theorem}

\begin{proof}
This is due to Rallis \cite[Theorem II.1.1]{R84}. By the injectivity of 
\begin{equation*}
\Theta_{\ur,\beta}(\beta \cdot \eta_0(\Nm)) \hookrightarrow \Ind_P^{\GU(V')}(\widetilde \eta_0)
\end{equation*}
and the irreducibility of $\Ind_P^{\GU(V')}(\widetilde \eta_0)$, by Lemma \ref{l:R equiv}, we have an isomorphism
\begin{equation*}
\Theta_{\ur,\beta}(\beta \cdot \eta_0(\Nm)) \cong \Ind_{P_F}^{\GL_2(F)}(\eta_0^{-1} \beta^{-1} \otimes \eta_0^{-1}) \otimes (\eta_0(\Nm)^{-1} \cdot \beta^{-1}).
\end{equation*}
Finally, by Lemma \ref{l:bs D}, the restriction of $\beta$ to $F^\times$ is exactly the quadratic character $\epsilon_{E/F}$, and this completes the proof.
\end{proof}

%

\subsection{Proof of Theorem \ref{t:theta aut ind}}\label{s:theta aut ind proof}

In this section, we use the calculations in the preceding sections to prove Theorem \ref{t:theta aut ind}, the main theorem of this section. 

Let $\chi$ and $\chi'$ be Hecke characters of $E^\times$. Recall from Section \ref{s:theta sim} that for every Schwartz function $\varphi \in \cS(\bX(\bA))$ we have automorphic forms $\theta_\varphi(\chi)$ and $\theta_\varphi'(\chi')$ on the adelic groups $H(\bA) \cong B_\bA^\times$ and $H'(\bA) \cong ((B_\bA')^\times \times \bA_E^\times)/\bA_F^\times$, respectively. Let $\Theta(\chi)$ denote the automorphic representation of $H(\bA)$ generated by $\theta_\varphi(\chi)$ for all $\varphi \in \cS(\bX(\bA))$ and let $\Theta'(\chi')$ denote the automorphic representation of $H'(\bA)$ generated by $\theta_\varphi'(\chi')$ for all $\varphi \in \cS(\bX(\bA))$.

\begin{proposition}\label{p:JL nonzero}
If $\pi_\chi^B \neq 0$, then $\Theta(\chi \cdot \xi) \neq 0$. Analogously, if $\pi_{\chi'}^{B'} \neq 0$, then $\Theta'(\chi' \cdot \xi') \neq 0$.
\end{proposition}

\begin{proof}
Recall that $\pi_\chi^B \neq 0$ if and only if $\chi_v|_{E_v^1} \neq 1$ for every place $v$ where $B_v$ is nonsplit. Let $v$ such a place, i.e.\ $B_v$ is nonsplit and $\chi_v|_{E_v^1} \neq 1$. By Lemma \ref{l:lift to U2}(a), we have $\Theta_v(\chi_v \xi_v) \neq 0$, and by Proposition \ref{p:zeta vanishing}, we have $Z_v(\tfrac{1}{2}, -, \chi_v \xi_v) \neq 0$. Now let $v$ be a place such that $B_v$ is split. By Lemma \ref{l:lift to U2}(b), we have $\Theta_v(\chi_v \xi_v) \neq 0$, and by Theorem \ref{t:decomp U(1,1)}(c), we have $Z_v(\tfrac{1}{2}, -, \chi_v \xi_v) \neq 0$. By Rallis inner product formula (Proposition \ref{p:split rallis}), $\Theta(\chi \cdot \xi) \neq 0$ if and only if all the local zeta integrals $Z_v(\tfrac{1}{2}, -, \chi_v \xi_v) \neq 0$, and hence we have shown that $\Theta(\chi \cdot \xi) \neq 0$.
\end{proof}

\begin{lemma}\label{l:lift cusp}
If $\chi, \chi'$ are Hecke characters of $\bA_E^\times$ whose restriction to $\bA_E^1$ is nontrivial, then $\Theta(\chi \cdot \xi)$ is a cuspidal automorphic representation of $B_\bA^\times$ and $\Theta'(\chi' \cdot \xi')$ is a cuspidal automorphic representation of $B_\bA'{}^\times$.
\end{lemma}

\begin{proof}
If $B \neq M_2(F)$, then the statement holds trivially. Now assume $B = M_2(F)$. We would like to prove that for any Schwartz function $\phi \in \cS(\bX(\bA))$, 
\begin{equation}\label{e:GL2 cuspidality}
\int_{F \backslash \bA_F} \theta_\phi(\chi)(\bn(b)g) \, db = 0, \qquad \text{where $\bn(b) \colonequals \left(\begin{matrix} 1 & b \\ 0 & 1 \end{matrix}\right)$.}
\end{equation}
Observe that if $g \notin \GL_2^+(\bA_F)$, then $\bn(b)g \notin \GL_2^+(\bA_F)$, and hence the integrand in \eqref{e:GL2 cuspidality} is identically zero. Now assume $g \in \GL_2^+(\bA_F)$ and pick $\alpha \in \bA_E^\times$ with $\Nm(\alpha) = \det(g)$. Then by definition
\begin{equation*}
\theta_\phi(\chi)(\bn(b)g) = \theta_{\omega_\psi(\alpha,g)\phi}(\chi)(\bn(b)),
\end{equation*}
and therefore it remains only to show 
\begin{equation*}
\int_{F \backslash \bA_F} \theta_\phi(\chi)(\bn(b)) \, db = 0.
\end{equation*}

Recall that if $B$ is split, then the $2$-dimensional $E$-space $W_0$ is a split Hermitian space and one has a decomposition $W_0 = W_1 + W_2$ into isotropic subspaces of dimension $1$. This induces a complete polarization $\bV = \bX' + \bY'$ given by $\bX' = \Res_{E/F}(V_0 \otimes W_1)$ and $\bY' = \Res_{E/F}(V_0 \otimes W_2)$. Then $\bA_E^1 \subset \U(V_0)$ stabilizes $\bX'$ and $\bY'$, and so for $\alpha \in \bA_E^1$, $b \in \bA_F$, and $\phi' \in \cS(\bX'(\bA))$,
\begin{equation*}
\omega_\psi(\alpha, \bn(b))\phi'(x) = \xi^{-1}(\alpha) \cdot \psi\left(\tfrac{1}{2} b x x^\trans\right) \cdot \phi'(x \alpha).
\end{equation*}
We have
\begin{align*}
\int_{F \backslash \bA_F} \theta_\phi(\chi)(\bn(b)) \, db
&= \int_{F \backslash \bA_F} \int_{E^1 \backslash \bA_E^1} \sum_{x \in \bX'(F)} (\omega_\psi(\alpha, \bn(b)))\phi'(x) \cdot (\chi \xi)(\alpha) \, d\alpha \, db \\
&= \int_{E^1 \backslash \bA_E^1} \sum_{x \in \bX'(F)} \int_{F \backslash \bA_F} \xi^{-1}(\alpha) \cdot \psi\left(\tfrac{1}{2} b x x^\trans\right) \cdot \phi'(x \alpha) \cdot (\chi \xi)(\alpha)  \, db \, d\alpha \\
&= \int_{E^1 \backslash \bA_E^1} \xi^{-1}(\alpha) \cdot \phi'(0) \cdot (\chi \xi)(\alpha) \, d\alpha 
= \phi'(0) \int_{E^1 \backslash \bA_E^1} \chi(\alpha) \, d\alpha = 0. \qedhere
\end{align*}
\end{proof}

\begin{theorem}\mbox{}\label{t:global lift}
Let $\chi$, $\chi'$ be Hecke characters of $\bA_E^\times$ whose restriction to $\bA_E^1$ is nontrivial. 
\begin{enumerate}[label=(\alph*)]
\item
If $\Theta(\chi \cdot \xi)$ is nonzero, then
\begin{equation*}
\Theta(\chi \cdot \xi) \cong \pi_{\chi}^B.
\end{equation*}

\item
If $\Theta'(\overline{\chi' \cdot \xi'{}^{-1}})$ is nonzero, then
\begin{equation*}
\Theta'(\overline{\chi' \cdot \xi'{}^{-1}})^\vee \cong \pi_{\chi'{}}^{B'} \otimes (\chi' \cdot \xi'{}^{-1}),
\end{equation*}
where the right-hand side is viewed as a representation of $H'(\bA) \cong ((B_\bA')^\times \times \bA_E^\times)/\bA_F^\times$ descended from the $(B_\bA')^\times \times \bA_E^\times$ representation written above.
\end{enumerate}
\end{theorem}

\begin{proof}
We prove (a) first. By our normalization (compare the local definition in Section \ref{s:howe duality} to the global definition in Section \ref{s:theta sim}), at a place $v$, the local representation corresponding to the global theta lift of $\chi \cdot \xi$ is the local theta lift of $(\chi_v \cdot \xi_v)^{-1}$. That is,
\begin{equation*}
\Theta(\chi \cdot \xi)_v \cong \Theta_v((\chi_v \cdot \xi_v)^{-1}) \cong \Theta_v(\chi_v^{-1} \cdot \xi_v^{-1}).
\end{equation*}
Theorem \ref{t:unram theta lift} gives a description of the right-hand side for every place $v$ such that 
\begin{enumerate}[label=$\cdot$]
\item
$v$ splits completely in $E$, or

\item
$v$ lies under a single place $w$ of $E$ and $\chi_w \from E_w^\times \to \CC^\times$ factors through $\Nm \from E_w^\times \to F_v^\times$.
\end{enumerate}
For each such place $v$, by Lemma \ref{l:bs D}, we have
\begin{equation*}
\bs(\alpha, d(\nu(\alpha))) = \xi(\alpha)^{-1}, \qquad \text{for all $\alpha \in E_v^\times$}.
\end{equation*}
Writing $\chi_v = \chi_{v,0}(\Nm)$, we have
\begin{equation*}
\Theta_v(\chi_v^{-1} \cdot \xi_v^{-1}) \cong \Theta_{\ur,\xi_v^{-1}}(\chi_v^{-1} \xi_v^{-1}) \cong \Ind_{P_{F_v}}^{\GL_2(F_v)}(\chi_{v,0} \epsilon_{E_v/F_v} \otimes \chi_{v,0}),
\end{equation*}
and therefore by Jacquet--Langlands, we have that $\Theta(\chi \cdot \beta) \cong \pi_{\chi}^B.$

The proof of (b) is very similar. In this case, because we complex-conjugate the theta kernel in the definition of the global theta lift $\Theta'$ (see Section \ref{s:theta sim}), we have
\begin{equation*}
\Theta'(\overline{\chi' \cdot \xi'{}^{-1}})_v^\vee \cong \Theta_v((\chi_v' \cdot \xi_v'{}^{-1})^{-1}) = \Theta_v(\chi_v'{}^{-1} \cdot \xi_v').
\end{equation*}
At every place $v$ of $F$ where everything is unramified, by Lemma \ref{l:bs' D},
\begin{equation*}
\bs'(d(\nu(\alpha)), \alpha) = \xi'(\alpha), \qquad \text{for all $\alpha \in E_v^\times$}.
\end{equation*}
Writing $\chi_v' = \chi_{v,0}'(\Nm)$ at each such place, Theorem \ref{t:unram theta lift} implies
\begin{equation*}
\Theta_v(\chi_v'{}^{-1} \cdot \xi_v') \cong \Theta_{\ur, \xi_v'}(\chi_v'{}^{-1} \cdot \xi_v') \cong \Ind_{P_{F_v}}^{\GL_2(F_v)}(\chi_{v,0}' \xi_v' \otimes \chi_{v,0}') \otimes (\chi_{v,0}' \cdot \xi_v'{}^{-1}).
\end{equation*}
Since $\xi_v'|_{F_v^\times} = \epsilon_{E_v/F_v}$, therefore by Jacquet--Langlands, $\Theta'(\overline{\chi' \cdot \xi'{}^{-1}})^\vee \cong \pi_{\chi'}^{B'} \otimes (\chi'{}^{-1} \cdot \xi').$
\end{proof}

Theorem \ref{t:theta aut ind} now follows from Proposition \ref{p:JL nonzero} and Theorem \ref{t:global lift}.

\begin{proof}[Proof of Theorem \ref{t:theta aut ind}]
If $\Theta(\chi \cdot \xi) = 0$, then by Proposition \ref{p:JL nonzero} we must have $\pi_\chi^B = 0$ and therefore $\Theta(\chi \cdot \xi) = \pi_\chi^B$. If $\Theta(\chi \cdot \xi) \neq 0$, then by Theorem \ref{t:global lift} we must have $\Theta(\chi \cdot \xi) \cong \pi_\chi^B$. The same argument holds to conclude the desired isomorphism for $\Theta'(\overline{\chi' \cdot \xi'{}^{-1}})$.
\end{proof}

\subsection{Period identities of CM forms}

We are now ready to prove an identity of toric integrals of automorphic forms in $\pi_\chi^B$ and $\pi_{\chi'}^{B'}$. We use the seesaw
\begin{equation*}
\begin{tikzcd}
H' \ar[dash]{d} \ar[dash]{dr} & H \ar[dash]{d} \ar[dash]{dl} \\
G & G'
\end{tikzcd} 
=
\begin{tikzcd}
\GU_E(\Res V) \ar[dash]{d} \ar[dash]{dr} & \GU_B(W^*) \ar[dash]{d} \ar[dash]{dl} \\
\GU_B(V)^\circ & \GU_E(W)
\end{tikzcd} 
\cong
\begin{tikzcd}
((B')^\times \times E^\times)/F^\times \ar[dash]{d} \ar[dash]{dr} & B^\times \ar[dash]{d} \ar[dash]{dl} \\
E^\times & E^\times
\end{tikzcd} 
\end{equation*}
Recall from Proposition \ref{p:compat} that our choice of splittings
\begin{equation*}
s \from \cG_{G \times H}(\bA) \to \CC^1, \qquad s' \from \cG_{G' \times H'}(\bA) \to \CC^1
\end{equation*}
enjoys the property that for $(\alpha, \beta) \in \cG_{G \times G'}(\bA)$,
\begin{equation*}
s'(\alpha,\beta) = \xi(\alpha) \cdot \xi'(\beta) \cdot s(\alpha,\beta).
\end{equation*}

\begin{theorem}\label{t:B B' adjoint}
For any Hecke characters $\chi$ and $\chi'$ of $E$,
\begin{equation*}
\langle \theta_\varphi(\chi \cdot \xi), \overline{\chi'} \rangle_{G'} = \langle \chi, \theta_\varphi'(\overline{\chi' \cdot \xi'{}^{-1}}) \rangle_G.
\end{equation*}
\end{theorem}

\begin{proof}
Unwinding definitions and using Proposition \ref{p:compat}, we have
\begin{align*}
\langle \theta_\varphi(\chi \cdot \xi), \overline{\chi'} \rangle_{G'}
&= \int_{[G']} \theta_\varphi(\chi \cdot \xi)(g') \cdot \chi'(g') \, dg' \\
&= \int_\cC \int_{[G_1']} \theta_\varphi(\chi \cdot \xi)(g_1'g_c') \cdot \chi'(g_1' g_c') \, dg_1' \, dc \\
&= \int_\cC \int_{[G_1']} \int_{[G_1]} \Theta(\omega_\psi(g_1g_c, g_1' g_c')\varphi) \cdot \chi(g_1 g_c) \cdot \xi(g_1 g_c) \cdot \chi'(g_1' g_c') \, dg_1 \, dg_1' \, dc \\
&= \int_\cC \int_{[G_1]} \int_{[G_1']} \chi(g_1 g_c) \Theta(\omega_\psi'(g_1g_c, g_1'g_c')\varphi) \cdot \xi'(g_1'g_c')^{-1} \cdot \chi'(g_1' g_c') \, dg_1' \, dg_1 \, dc \\
&= \int_\cC \int_{[G_1]} \chi(g_1 g_c) \overline{\theta_\varphi'(\overline{\chi' \cdot \xi'{}^{-1}})}(g_1 g_c) \, dg_1 \, dc \\
&= \langle \chi, \theta_\varphi'(\overline{\chi' \cdot \xi'{}^{-1}}) \rangle_G. \qedhere
\end{align*}
\end{proof}

Combining Theorems \ref{t:global lift} and \ref{t:B B' adjoint}, we obtain the following result:

\begin{theorem}\label{t:identity of periods}
Let $\chi, \chi'$ be Hecke characters of $E$ and let $\varphi \in \cS(\bX(\bA))$. Then
\begin{equation*}
f_\chi^B \colonequals \theta_\varphi(\chi \cdot \xi) \in \pi_\chi^B, \qquad f_{\chi'}^{B'} \colonequals \overline{\theta_\varphi'(\overline{\chi' \cdot \xi'{}^{-1}})} \in \pi_{\chi'}^{B'},
\end{equation*}
and we have
\begin{equation*}
\int_{\bA_F^\times E^\times \backslash \bA_E^\times} f_\chi(g) \cdot \chi'(g) \, dg = \int_{\bA_F^\times E^\times \backslash \bA_E^\times} \chi(g) \cdot f_{\chi'}(g) \, dg.
\end{equation*}
\end{theorem}

\section{Special vectors in the Weil representation}\label{ch:schwartz}

Recall that $F$ is a totally real field and $E = F(\bi)$ is a CM extension of $F$. We choose the trace-free element $\bi \in E$ so that $u = \bi^2 \in F$ has the property that for any finite place $v$ of $F$,
\begin{equation*}
\val_v(u) = \begin{cases}
0 & \text{if $E_v/F_v$ is unramified} \\
1 & \text{if $E_v/F_v$ is ramified.}
\end{cases}
\end{equation*}
For the rest of the paper, we take $\psi$ to be the standard additive character of $F \backslash \bA_F$ (see Section \ref{s:tamagawa}). Recall that if $v$ is a finite place of $F$, then $\psi_v$ is trivial on $\pi_v^{-d_v} \cO_v$ but nontrivial on $\pi_v^{-d_v-1} \cO_v$. Furthermore, recall that we let $dx$ be the additive Haar measure on $\bA_F$ self-dual with respect to $\psi$ and that $\vol(\cO_{F_v}, dx_v) = q_v^{-d_v/2}.$

In this section, we will explicitly realize the positive-weight Hecke eigenforms as theta lifts. More precisely, we will specify Schwartz functions $\phi_l'$ for $l \in \bZ_{\geq 0}$ such that if $\chi_\infty(z) = z^{-k}$ on $\CC^1$, then the theta lift $\theta_{\phi_l'}(\chi \xi)$ is a Hecke eigenform of weight $|k|+1+2l$. Note that by construction (Section \ref{s:theta sim}), negative-weight Hecke eigenforms are not theta lifts since they are not supported on $\GL_2(F) \GL_2(\bA_F)^+$.

Fix a place $v$ of $F$. In this section, we work place by place, and drop the subscript $v$ throughout. Let $\W$ be a $2$-dimensional $E$-vector space endowed with the skew-Hermitian form
\begin{equation*}
\langle (x_1, x_2), (y_1, y_2) \rangle = \overline x_1 y_2 - \overline x_2 y_1
\end{equation*}
with respect to a fixed basis $w_1, w_2$ of $\W$. Let $\V$ be a $1$-dimensional $E$-vector space endowed with the Hermitian form $(\alpha, \beta) = \alpha \overline \beta.$ Setting $\W_i = \Span_\CC(w_i)$ for $i = 1,2$, we have a decomposition $\W = \W_1 + \W_2$ of $\W$ into maximal isotropic subspaces, and this induces a complete polarization of $\bV$ given by
\begin{equation*}
\bV = \bX' + \bY', \qquad \bX' = \V \otimes \W_1, \qquad \bY' = \V \otimes \W_2.
\end{equation*}
Fix a splitting
\begin{equation*}
\bs \from \G(\U(\V) \times \U(\W)) \to \CC^1
\end{equation*}
of the cocycle $z_{\bY'}$ with respect to the map
\begin{equation*}
\iota \from \G(\U(\V) \times \U(\W)) \to \Sp(\bV), \qquad (h,g) \mapsto (v \otimes w \mapsto h^{-1} v \otimes w g).
\end{equation*}
This determines a homomorphism
\begin{equation*}
\tilde \iota \from \G(\U(\V) \times \U(\W)) \to \Mp(\bV)_{\bY'}, \qquad (h,g) \mapsto (\iota(h,g), \bs(h,g)).
\end{equation*}
Recall from Equation \eqref{e:Weil L} and Lemma \ref{l:bs D} that for $\phi \in \cS(\bX')$ and $(h,g) \in \G(\U(\V) \times \U(\W))$,
\begin{equation}\label{e:weil split similitude}
\omega_\psi(h,g) \phi(x) = \xi^{-1}(h) |h|^{-1/2} (\omega_\psi(d(\nu(g)^{-1})g)\phi)(xh^{-1}).
\end{equation}
One can choose a basis of $\bX'$ and $\bY'$ so that
\begin{equation*}
\iota(D(a)) = \left(\begin{smallmatrix}
a & & & \\ & a & & \\ & & a^{-1} & \\ & & & a^{-1}
\end{smallmatrix}\right), \;\;
\iota(U(a')) = \left(\begin{smallmatrix}
1 & & a' & \\ & 1 & & a' \\ & & 1 & \\ & & & 1
\end{smallmatrix}\right), \;\;
\iota(W) = \left(\begin{smallmatrix}
& & 1 & \\ & & & 1 \\ -1 & & & \\ & -1 & & 
\end{smallmatrix}\right).
\end{equation*}
By the computations of Section \ref{s:split splittings} and Equations \eqref{e:weil explicit m}, \eqref{e:weil explicit n}, and \eqref{e:weil explicit w},
\begin{align}
\label{e:weil SL2 d}
\omega_\psi(1, D(a)) \varphi(x) &= \xi(a)^{-1} \cdot |\det a| \cdot \varphi(xa) \\
\label{e:weil SL2 u}
\omega_\psi(1, U(a')) \varphi(x) &= \psi\left(\tfrac{1}{4} \Tr_{E/F}(a' x \overline x)\right) \cdot \varphi(x) \\
\label{e:weil SL2 w}
\omega_\psi(1, W) \varphi(x) &= (u, -1)_F \cdot \gamma_F(u, \tfrac{1}{2} \psi) \cdot \int_{F^2} \varphi(y) \psi\left(\tfrac{1}{2} \Tr_{E/F} (x \overline y)\right) \, dy
\end{align}

If $v$ is a finite place, let $c(\pi_\chi)$ be the conductor of $\pi_\chi$ and let $K_0'(N)$ be the compact open subgroup as defined in Section \ref{s:conductor}. Writing $d(\nu) = \left(\begin{smallmatrix} 1 & 0 \\ 0 & \nu \end{smallmatrix}\right) \in \GL_2(F)$ for $\nu \in F^\times$, define
\begin{equation*}
K_0(N) \colonequals
\begin{cases}
K_0'(N) & \text{if $F$ has odd residue characteristic,} \\
d(2) K_0'(N) d(1/2) & \text{if $F$ has even residue characteristic.}
\end{cases}
\end{equation*}

\subsection{Schwartz functions}\label{s:schwartz}

In this section, we introduce Schwartz functions that transform nicely under the Weil representation. These functions have been considered in various places before. At the finite places, they have appeared for example in \cite[Proposition 2.5.1]{P06}, \cite[N1]{X07}. At the infinite places, our choice is constructed from a confluent hypergeometric function ${}_1 F_1(a,b,t)$ of the first type. This is related to the role of hypergeometric functions in matrix coefficients of representations of $\SL_2(\RR)$ (see for example \cite[Appendix]{X07}, \cite[Chapters 6, 7]{VK91}).

\subsubsection{Infinite places}

In this section, let $v$ be an infinite place of $F$. 

\begin{definition}
For $k \in \bZ$ and $l \in \bZ_{\geq 0}$, define
\begin{equation*}
\phi_{k,l}'(z) \colonequals
\begin{cases}
{}_1 F_1(-l, k+1,4 \pi z \overline z) \overline z^k e^{-2 \pi z \overline z} & \text{if $k \geq 0$,} \\
{}_1 F_1(-l,-k+1,4 \pi z \overline z) z^{-k} e^{-2 \pi z \overline z} & \text{if $k < 0$,}
\end{cases}
\end{equation*}
where ${}_1 F_1(a,b,t)$ is the Kummer confluent hypergeometric function for constants $a,b$
\begin{equation*}
{}_1 F_1(a,b,t) \colonequals \sum_{j=0}^\infty \frac{(a)_j}{(b)_j} \frac{1}{j!} t^j, \qquad \text{where $(a)_0 \colonequals 1,$ and $(a)_j \colonequals a(a+1)(a+2) \cdots (a+j-1).$}
\end{equation*}
Observe that ${}_1 F_1(a,b,t)$ is entire in $t$ so long as $b \notin \bZ_{\leq 0}$, so that in particular, $\phi_{k,l}'$ is entire for all $k \in \bZ$ and $l \in \bZ_{\geq 0}$.
\end{definition}


The following lemma is well known.

\begin{lemma}\label{l:1F1 props}
\begin{enumerate}[label=(\alph*)]
\item
The function ${}_1 F_1(a,b,t)$ is a solution to the differential equation
\begin{equation*}
t f''(t) + (b-t) f'(t) - af(t) = 0.
\end{equation*}

\item
If $\real(\alpha) > 0$ and $\real(c) > 0$, then
\begin{equation*}
\int_0^\infty t^{\alpha-1} e^{-ct} {}_1 F_1(a,b,-t) \, dt = c^{-\alpha} \Gamma(\alpha) {}_2 F_1\left(a, \alpha, b, -\tfrac{1}{c}\right),
\end{equation*}
where
\begin{equation*}
{}_2 F_1(a,\alpha,b,-\tfrac{1}{c}) = \sum_{j=0}^\infty \frac{(a)_j (\alpha)_j}{(b)_j} \frac{1}{j!} \left(-\frac{1}{c}\right)^j.
\end{equation*}
\end{enumerate}
\end{lemma}

\begin{lemma}\label{l:infinite hecke}
For $\alpha \in \CC^1$ and $r(\theta) = \left(\begin{smallmatrix} \cos(\theta) & \sin(\theta) \\ -\sin(\theta) & \cos(\theta) \end{smallmatrix}\right) \in \SO(2)$,
\begin{equation*}
\omega_\psi(\alpha, r(\theta)) \phi_{k,l}' = \xi(\alpha^{-1}) \alpha^{-k} e^{i(|k| + 1 + 2l)\theta} \phi_{k,l}'.
\end{equation*}
\end{lemma}

\begin{proof}
We follow a similar proof strategy to \cite[Proposition 2.2.5]{X07}. We compute on the Lie algebra $\frak{sl}_2(\RR)$. It is well known that for $X_+ = \left(\begin{smallmatrix} 0 & 1 \\ 0 & 0 \end{smallmatrix}\right),$ $X_- = \left(\begin{smallmatrix} 0 & 0 \\ 1 & 0 \end{smallmatrix}\right)$,
\begin{equation*}
\omega_\psi(X_+) \phi = 2 \pi i z \overline z \phi, \qquad
\omega_\psi(X_-) \phi = -\frac{1}{2 \pi i} \frac{\partial}{\partial z} \left(\frac{\partial}{\partial \overline z} \phi\right).
\end{equation*}
We first handle the case $k \geq 0$. For any doubly differentiable function $f$ satisfying the differential equation
\begin{equation*}
t f''(t) + (k+1-t) f'(t) = -l f(t),
\end{equation*}
we have, following from a long calculus computation,
\begin{align*}
\omega_\psi(X_+ &- X_-)(f(4 \pi z \overline z) \overline z^k e^{-2 \pi z \overline z}) \\
&= i \left[(k+1) f(4 \pi z \overline z) - 2((k+1-4 \pi z \overline z) f'(4 \pi z \overline z) + 4 \pi z \overline z f''(4 \pi z \overline z))\right] \overline z^k e^{-2 \pi z \overline z}  \\
&= i (k+1+2l) f(4 \pi z \overline z) \overline z^k e^{-2 \pi z \overline z}. 
\end{align*}
By Lemma \ref{l:1F1 props}(a), ${}_1 F_1(-l,k+1,t)$ is such an $f(t)$ and hence the desired conclusion follows. 

Now assume $k < 0$. For any doubly differentiable function $f$ satisfying the differential equation
\begin{equation*}
t f''(t) + (-k+1-t) f'(t) = -l f(t),
\end{equation*}
we have
\begin{align*}
\omega_\psi(X_+ &- X_-)(f(4 \pi z \overline z) z^{-k} e^{-2 \pi z \overline z}) \\
&= i \left[(-k+1) f(4 \pi z \overline z) - 2((-k+1-4 \pi z \overline z) f'(4 \pi z \overline z) + 4 \pi z \overline z f''(4 \pi z \overline z))\right] z^{-k} e^{-2 \pi z \overline z} \\
&= i (-k+1+2l) f(4 \pi z \overline z) z^{-k} e^{-2 \pi z \overline z}.
\end{align*}
By Lemma \ref{l:1F1 props}(a), ${}_1 F_1(-l,-k+1,t)$ is such an $f(t)$, and so the desired conclusion follows. 

Finally, it is easy to see that $\omega_\psi(\alpha,1) \phi_{k,l}' = \xi(\alpha^{-1}) \alpha^{-k} \phi_{k,l}',$ and it follows that
\begin{equation*}
\omega_\psi(\alpha, r(\theta))\phi_{k,l}' = \xi(\alpha^{-1}) \alpha^{-k} e^{-(|k|+1+2l)\theta} \phi_{k,l}'. \qedhere
\end{equation*}
\end{proof}

The following lemma is useful in understanding the relationship between the $\phi_{k,l}'$ with respect to the Maass--Shimura operator on modular forms.

\begin{lemma}\label{l:schwartz maass shimura}
Write $z = x + y i \in \CC$. For $y \neq 0$, we have
\begin{equation*}
\delta_{|k|+1}^l\left(y^{1/2} e^{2 \pi i x v \overline v} \phi_{k,0}'(v \sqrt{y})\right)
= \frac{(|k|+1)_l}{(2\pi i)^l (2i)^l} \cdot \left(y^{-l+1/2} e^{2 \pi i x v \overline v} \phi_{k,l}'(v \sqrt{y})\right).
\end{equation*}
\end{lemma}

\begin{proof}
This amounts to showing
\begin{equation*}
\left(\frac{\partial}{\partial z} + \frac{|k|+1+2l}{z - \overline z}\right)\left[y^{-l+1/2} e^{2 \pi i x v \overline v} \phi_{k,0}'(v \sqrt{y})\right] 
= \frac{|k|+1+l}{2i} \cdot \left(y^{-l-1+1/2} e^{2 \pi i x v \overline v} \phi_{k,l+1}'(v \sqrt{y})\right).
\end{equation*}
Unwinding definitions, this amounts to showing
\begin{align*}
\left(\frac{\partial}{\partial z} + \frac{|k|+1+2l}{z - \overline z}\right)&\left(y^{-l} {}_1F_1(-l,|k|+1,4 \pi v \overline v y) e^{2 \pi i v \overline v z}\right) \\
&=
\frac{|k|+1+l}{2i} \cdot \left(y^{-l-1} {}_1F_1(-l-1,|k|+1,4 \pi v \overline v y) e^{2 \pi i v \overline v z}\right).
\end{align*}
Verifying this is a straightforward calculation. For example, the coefficient of $y^{-l-1}$ on the left-hand side is equal to $\left(\frac{-l}{2i} + \frac{|k|+1+2l}{2i}\right) \cdot e^{2 \pi i v \overline v z}$, and this agrees with the right-hand side.
\end{proof}

\subsubsection{Finite nonsplit places}

In this section, let $v$ be a finite nonsplit place of $F$ lying under a single prime $w$ of $E$. Then $E_w$ is a field and $E_w/F_v$ is either unramified or ramified. Assume that $E_w, F_v$ have odd residue characteristic. We drop the subscripts $w$ and $v$ throughout this section.

\begin{definition}
Define
\begin{equation*}
\phi'(x) \colonequals \begin{cases}
\ONE_{\cO_E}(x) & \text{if $\chi$ is unramified,} \\
\chi(x) \ONE_{\cO_E^\times}(x) & \text{otherwise.}
\end{cases}
\end{equation*}
\end{definition}

\begin{lemma}\label{l:nonsplit hecke unram}
Let $\psi'$ be an unramified nontrivial additive character of $F$. For $h \in \cO_E^\times$ and $g = \left(\begin{smallmatrix} a & b \\ c & d \end{smallmatrix}\right) \in K_0 \colonequals K_0(c(\rho_\chi))$ such that $\Nm(h) = \det(g)$, we have
\begin{equation*}
\omega_{\psi'}(h,g) \phi' =  (\chi \xi)^{-1}(h) \cdot (\chi \epsilon_{E/F})(a) \cdot \phi'.
\end{equation*}
\end{lemma}

\begin{proof}
See \cite[Proposition 2.2.4]{X07}, \cite[Proposition 2.5.1]{P06}, \cite[Lemma 8.6]{Ch18}.
\end{proof}

\begin{lemma}\label{l:nonsplit hecke}
For $h \in \cO_E^\times$ and $g = \left(\begin{smallmatrix} a & b \\ c & d \end{smallmatrix}\right)$ such that $\Nm(h) = \det(g)$, we have
\begin{equation*}
\omega_\psi(h,d(\delta)^{-1} g d(\delta)) \phi' = (\chi \xi)^{-1}(h) \cdot (\chi \epsilon_{E/F})(a) \cdot \phi'.
\end{equation*}
\end{lemma}

\begin{proof}
Since $\psi$ has conductor $\delta$, $\psi'(x) \colonequals \psi(\delta x)$ is an unramified nontrivial additive character of $F$. The conclusion follows by Equation \eqref{e:change psi} and Lemma \ref{l:nonsplit hecke unram}.
\end{proof}

\subsubsection{Finite split places}\label{s:schwartz finite split}

In this section we let $v$ be a finite split place of $F$. Then $E_v \cong F_v \oplus F_v$. We drop the subscript $v$ throughout this section.

\begin{definition}
For a character $\chi = \chi_1 \otimes \chi_2 \from F^\times \times F^\times \to \CC^\times$, define
\begin{equation*}
\phi'(x_1, x_2) \colonequals 
\begin{cases}
\ONE_{\cO_F}(x_1) \ONE_{\cO_F}(x_2) & \text{if $\chi$ is unramified,} \\
\chi(x_1,x_2) \ONE_{\cO_F^\times}(x_1) \ONE_{\cO_F^\times}(x_2) & \text{otherwise.}
\end{cases}
\end{equation*}
\end{definition}

\begin{lemma}\label{l:split hecke unram}
Let $\psi'$ be an unramified nontrivial additive character of $F$. For $h \in \cO_F^\times \times \cO_F^\times$ and $g = \left(\begin{smallmatrix} a & b \\ c & d \end{smallmatrix}\right) \in K_0$ with $\Nm(h) = \det(g)$, we have
\begin{equation*}
\omega_{\psi'}(h,g) \phi' = (\chi \xi)^{-1}(h) \cdot \chi_1(a) \chi_2(a) \cdot \phi'.
\end{equation*}
\end{lemma}

\begin{proof}
See \cite[Proposition 2.2.4]{X07}, \cite[Proposition 2.5.1]{P06}, \cite[Lemma 8.9]{Ch18}.
\end{proof}

By the same argument as in Lemma \ref{l:nonsplit hecke},

\begin{lemma}\label{l:split hecke}
For $h \in \cO_F^\times \times \cO_F^\times$ and $g = \left(\begin{smallmatrix} a & b \\ c & d \end{smallmatrix}\right) \in K_0$ with $\Nm(h) = \det(g)$, we have
\begin{equation*}
\omega_{\psi}(h,d(\delta)^{-1} g d(\delta)) \phi' = (\chi \xi)^{-1}(h) \cdot \chi_1(a) \chi_2(a) \cdot \phi'.
\end{equation*}
\end{lemma}

\subsection{Local zeta integrals}\label{ss:local}

In this section, we calculate the local zeta integrals $Z(\tfrac{1}{2}, \Phi_v, \chi_v)$ for the Siegel--Weil section $\Phi_v = \Phi_v^{\O,\Sp}(\delta(\phi_v' \otimes \overline {\phi_v'}))$ (see Section \ref{s:rallis}).

\subsubsection{Infinite nonsplit places}

Let $v$ be an infinite nonsplit place. We say that $\chi_v$ has infinity type $(k_1, k_2)$ if
\begin{equation*}
\chi_v \from \CC^\times \to \CC^\times, \qquad z \mapsto z^{-k_1} \overline z^{-k_2}.
\end{equation*}
Assume that $\chi_v(z) = z^k$ for $z \in \CC^1$, so that either $\chi_v$ is of type $(-k+j,j)$ or $(-j,k-j)$ for some integer $j$. Pick an integer $l \in \bZ_{\geq 0}$ and take 
\begin{equation*}
\phi_v'(z) \colonequals \phi_{k,l}'(z) = \begin{cases}
{}_1 F_1(-l, k+1,4 \pi z \overline z) \overline z^k e^{-2 \pi z \overline z} & \text{if $k \geq 0$,} \\
{}_1 F_1(-l,-k+1,4 \pi z \overline z) z^{-k} e^{-2 \pi z \overline z} & \text{if $k < 0$.}
\end{cases}
\end{equation*}

\begin{lemma}
Let $v$ be an infinite nonsplit place. Then
\begin{equation*}
Z_v(\tfrac{1}{2}, \Phi_v, \chi_v) = \vol(\CC^1) \langle \phi', \phi' \rangle = \frac{(2\pi)^2}{4^{|k|+1} \pi^{|k|+1}} \cdot \frac{l! (|k|)!^2}{(l+|k|)!}.
\end{equation*}
\end{lemma}

\begin{proof}
By Lemma \ref{l:infinite hecke}, $\omega_\psi(\alpha, 1) \phi_v' = \xi(\alpha^{-1}) \alpha^{-k} \phi_v'.$ Thus
\begin{align*}
Z_v(\tfrac{1}{2}, \Phi_v, \chi_v)
&= \int_{\CC^1} \langle \omega_\psi(g,1) \phi_v', \phi_v' \rangle (\chi_v \xi_v)(g) \, dg = \vol(\CC^1) \langle \phi_v', \phi_v' \rangle = \pi^{-1} \langle \phi_v', \phi_v' \rangle. 
\end{align*}
We have
\begin{align*}
\langle \phi_v', \phi_v' \rangle
&= \int_\CC {}_1 F_1(-l, |k|+1, 4 \pi z \overline z)^2 \cdot (z \overline z)^{|k|} \cdot e^{-4 \pi z \overline z} \, dz \, d\overline z \\
&= \frac{2 \pi}{(4 \pi)(4 \pi)^{|k|}} \int_0^\infty {}_1 F_1(-l, |k|+1, t)^2 \cdot t^{|k|} \cdot e^{-t} \, dt  \\
&= \frac{2 \pi}{(4 \pi)^{|k|+1}} \frac{l! (|k|)!^2}{(l+|k|)!} = \frac{2 \pi}{(4 \pi)^{|k|+1}} \frac{(|k|)!}{\binom{l+|k|}{|k|}}. \qedhere
\end{align*}
\end{proof}

\subsubsection{Finite nonsplit places}

Recall from Section \ref{ch:schwartz} that we set
\begin{equation*}
\phi_v'(x) = \begin{cases}
\ONE_{\cO_{E_v}}(x) & \text{if $\chi_v$ is unramified,} \\
\chi_v(x) \ONE_{\cO_{E_v}^\times}(x) & \text{if $\chi_v$ is ramified.}
\end{cases}
\end{equation*}

\begin{lemma}
Let $v$ be a finite nonsplit place. If $E_v/F_v$ is unramified, then
\begin{equation*}
Z_v(\tfrac{1}{2}, \Phi_v, \chi_v) =
\begin{cases}
q_v^{-d_v/2} & \text{if $E_v/F_v$ is unramified and $\chi_v$ is unramified,} \\
q_v^{-d_v/2} (1 - q_v^{-2}) & \text{if $E_v/F_v$ is unramified and $\chi_v$ is ramified,} \\
q_{v}^{-1} q_v^{-d_v/2} & \text{if $E_v/F_v$ is ramified and $\chi_v$ is unramified,} \\
q_{v}^{-1}q_v^{-d_v/2}(1 - q_{v}^{-1})  & \text{if $E_v/F_v$ is ramified and $\chi_v$ is ramified.}
\end{cases}
\end{equation*}
\end{lemma}

\begin{proof}
By Lemma \ref{l:nonsplit hecke}, for $g \in E_v^1$, we have $\omega_\psi(g,1) \phi' = (\chi_v \xi_v)^{-1}(g) \cdot \phi'.$ This implies that
\begin{align*}
Z_v(\tfrac{1}{2}, \Phi_v, \chi_v)
&= \vol(E_v^1,d^1 x_v^{\Tam}) \int_{E_v^1} \langle \omega_\psi(g,1) \phi', \phi' \rangle (\chi \xi_{\bY'})_v(g) \, dg \\
&= \vol(E_v^1, d^1 x_v^{\Tam})^2 \langle \phi', \phi' \rangle 
= \begin{cases}
\vol(E_v^1, d^1 x_v^{\Tam})^2 \vol(\cO_{E_v},dx_v) & \text{if $\chi_v$ is unram,} \\
\vol(E_v^1, d^1 x_v^{\Tam})^2 \vol(\cO_{E_v}^\times,dx_v) & \text{otherwise.} 
\end{cases}
\end{align*}
The desired conclusion now follows from the measures in Section \eqref{s:tamagawa}.
\end{proof}

\subsubsection{Finite split places}

Let $v$ be a finite split place and write $\chi_v = \chi_{1,v} \otimes \chi_{2,v} \from F_v^\times \times F_v^\times \to \CC^\times.$ Recall that
\begin{equation*}
\phi'(x_1,x_2) \colonequals
\begin{cases}
\ONE_{\cO_{F_v}}(x_1) \ONE_{\cO_{F_v}}(x_2) & \text{if $\chi_v$ is unramified,} \\
\chi_v(x_1,x_2) \ONE_{\cO_{F_v}^\times}(x_1) \ONE_{\cO_{F_v}^\times}(x_2) & \text{otherwise.}
\end{cases}
\end{equation*}

\begin{lemma}
Let $v$ be a finite split place and assume that $\chi_v$ is unramified. Then
\begin{equation*}
Z_v(\tfrac{1}{2}, \Phi_v, \chi_v) = q_v^{-3d_v/2} \cdot \frac{L_v(1,\chi_{1,v} \otimes \chi_{2,v}^{-1}) L_v(1,\chi_{1,v}^{-1} \otimes \chi_{2,v})}{L_v(2, \varepsilon_{E/F})}.
\end{equation*}
\end{lemma}

\begin{proof}
In this setting, $E_v^1 = \{(a, a^{-1}) \in F_v^\times \times F_v^\times\}$. By Lemma \ref{l:bs D},
\begin{equation*}
\omega_\psi((a,a^{-1}), 1) \phi'(x_1, x_2) 
= \xi_v(a,a^{-1})^{-1} \phi'(x_1 a^{-1}, x_2a) 
= \xi_v(a, a^{-1})^{-1} \ONE_{a \cO_{F_v}}(x_1) \ONE_{a^{-1} \cO_{F_v}}(x_2).
\end{equation*}
Hence
\begin{align*}
\langle \omega_\psi((a,a^{-1}), 1) \phi', \phi' \rangle
&= \int_{\bX_v'} \xi_v(a,a^{-1})^{-1} \ONE_{a\cO_{F_v}}(x_1) \ONE_{a^{-1} \cO_{F_v}}(x_2) \ONE_{\cO_{F_v}}(x_1) \ONE_{\cO_{F_v}}(x_2) \, dx_1 \, dx_2 \\
&= \xi_v(a,a^{-1})^{-1} \vol(a \cO_{F_v} \cap \cO_{F_v},dx_v) \vol(a^{-1} \cO_{F_v} \cap \cO_{F_v},dx_v) \\
&= \xi_v(a,a^{-1})^{-1} \tfrac{1}{q_v^{|\val(a)|}} \vol(\cO_{F_v},dx_v)^2 = \xi_v(a,a^{-1})^{-1} \tfrac{1}{q_{v}^{|\val(a)|}} q_v^{-d_v}.
\end{align*}
We therefore have, writing $\pi = \pi_v$ for a uniformizer of $F_v$,
\begin{align*}
Z_v(\tfrac{1}{2}, \Phi_v, \chi_v)
&= \int_{F_v^\times} \langle \omega_\psi(a, a^{-1}) \phi_v', \phi_v' \rangle \xi_v(a,a^{-1}) \chi_v(a,a^{-1}) \, da \\
&= \sum_{n \in \bZ} \int_{\cO_{F_v}^\times} \langle \omega_\psi(\pi^n a, \pi^{-n} a^{-1}) \phi_v', \phi_v' \rangle \xi_v(\pi^n a, \pi^{-n} a^{-1}) \chi_v(\pi^n a, \pi^{-n} a^{-1}) \, da \\
&= q_v^{-3d_v/2} \cdot \frac{1-q_v^{-2}}{(1-q_v^{-1}\chi_v(\pi^{-1},\pi))(1-q_v^{-1}\chi_v(\pi,\pi^{-1}))} \\
&= q_v^{-3d_v/2} \cdot \frac{L_v(1,\chi_{1,v} \otimes \chi_{2,v}^{-1}) L_v(1,\chi_{1,v}^{-1} \otimes \chi_{2,v})}{L_v(2, \varepsilon_{E/F})}. \qedhere
\end{align*}
\end{proof}

\begin{lemma}
Let $v$ be a finite split place and assume that $\chi_v$ is ramified. Then
\begin{equation*}
Z_v(\tfrac{1}{2}, \Phi_v, \chi_v) = q_v^{-3d_v/2}(1-q_v^{-1})^2.
\end{equation*}
\end{lemma}

\begin{proof}
We have $\omega_\psi((a,a^{-1}), 1) \phi'(x_1, x_2) = \xi_v(a,a^{-1})^{-1} \chi_v(a,a^{-1})^{-1} \ONE_{a \cO_{F_v}^\times}(x_1) \ONE_{a^{-1} \cO_{F_v}^\times}(x_2)$ so that
\begin{equation*}
\langle \omega_\psi((a,a^{-1}),1) \phi', \phi' \rangle = \xi_v(a,a^{-1})^{-1} \chi_v(a,a^{-1})^{-1} \vol(\cO_{F_v}^\times,dx_v)^2 \ONE_{\cO_{F_v}^\times}(a).
\end{equation*}
Thus $Z_v(\tfrac{1}{2}, \Phi_v, \chi_v)  = \vol(\cO_{F_v}^\times, dx_v)^2 \vol(\cO_{F_v}^\times,d^1 x_v^{\Tam})$.
\end{proof}

\section{An explicit Rallis inner product formula}\label{ch:explicit rallis}

Let $F$ be a totally real number field and let $E/F$ be a CM extension. Let $\eta_1,\ldots, \eta_n$ be the real embeddings of $F$. Let $\chi \from E^\times \backslash \bA_E^\times \to \CC^\times$ be a Hecke character of infinity type $(k+j,j)$ where $k = (k_1, \ldots, k_n), j = (j_1, \ldots, j_n) \in \bZ^n$. Assume that $B = M_2(F)$ and let $W_0 = \Res_{B/E} B = W_1 + W_2$ be a decomposition of the $E$-space $W_0$ into totally isotropic subspaces. Set $\bX' = \Res_{E/F}(E \otimes W_1), \bY' = \Res_{E/F}(E \otimes W_2)$, and define a Schwartz function $\phi' = \otimes_v \phi_v' \in \cS(\bX'(\bA))$ as in Section \ref{ch:schwartz}:
\begin{equation*}
\phi_{l,v}'(z) \colonequals \begin{cases}
{}_1 F_1(-l_i, k_i+1, 4 \pi z \overline z) \overline z^k e^{-2 \pi z \overline z} & \text{if $v = \eta_i \mid \infty$ and $k \geq 0$,} \\
{}_1 F_1(-l_i, -k_i+1, 4 \pi z \overline z) z^{-k} e^{-2 \pi z \overline z} & \text{if $v = \eta_i \mid \infty$ and $k < 0$,} \\
\ONE_{\cO_{E_v}}(z) & \text{if $v$ is nonsplit and $\chi_v$ is unramified,} \\
\chi_v(z) \ONE_{\cO_{E_v}^\times}(z) & \text{if $v$ is nonsplit and $\chi_v$ is ramified,} \\
\ONE_{\cO_{F_v}}(z_1) \ONE_{\cO_{F_v}}(z_2) & \text{if $v$ splits and $\chi_v$ is unramified,} \\
\chi_v(z_1, z_2)^{-1} \ONE_{\cO_{F_v}^\times}(z_1) \ONE_{\cO_{F_v}^\times}(z_2) & \text{if $v$ splits and $\chi_v$ is ramified.}
\end{cases}
\end{equation*}
Define
\begin{equation*}
\Sigma_\chi \colonequals \{v : \text{$\chi_v$ is unramified}\}, \qquad
\Sigma_{\widetilde \chi} \colonequals \{v : \text{$\widetilde \chi_v$ is unramified}\}.
\end{equation*}
For each place $v$ of $F$, define
\begin{equation*}
C_v
\colonequals \begin{cases}
\frac{(2\pi)^2}{4^{|k_i|+1} \pi^{|k_i|+1}} \cdot \frac{l_i! (|k_i|)!^2}{(l_i+|k_i|)!} & \text{if $v = \eta_i \mid \infty$} \\
q_v^{-d_v/2} & \text{if $v \notin \Sigma_\chi$, $v \notin \Sigma_{\widetilde \chi}$, $v$ unram} \\
q_v^{-d_v/2}(1 - q_v^{-2}) & \text{if $v \in \Sigma_\chi$, $v \notin \Sigma_{\widetilde \chi}$, $v$ unram} \\
q_v^{-d_v/2} & \text{if $v \in \Sigma_\chi$, $v \in \Sigma_{\widetilde \chi}$, $v$ unram} \\
q_v^{-d_v/2} q_v^{-1}(1-q_v^{-2})^{-1}(1-\widetilde \chi_w(\pi_w) q_v^{-1})  & \text{if $v \notin \Sigma_\chi$, $v \notin \Sigma_{\widetilde \chi}$, $v$ ram} \\
q_v^{-d_v/2} q_v^{-1}(1-q_v^{-1})(1-q_v^{-2})^{-1}(1 - \widetilde \chi_v(\pi_v) q_v^{-1}) & \text{if $v \in \Sigma_\chi$, $v \notin \Sigma_{\widetilde \chi}$, $v$ ram} \\
q_v^{-d_v/2} q_v^{-1}(1-q_v^{-1})(1-q_v^{-2})^{-1} & \text{if $v \in \Sigma_\chi$, $v \in \Sigma_{\widetilde \chi}$, $v$ ram} \\
q_v^{-3d_v/2} & \text{if $v \notin \Sigma_\chi$, $v \notin \Sigma_{\widetilde \chi}$, $v$ split} \\
q_v^{-3d_v/2}\frac{(1 - (\chi_{1,v}\chi_{2,v}^{-1})(\pi_v)q_v^{-1})(1 - (\chi_{1,v}^{-1}\chi_{2,v})(\pi_v) q_v^{-1})}{(1+q_v^{-1})} & \text{if $v \in \Sigma_\chi$, $v \notin \Sigma_{\widetilde \chi}$, $v$ split} \\
q_v^{-3d_v/2} (1-q_v^{-1})(1+q_v^{-1})^{-1} & \text{if $v \in \Sigma_\chi$, $v \in \Sigma_{\widetilde \chi}$, $v$ split}
\end{cases}
\end{equation*}

\begin{theorem}\label{t:explicit rallis}
The Petersson inner product of the theta lift $\theta_{\phi'}(\chi \xi)$ is
\begin{equation*}
\langle \theta_{\phi'}(\chi \xi), \theta_{\phi'}(\chi \xi) \rangle = \frac{\rho_F}{\rho_E} \cdot
\frac{L(1,\widetilde \chi)}{\zeta(2)} \cdot \prod_v C_v,
\end{equation*}
where $C_v = 1$ at all but finitely many places. In particular, $\theta_{\phi'}(\xi \chi) \neq 0$ if $\chi$ is nontrivial on $\bA_E^1$.
\end{theorem}

\begin{proof}
By the results of Section \ref{ss:local}, it is a straightforward comparison to see that
\begin{equation*}
Z(\tfrac{1}{2}, \Phi_v, \chi_v) = C_v \cdot \frac{L_v(1, \widetilde \chi)}{\zeta_v(2)} \qquad \text{for all places $v$ of $F$}.
\end{equation*}
Since all but finitely many places simultaneously satisfy the conditions $d_v = 0$, $v \notin \Sigma_\chi$, $v \notin \Sigma_{\widetilde \chi}$, and $v$ is split or unramified, we see that $C_v = 1$ for all but finitely many places, and the desired equation now follows from the doubling method. Observe that the factor $\rho_F/\rho_E$ comes from the fact definition of the Tamagawa measure on $\bA_E^1$ and the local measures on $E_v^1$ (Section \ref{s:tamagawa}).

Finally, since $C_v \neq 0$ for all $v$, it follows that $\theta_{\phi'}(\chi \xi) \neq 0$ if and only if $L(1, \widetilde \chi) \neq 0$. But $L(1, \widetilde \chi) \neq 0$ if and only if $\chi$ is trivial on $\bA_E^1$, so the final assertion holds.
\end{proof}

The Shimura--Maass differential operator 
\begin{equation*}
\delta_k \colonequals \frac{1}{2\pi i}\left(\frac{\partial}{\partial z} + \frac{k}{z - \overline z}\right)
\end{equation*}
maps real analytic modular forms of weight $k$ to real analytic modular forms of weight $k+2$. Define the composite operator
\begin{equation}\label{e:maass}
\delta_k^l \colonequals \delta_{k + 2l} \circ \cdots \circ \delta_{k+2} \circ \delta_k
\end{equation}
mapping real analytic modular forms of weight $k$ to real analytic modular forms of weight $k + 2l$. 

Let $f_\chi$ be the normalized newform of weight $|k| + 1 = (|k_1| + 1, \ldots, |k_n| + 1)$ in $\pi_\chi$. For $l = (l_1, \ldots, l_n)$, let $F_\chi^l$ denote the automorphic form on $\GL_2(\bA_F)$ corresponding to $\delta_{|k|+1}^l f_\chi$.


\begin{theorem}\label{t:newform}
If $\chi$ does not factor through the norm map $\bA_E^\times \to \bA_F^\times$, we have
\begin{equation*} 
\theta_{\phi_l'}(\chi \xi) = D_l \cdot F_\chi^l, \qquad \text{for some $D_l \neq 0$.}
\end{equation*}
\end{theorem}

\begin{proof}
First recall that by Theorem \ref{t:global lift}(a), the theta lift $\theta_{\phi'}(\chi \xi)$ is an automorphic form in the automorphic induction $\pi_\chi$ to $\GL_2(\bA_F)$. If $f$ is a Hecke eigenform of weight $|k|+1+2l$ in $\pi_\chi$, then it must satisfy that for all $r(\theta) \colonequals r(\theta_1) \cdots r(\theta_n)$ with $r(\theta_j) \in \SO(2)$ and $k_0 = \left(\begin{smallmatrix} a & b \\ c & d \end{smallmatrix}\right) \in K_0 \colonequals \prod_{v \nmid \infty} K_{0,v}$ with $\det(k_0) = 1$, we have
\begin{equation}\label{e:eigen criterion}
f(g r(\theta) d(\frak d)^{-1} k_0 d(\frak d)) = \prod_{j=1}^n e^{i(|k_j|+1+2l_j)\theta_j} (\chi\epsilon_{E/F})(a) f(g) \qquad \text{for all $g \in \GL_2(\bA_F)$}.
\end{equation}
By Casselman's theorem \cite[Theorem 1]{Ca73}, the dimension of automorphic forms satisfying \eqref{e:eigen criterion} must have dimension $1$. 
Therefore to see that $\theta_{\phi'}(\chi \xi)$ is a (possibly zero!) multiple of $F_\chi^l$, we need only see that it satisfies \eqref{e:eigen criterion}.

We first recall the definition of the theta lift $\theta_{\phi'}(\chi \xi)$ on $\GL_2(\bA_F)$. If $g \in \GL_2(\bA_F)^+ \colonequals \{g \in \GL_2(\bA_F) : \det(g) \in \Nm(\bA_E^\times)\}$, then for any $h \in \bA_E^\times$ such that $\det(g) = \Nm(h)$,
\begin{equation*}
\theta_{\phi'}(\chi \xi)(g) = \int_{[E^1]} \Theta(\omega_\psi(hh_1,g) \phi') \cdot (\chi\xi)(h h_1) \, dh_1.
\end{equation*}
We define $\theta_{\phi'}(\chi \xi)$ on $\GL_2(F) \GL_2(\bA_F)^+ = \left\{g \in \GL_2(\bA_F) : \det(g) \in F^\times \Nm(\bA_E^\times)\right\}$ by
\begin{equation*}
\theta_{\phi'}(\chi \xi)(\gamma g) = \theta_{\phi'}(\chi \xi)(g), \qquad \text{for $\gamma \in \GL_2(F), g \in \GL_2(\bA_F)^+$}.
\end{equation*}
Note that $\GL_2(F) \GL_2(\bA_F)^+$ is an index-2 subgroup of $\GL_2(\bA_F)$. We define $\theta_{\phi'}(\chi \xi)$ on $\GL_2(\bA_F)$ by extending by 0 outside $\GL_2(F) \GL_2(\bA_F)$. Define $K_0 \colonequals \prod_v K_{0,v}$, where $K_{0,v} \subset \GL_2(\cO_{F_v})$ as defined in Section \ref{ch:schwartz}. Note that $K_0 \subset \GL_2(F) \GL_2(\bA_F)^+$. By Lemmas \ref{l:infinite hecke}, \ref{l:nonsplit hecke}, and \ref{l:split hecke}, for $r(\theta) = r(\theta_1) \cdots r(\theta_n)$ with $r(\theta_j) \in \SO(2)$ and $k_0 = \left(\begin{smallmatrix} a & b \\ c & d \end{smallmatrix}\right) \in K_0 \cap \GL_2(\bA_F)^+$, 
\begin{equation*}
\omega_\psi(h_0, r(\theta) d(\frak d)^{-1} k_0 d(\frak d)) \phi_l' = \prod_{j=1}^n e^{i(|k_j|+1+2l_j)\theta_j} (\chi \xi)^{-1}(h h_0) (\chi \epsilon_{E/F})(a) \phi_l',
\end{equation*}
where $h_0 \in \bA_E^\times$ is such that $\Nm(h_0) = \det(k_0)$. This implies that for any $g \in \GL_2(\bA_F)^+$ and any $h \in \bA_E^\times$ with $\Nm(h) = \det(g)$,
\begin{align*}
\theta_{\phi_l'}&(\chi \xi)(g r(\theta) d(\frak d)^{-1} k_0 d(\frak d)) \\
&= \int_{[E^1]} \Theta(\omega_\psi(h h_1 h_0, g r(\theta) d(\frak d)^{-1} k_0 d(\frak d)) \phi_l') \cdot (\chi \xi)(h h_1 h_0) \, dh_1 \\
&= \prod_{j=1}^n \int_{[E^1]} \Theta(\omega_\psi(h h_1, g) \phi_l') \cdot e^{i(|k_j|+1+2l_j)\theta_j} \cdot (\chi \xi)^{-1}(h_0) \cdot (\chi \epsilon_{E/F})(a) \cdot (\chi \xi)(h h_1 h_0) \, dh_1 \\
&= \prod_{j=1}^n e^{i(|k_j|+1+2l_j)\theta_j} (\chi \epsilon_{E/F})(a) \cdot \int_{[E^1]} \Theta(\omega_\psi(h h_1, g) \phi_l') \cdot (\chi \xi)(h h_1) \, dh_1 \\
&= \prod_{j=1}^n e^{i(|k_j|+1+2l_j)\theta_j} (\chi \epsilon_{E/F})(a) \cdot \theta_{\phi_l'}(\chi \xi)(g). \qedhere
\end{align*}
\end{proof}

In the next result, we give an exact formula (up to $\CC^1$) for the constant $D_l$ in the case that $F = \bQ$. One can do this for the general case by comparing Theorem \ref{t:explicit rallis} to known formulas for the Petersson inner product of Hilbert modular forms, but the formula for $D_l$ will be more complicated. In the following, we use \cite[Equation (2.5)]{S76}, \cite[Section 5]{H81} together with the factors at bad places as determined in \cite[Section 4.2]{Co18}.

Let $g$ be a twist of $f_\chi$ which is twist-minimal by $\chi_g$. Let $N$, $N_g,$ be the levels of $f_\chi,$ $g$, and let $N_{\chi, g}$ be the conductor of $\chi_g$. For every prime $p$, let $p^{r_g}$ be the exact power of $p$ dividing $N_g$, and $p^{r_{\chi_g}}$ be the exact power of $p$ dividing $N_{g,\chi}$. We denote by $L(s, \ad, f_\chi)$ the adjoint $L$-function of $f_\chi$. Define
\begin{equation*}
\ell_p \colonequals
\begin{cases}
1 & \text{if $p \nmid N$,} \\
(1 + 1/p) L_p(\ad, f, 1) & \text{if $p \nmid N_g$ and $p \mid N$,} \\
(1 + 1/p) & \text{if $p \mid\mid N_g$ and $p \mid\mid N$,} \\
(1 + 1/p)(1-1/p^2)^{-1} & \text{if $p \mid\mid N_g$ and $p^2 \mid N$,} \\
(1 + 1/p) & \text{if $r_g = r_{\chi_g} \geq 1$ and $p^{r_g} \mid\mid N$}, \\
(1 + 1/p)(1 - 1/p)^{-1} & \text{if $r_g = r_{\chi_g} \geq 1$ and $p^{r_g+1} \mid N$,} \\
1 & \text{if $r_g \geq 2$ and $r > r_{\chi_g}$.} \\
\end{cases}
\end{equation*}
Note that, comparing to \cite[Section 4.2]{Co18}, the last case comes from the fact that $\pi_\chi \cong \pi_\chi \otimes \det(\varepsilon_{E/F})$.

\begin{theorem}\label{t:Dl}
Assume $F = \bQ$ and let $K_0$ is any maximal compact subgroup of $\GL_2(\bA_{\bQ, \rm fin})$ containing $K \colonequals \prod_v, K_{0,v}(c_v(\pi_\chi))$ (Sections \ref{s:conductor}, \ref{ch:schwartz}). Then
\begin{equation*}
|D_l|^2 
=
\left|\frac{(2\pi i)^l (2i)^l}{(|k|+1)_l}\right|^2 \cdot \rho_E^{-2} \cdot \prod_v C_v \cdot \zeta(2) \cdot (2\pi)^{-1} \cdot 2 \cdot \frac{[K_0 : K]}{\frac{\pi}{3} \cdot [\PSL_2(\bZ) : \Gamma_1(N)]} \cdot  \frac{(4\pi)^{|k|+1}}{|k|!} \cdot \prod_p \ell_p.
\end{equation*}
In particular, $|D_l| \sim \pi^l$ and, up to an element of $\CC^1$, $\theta_{\phi_0'}(\chi \xi)$ is an algebraic holomorphic Hecke eigenform of weight $k+1$ and level $c(\chi)$.
\end{theorem}

\begin{proof}
By Lemma \ref{l:schwartz maass shimura}, we have
\begin{equation*}
\theta_{\phi_l'}(\chi \xi) = \frac{(2\pi i)^l (2i)^l}{(|k|+1)_l} \cdot \delta_{|k|+1}^l\left(\theta_{\phi_0'}(\chi \xi)\right).
\end{equation*}
It therefore suffices to calculate $|D_0|^2 = \langle \theta_{\phi_0'}(\chi \xi), \theta_{\phi_0'}(\chi \xi) \rangle/\langle F_\chi, F_\chi \rangle.$ 

Following \cite[Lemmas 6.1, 6.3]{IP16a}, we have
\begin{equation*}
\langle F_\chi, F_\chi \rangle = (2 \pi) \cdot \zeta(2)^{-1} \cdot 2^{-1} \cdot [K_0 : K]^{-1} \cdot \frac{\pi}{3}[\PSL_2(\bZ) : \Gamma_1(N)] \cdot \langle f_\chi, f_\chi \rangle,
\end{equation*}
where $\langle f_\chi, f_\chi \rangle$ is normalized as in \cite[Eq.\ (2.1)]{S76}: for any cusp form $f$ of weight $\kappa$ and level $\Gamma$, set
\begin{equation*}
\langle f, f \rangle \colonequals \frac{1}{\vol(\frak h/\Gamma)} \int_{\frak h/\Gamma} f(x + i y) \overline{f(x + i y)} y^k \frac{dx \, dy}{y^2}.
\end{equation*}
By \cite[Eq.\ (2.5)]{S76}, \cite[Section 5]{H81}, \cite[Section 4.2]{Co18},
\begin{equation*}
\langle f_\chi, f_\chi \rangle = \zeta(2)^{-1} \cdot \frac{|k|!}{(4\pi)^{|k|+1}} \cdot \prod_{p \mid N} \ell_p^{-1} \cdot L(1, \ad, f_\chi).
\end{equation*}
We have $L(1, \ad, f_\chi) = L(1, \widetilde \chi) \cdot L(1, \varepsilon_{E/F}) = L(1, \widetilde \chi) \cdot \rho_E.$ 
Therefore, by Theorem \ref{t:explicit rallis},
\begin{equation*}
|D_0|^2 
= \rho_E^{-2} \cdot \prod_v C_v \cdot \zeta(2) \cdot (2\pi)^{-1} \cdot 2 \cdot \frac{[K_0 : K]}{\frac{\pi}{3} \cdot [\PSL_2(\bZ) : \Gamma_1(N)]} \cdot  \frac{(4\pi)^{|k|+1}}{|k|!} \cdot \prod_p \ell_p.
\end{equation*}
Since $E$ is CM by construction so that $\rho_E \sim \pi$ and since $C_\infty \sim \pi^2/\pi^{|k|+1}$, we see:
\begin{equation*}
|D_0|^2 \sim \pi^{-2} \cdot \pi^{2} \cdot \pi^2 \cdot \pi^{-1} \cdot \pi^{-1} \in \overline \QQ. \qedhere
\end{equation*}
\end{proof}

\section{An example: the canonical Hecke character for $\bQ(\sqrt{-7})$}\label{ch:example}

Let $F = \bQ$ and let $E = \bQ(\sqrt{-7})$. Then $E$ has class number $1$ and there is a unique \textit{canonical character} $\chi_{\rm can}'$ in the sense of Rohrlich \cite{Ro80} (see also \cite[Page 52]{Yang thesis}). Explicitly, $\chi_{\rm can}'$ can be described as follows. First consider the character
\begin{equation*}
\epsilon \from \cO_E/(\sqrt{-7}) \cong \bZ/7\bZ \stackrel{\left(\frac{\cdot}{7}\right)}{\longrightarrow} \{\pm 1\}.
\end{equation*}
Then $\epsilon(-1) = -1$ and hence the map on principal ideals
\begin{equation*}
P(\sqrt{-7}) = \{\alpha \cO_E : \text{$\alpha \in E^\times$ is relatively prime to $7$}\} \to E^\times, \qquad \alpha \cO_E \mapsto \epsilon(\alpha) \alpha
\end{equation*}
is a well-defined homomorphism. Since $E$ has class number $1$, then $P(\sqrt{-7}) = I(\sqrt{-7})$, and the above defines a Hecke character of $E^\times$. We define $\chi_{\rm can} \colonequals \chi_{\rm can}' \cdot || \cdot ||_{\bA_K}^{1/2}$ to be the normalized \textit{unitary} Hecke character of $E^\times$. It's easy to see that for $n > 0$:
\begin{enumerate}[label=(\alph*)]
\item
$\chi_{\rm can}^n$ has $\infty$-type $(n,0)$.

\item
$\chi_{\rm can}^n$ has conductor $\sqrt{-7} \cO_E$ if $n$ is odd and conductor $\cO_E$ if $n$ is even.

\item
$\chi_{\rm can}|_{\bA_F^\times} = \varepsilon_{E/F}$.
\end{enumerate}

\subsection{Two quaternion algebras}

We now compute the local epsilon factors $\epsilon_v(\BC(\pi_{\chi_{\rm can}^n}) \otimes \chi_{\rm can}^m)$. At $v = \infty$, this calculation depends on whether $n+1 > m$ or $n + 1\leq m$. At the local places, this can be calculated using \cite[Section 1]{T83}. The interesting finite place is $v = 7$.
\begin{enumerate}[label=(\alph*)]
\item
Momentarily let $v$ be a real place of a number field $F$, take $f$ to be any automorphic form of $\GL_2$ of weight $k$ at $v$ and let $\Omega$ be a Hecke character of $E$ such that $\Omega_v(z) = z^{l_1} \overline z^{l_2}$. Then
\begin{equation*}
\epsilon_v(f, \Omega) \cdot \omega_v(-1) = \begin{cases}
+1 & \text{if $k \leq l_1 - l_2$,} \\
-1 & \text{if $k > l_1 - l_2$.}
\end{cases}
\end{equation*}
Since $\pi_{\chi_{\rm can}^n}$ has weight $n+1$, this implies that
\begin{equation*}
\epsilon_\infty(\BC(\pi_{\chi_{\rm can}^n}) \otimes \chi_{\rm can}^m) \cdot \omega_\infty(-1) = \begin{cases}
+1 & \text{if $n +1 \leq m$,} \\
-1 & \text{if $n + 1 > m$.}
\end{cases}
\end{equation*}

\item
Since $\chi_{{\rm can},v}$ factors through $\Nm$ for all $v \nmid 7$, the representation $\Ind_{\cW_{E_v}}^{\cW_{F_v}}(\chi_{{\rm can},v})$ is decomposable. By \cite[Proposition 1.6]{T83}, for any Hecke character $\Omega$, we have 
\begin{equation*}
\epsilon_v(\BC(\pi_{\chi_{\rm can}}) \otimes \Omega) \cdot \omega_v(-1) = +1 \qquad \text{for all $v \nmid 7$}.
\end{equation*}

\item
First observe that $\Res_{\cW_E} \Ind_{\cW_E}^{\cW_F}(\chi) = \chi \oplus \chi^\tau$ for any character $\chi$ of $\cW_E$. Since base change on the $\GL_2$ side corresponds to restriction on the Galois side, we have
\begin{equation*}
\epsilon_7(\BC(\pi_{\chi_{\rm can}}) \otimes \Omega) = \epsilon_7(\Res_{\cW_E} \Ind_{\cW_E}^{\cW_F} (\chi) \otimes \Omega) = \epsilon_7(\chi_{\rm can} \Omega) \epsilon_7(\chi_{\rm can}^\tau \Omega),
\end{equation*}
where the last equality holds because local $\epsilon$-factors change direct sums to products. By \cite[Lemma 3.2]{Yang thesis}, we have
\begin{equation*}
\epsilon_7(\chi_{\rm can} \Omega) = - \left(\tfrac{2}{7}\right) \sqrt{-1} = \epsilon_7(\chi_{\rm can}^\tau \Omega).
\end{equation*}
Since $\chi_{\rm can}|_{F^\times} = \varepsilon_{E/F}$, the automorphic representation $\pi_{\rm can}$ has trivial central character and hence the above calculation shows $\epsilon_7(\BC(\pi_{\chi_{\rm can}}) \otimes \Omega) \omega_7(-1) = -1$. By the above argument, 
\begin{equation*}
\epsilon_7(\BC(\pi_{\chi_{\rm can}^n}) \otimes \chi_{\rm can}^m) \cdot \omega_7(-1) = \begin{cases}
+1 & \text{if $n$ is even,} \\
-1 & \text{if $n$ is odd.}
\end{cases}
\end{equation*}
\end{enumerate}


We can now discuss the possibilities for the quaternion algebra determined by the pair of Hecke characters $\chi_{\rm can}^n$ and $\chi_{\rm can}^m$. First observe that the central character condition $\chi_{\rm can}^n \chi_{\rm can}^m \varepsilon_{E/F} = 1$ on $\bA^\times$ implies that $n$ and $m$ must have different parity. We now have two cases:
\begin{enumerate}[label=(\roman*)]
\item
If $n$ is odd, then $\epsilon_v(\BC(\pi_{\chi_{\rm can}^n}) \otimes \chi_{\rm can}^m) = -1$ if and only if $v = 7$. This implies that if $L(\BC(\pi_{\chi_{\rm can}^n}) \otimes \chi_{\rm can}^m, \tfrac{1}{2}) \neq 0$, then necessarily $n + 1 > m$ so that $\epsilon_\infty(\BC(\pi_{\chi_{\rm can}^n}) \otimes \chi_{\rm can}^m) = -1$ and hence $\Sigma_{\pi_{\chi_{\rm can}^n},\chi_{\rm can}^m} = \{7, \infty\}.$

\item
If $n$ is even, then $\epsilon_v(\BC(\pi_{\chi_{\rm can}^n}) \otimes \chi_{\rm can}^m) = +1$ for all finite $v$. This implies that if $L(\BC(\pi_{\chi_{\rm can}^n}) \otimes \chi_{\rm can}^m, \tfrac{1}{2}) \neq 0$, then necessarily $n + 1 \leq m$ so that $\epsilon_\infty(\BC(\pi_{\chi_{\rm can}^n}) \otimes \chi_{\rm can}^m) = +1$ and hence $\Sigma_{\pi_{\chi_{\rm can}^n},\chi_{\rm can}^m} = \varnothing.$
\end{enumerate}
Summarizing, take $n,m$ to have opposite parity, we have the chart
\begin{equation}\label{e:dichotomy chart}
\begin{tabular}{c | c | c}
& $\begin{gathered} n+1 > m \\ \epsilon_\infty = -1\end{gathered}$ & $\begin{gathered} n+1 \leq m \\ \epsilon_\infty = +1 \end{gathered}$ \\ \hline
$\epsilon = +1$ & $\begin{gathered} \text{$n$ odd, $m$ even} \\ \epsilon_7 = -1 \\ \text{(definite, ramified at $7,\infty$)}  \end{gathered}$ & $\begin{gathered} \text{$n$ even, $m$ odd} \\ \epsilon_7 = +1  \\ \text{(indefinite---in fact, split!)}  \end{gathered}$ \\ \hline
$\epsilon = -1$ & $\begin{gathered} \text{$n$ even, $m$ odd} \\ \epsilon_7 = +1 \end{gathered}$ & $\begin{gathered} \text{$n$ odd, $m$ even} \\ \epsilon_7 = -1 \end{gathered}$
\end{tabular}
\end{equation}
Waldspurger's formula is in the setting of $\epsilon = +1$, and our main theorem (Theorem \ref{t:identity of periods}) gives an identity between the two $\epsilon = +1$ boxes, taking $B = M_2(F)$ and $B' = B_{\{7,\infty\}}$. In Sections \ref{ch:schwartz} and \ref{ch:explicit rallis}, we constructed a family of Schwartz functions such that their theta lifts realize the newform and its images under iterates of the Shimura--Maass operator. We recall this construction next.

\subsection{Torus periods of a weight-$(3+2l)$ CM form}

Take the special case $n=2$ and let $m = 3 + 2l$, where $l \geq 0$. As in Section \ref{s:schwartz}, we take $\phi_l' \colonequals \otimes_v \phi_{l,v}'$ where
\begin{equation*}
\phi_{l,v}'(z) = \begin{cases}
{}_1 F_1(-l,3,4\pi z \overline z) z^2 e^{-2 \pi z \overline z} & \text{if $v \mid \infty$,} \\
\ONE_{\cO_{F_v}}(z_1) \cdot \ONE_{\cO_{F_v}}(z_2) & \text{if $v \nmid \infty$.}
\end{cases}
\end{equation*}
If we set $\xi = \chi_{\rm can}$, 
\begin{equation*}
C_v = \begin{cases}
\frac{(2 \pi)^2}{4^3 \pi^4} \cdot \frac{l! \cdot 4}{(l+2)!} = \frac{1}{2 (l+2)(l+1) \pi^2} & \text{if $v \mid \infty$,} \\
1 & \text{if $v \neq 7$,} \\
\frac{1}{7}(1 - 49^{-1})^{-1}(1- 7^{-1}) = \frac{1}{8} & \text{if $v = 7$},
\end{cases}
\end{equation*}
so that by Theorem \ref{t:global lift}(b) and Theorem \ref{t:newform}, the theta lift $\theta_{\phi_l'}(\chi_{\rm can}^2 \xi)$ is a Hecke eigenform on $\GL_2(\bA_\bQ)$ in $\pi_{\chi_{\rm can}^2}$. Furthermore, again by Theorem \ref{t:explicit rallis},
\begin{equation*}
\langle \theta_{\phi_l'}(\chi_{\rm can}^2 \xi), \theta_{\phi_l'}(\chi_{\rm can}^2 \xi) \rangle = \left(\frac{2\pi}{\sqrt{7} \cdot 2}\right)^{-1} \cdot \frac{1}{16 \cdot (l+2) \cdot (l+1) \cdot \pi^2} \cdot \frac{L(1, \widetilde \chi_{\rm can}^2)}{\zeta(2)}.
\end{equation*}
And as before, by Theorem \ref{t:B B' adjoint}, 
\begin{equation*}
\int_{[E^\times]} \theta_{\phi_l'}(\chi_{\rm can}^2 \cdot \xi)(g) \cdot \chi_{\rm can}^{3 + 2l}(g) \, dg = \int_{[E^\times]} \chi_{\rm can}^2(g) \cdot \overline{\theta_{\phi_l'}'(\overline{\chi_{\rm can}^{3 + 2l} \cdot \xi'{}^{-1}})}(g) \, dg,
\end{equation*}
where by Theorem \ref{t:global lift}(b) the theta lift $\overline{\theta_{\phi_l'}'(\overline{\chi_{\rm can}^{3+2l} \cdot \xi'{}^{-1}})}$ is an automorphic form in $\pi_{\chi_{\rm can}^{3 + 2l}}^{B'}$.


\subsection{Nonvanishing torus periods} 

Recall that by Theorem \ref{t:newform}, $\theta_{\phi_l'}(\chi_{\rm can}^2 \xi)$ is a nonzero Hecke eigenform of weight $3 + 2l$ in $\pi_{\chi_{\rm can}^2}$. Choose a basis $1, \bi, \bj, \bi\bj$ for the split quaternion algebra $M_2(\bQ)$ such that $\bi^2 = u = -7$, $\bj^2 = -1/7$. Then our torus embedding is
\begin{equation*}
E^\times \hookrightarrow \GL_2(\bQ), \qquad a + b \bi \mapsto \left(\begin{smallmatrix} a & -2 b \\ -bu/2 & a \end{smallmatrix}\right).
\end{equation*}
In particular, for any finite place $p$ of $\bQ$, the induced embedding $E_p^\times \hookrightarrow \GL_2(\bQ_p)$ makes the compact open subgroup 
\begin{equation*}
K_{0,p} \colonequals 
\begin{cases}
\GL_2(\bZ_p) & \text{if $p \neq 7$,} \\
\left\{\left(\begin{smallmatrix} a & b \\ c & d \end{smallmatrix}\right) \in \GL_2(\bZ_p) : c \in 7\bZ_7\right\} & \text{if $p = 7$,}
\end{cases}
\end{equation*}	
into an \textit{optimal} compact open subgroup (in the sense of Gross \cite[Proposition 3.2]{G88}) with respect to $\chi_{\rm can}^{3 + 2l}$, which is unramified at $v \nmid 7$ and has conductor $1$ at $v \mid 7$. By \cite{GP91}, a Hecke eigenform with respect to the above compact open subgroup of $\GL_2(\bA_\bQ)$ is locally (up to a scalar) the Gross--Prasad test vector. Therefore, 
\begin{equation*}
\int_{[E^\times]} \theta_{\phi_l'}(\chi_{\rm can}^2 \xi)(g) \cdot \chi_{\rm can}^{3 + 2l}(g) \, dg \neq 0
\end{equation*}
and combining this with Theorems \ref{t:identity of periods},  \ref{t:newform}, and \ref{t:Dl}, we obtain:

\begin{corollary}
Let $B' = B_{7,\infty}$ denote the definite quaternion algebra over $\bQ$ ramified at exactly $7$ and $\infty$. Define
\begin{equation*}
f_{\chi_{\rm can}^2}^{(l)} \colonequals \theta_{\phi_l'}(\chi_{\rm can}^2 \xi), \qquad f_{\chi_{\rm can}^{3+2l}}^{B'} \colonequals \overline{\theta_{\phi_l'}'(\overline{\chi_{\rm can}^{3+2l} \xi'})}.
\end{equation*}
Then:
\begin{enumerate}[label=(\alph*)]
\item
$\pi^{-l} \cdot f_{\chi_{\rm can}^2}^{(l)}$ is an algebraic Hecke eigenform of weight $3 + 2l$ in $\pi_{\chi_{\rm can}^2}$.

\item
$f_{\chi_{\rm can}^{3+2l}}^{B'}$ is an automorphic form in the Jacquet--Langlands transfer $\pi_{\chi_{\rm can}^{3+2l}}^{B'}$,

\item
there is an identity of nonzero torus periods
\begin{equation*}
0 \neq \int_{[E^\times]} f_{\chi_{\rm can}^2}^{(l)}(g) \cdot \chi_{\rm can}^2(g) \, dg = \int_{[E^\times]} \chi_{\rm can}^2(g) \cdot f_{\chi_{\rm can}^{3+2l}}^{B'}(g) \, dg.
\end{equation*}
\end{enumerate}
\end{corollary}

\bibliographystyle{plain}

\end{document}